%% file: GPPE.tex
\documentclass[final,hidelinks,onefignum,onetabnum]{siamart220329}

\input{ex_shared}

\ifpdf
\hypersetup{
	pdftitle={},
	pdfauthor={}
}
\fi

\usepackage{comment}

\begin{document}
	
	\maketitle
	\begin{abstract}
		We propose a fully decoupled, structure-preserving relaxation Crank--Nicolson finite element method (FEM) for the coupled Gross--Pitaevskii--Poisson (GPP) system modeling ultracold plasmas. By introducing suitable auxiliary variables to reformulate the nonlinear interaction and charge density terms, the original system is recast into an equivalent form that enables a linear, fully decoupled numerical scheme. The proposed method preserves key physical invariants, including the mass of each component and a modified discrete energy, at the fully discrete level. We establish the well-posedness and uniqueness of the scheme and rigorously derive optimal error estimates, achieving second-order accuracy in time and optimal $(k+1)$-th order convergence in space for $P^k$ finite element approximations. Numerical experiments confirm the theoretical results and demonstrate the effectiveness of the method in preserving conservation properties and accurately capturing complex dynamical behaviors of the coupled GPP system.
	\end{abstract}
	
	
	
	
	\section{Introduction}
	
	In this paper, we are interested in the development of efficient numerical methods for solving the coupled Gross--Pitaevskii--Poisson (GPP) model for ultracold plasmas \cite{CGPP_2020}, 
	\begin{subequations}\label{TargetEq}
		\begin{align}
			&\mathbf{i}\partial_t \psi_+ =\left[ -\frac{1}{2} \nabla ^2 + \left(g \vert {\psi_+} \vert ^2 + G \vert {\psi_-} \vert ^2 \right) + q\phi \right] \psi_+, && \text{in } \R^d \times (0,T], \label{TargetEq1} \\
			&\mathbf{i}\partial_t \psi_- =\left[ -\frac{1}{2} \nabla ^2 + \left(g \vert {\psi_-} \vert ^2 + G \vert {\psi_+} \vert ^2 \right) - q\phi \right] \psi_-,&& \text{in } \R^d \times (0,T], \label{TargetEq2} \\ 
			&\nabla^2 \phi = -4 \pi q \left(\vert {\psi_+} \vert ^2 - \vert {\psi_-} \vert ^2 \right),&& \text{in } \R^d \times (0,T], \label{TargetEq3}
		\end{align}
	\end{subequations} 
	where $\mathbf{i} = \sqrt{-1}$ denotes the imaginary unit, and $\mathbf{x}=(x_1,x_2,...x_d) \in \R^d$ with $d=1,2,3$. 
	The model describes the dynamics of two interacting quantum wave functions $\psi_\pm$, which are coupled through a self-consistent Poisson potential $\phi$ within the framework of the ultra-cold plasma model \cite{CGPP_2020}. 
	The real-valued parameters $g$, $G$, and $q$ denote the self-interaction coefficient, interspecies interaction strength, and scaled charge parameter, respectively. Here, $|\psi_\pm|^2$ denote the charge densities of the two components, $|\psi_\pm|^2\psi_\pm$ are the self-interaction terms, $|\psi_\mp|^2\psi_\pm$ are the cross-interaction terms, and $\phi\,\psi_\pm$ describe the Coulomb coupling induced by the self-consistent electrostatic potential $\phi$.

	The nonlinear Schrödinger equation (NLSE), 
	\begin{equation}
		\mathbf{i}\partial_t \psi =\left[ -\nabla ^2 + f( \vert {\psi} \vert ^2 ) \right] \psi \label{NLS equation},
	\end{equation}
	serves as a fundamental model for a wide range of physical phenomena, including nonlinear optics \cite{GPoptschen2006foundations,GPoptssulem1999nonlinear}, Bose-Einstein condensates (BECs) \cite{GPBose-Einsteincondensatesbao2012mathematical,GPBose-Einsteincondensateserdos2010derivation,GPBose-Einsteincondensateslieb2001rigorous}, deep-water modulation \cite{GPWaterperegrine1983water,GPWateryuen1980instabilities}, and other applications. When the self-interaction takes the form of a cubic nonlinearity, 
	$f(\vert {\psi} \vert ^2)=g\vert {\psi} \vert ^2$, the NLSE \eqref{NLS equation} is commonly referred to as the Gross-Pitaevskii equation (GPE) \cite{GPBose-Einsteincondensatesbao2012mathematical}. 
	
	
	The GPE \eqref{NLS equation} can be coupled with a Poisson equation to model interactions between the wave function $\psi$ and an electrostatic or gravitational potential $\phi$. The resulting  Gross--Pitaevskii--Poisson (GPP) equations \cite{liu2026structure, GPPEruffini_systems_1969} are given by
	\begin{subequations}\label{GPP1 equations}
		\begin{align}
			\mathbf{i}\partial_t \psi &= \left[ -\nabla^2 + {g|\psi|^2} + q\phi \right]\psi, \\
			\nabla^2 \phi &= -\mu \left(|\psi|^2 - c \right).
		\end{align}
	\end{subequations}
	Here, $\mu>0$ ($<0$) corresponds to repulsive (attractive) interactions, and $c$ denotes the background charge density. This system captures the self-consistent coupling between the wave function and the potential, enabling the study of long-range interactions.
	
	Coupling with the Poisson equation introduces significant numerical challenges, primarily due to the increased nonlinearity of the system and the need to preserve appropriate invariant properties at the same time \cite{liu2026structure}. Take the GPP system \eqref{GPP1 equations} as an example, the model contains two distinct nonlinear components: the self-repulsion term $|\psi|^2\psi$ in the Gross--Pitaevskii equation and the charge density term $|\psi|^2$ in the Poisson equation. Traditional approaches typically employ different numerical techniques to handle these two nonlinearities. For instance, in \cite{gong2022sav}, a Crank–Nicolson scalar auxiliary variable (SAV) scheme was proposed, where the self-repulsion term $|\psi|^2\psi$ is treated using the SAV approach, while the charge density is handled via iterative methods. Similarly, in \cite{DGyi2022mass}, the self-repulsion term is addressed using a relaxation technique, whereas the charge density is again solved iteratively.
	In contrast, \cite{liu2026structure} proposed a unified relaxation-based approach that introduces a single auxiliary variable to simultaneously reformulate the different nonlinear terms in both equations. 
	
	For systems with multiple interacting species, the GPE \eqref{NLS equation} can be generalized into a coupled system to account for interspecies interactions and self-consistent potentials. A two-species coupled GPE takes the form:
	\begin{subequations}\label{coulped NLS equation}
		\begin{align}
			&\mathbf{i}\partial_t \psi_+ =\left[ -\frac{1}{2} \nabla ^2 + \left(g \vert {\psi_+} \vert ^2 + G \vert {\psi_-} \vert ^2 \right)  \right] \psi_+ , \\
			&\mathbf{i}\partial_t \psi_- =\left[ -\frac{1}{2} \nabla ^2 + \left(g \vert {\psi_-} \vert ^2 + G \vert {\psi_+} \vert ^2 \right)  \right] \psi_-,
		\end{align}
	\end{subequations}
	where $\psi_{+}$ and $\psi_{-}$ denote the wave functions of cations and anions, respectively, and $g, G$ are dimensionless constants describing interaction strength \cite{GPandGPPantoine_computational_2013}. 
	This formulation is widely used in the study of multi-component BECs \cite{GPBose-Einsteincondensatesbao2012mathematical,CGPpitaevskii2003bose} and laser beam interactions \cite{CGPbao2007time, CGPnewell1985solitons}.

	In the coupled GPE system \eqref{coulped NLS equation}, the four nonlinear interaction terms $\vert \psi_\pm \vert^2 \psi_\pm$, $\vert \psi_\pm \vert^2 \psi_\mp$ are commonly treated either explicitly or implicitly via iterative solvers. 
	A Crank–Nicolson-type method in \cite{Adhikari2001} treats the nonlinear terms explicitly, yielding a linear scheme, but requires small mesh sizes and time steps to avoid oscillations and may reduce temporal accuracy to first order.
	In \cite{BaoCai2011}, a Crank–Nicolson approach with averaged nonlinear terms leads to a nonlinear and costly implicit scheme, while a stabilized backward Euler method treats nonlinearities explicitly for improved efficiency.
	Second-order time-splitting sine spectral methods \cite{Bao2004,Wang2007} also handle the nonlinear terms explicitly, resulting in schemes that require solving a sequence of linear subproblems (e.g., seven linear systems for the coupled unknowns $\psi_+$ and $\psi_-$ per time step), while remaining explicit, unconditionally stable, time-reversible, time-translationally invariant, and spectrally accurate in space with second-order accuracy in time. 
	Other splitting approaches via coordinate transformation are discussed in \cite{ming2014efficient}. High-order IMEX spectral schemes \cite{AntoineBesseRispoli2016} treat linear terms implicitly and nonlinear terms explicitly.
	
	For the GPP model \eqref{TargetEq}, the dynamics are governed by the interplay among the self-interaction \(g\), interspecies repulsion \(G\), and charge parameter \(q\). In the repulsive case (\(g>0\)), a uniformly mixed neutral state is stable at low densities, while modulational instability and density-wave patterns arise only when \(G>g\) and the density exceeds a threshold. In the attractive case (\(g<0\)), overlapped neutral bright solitons exist for \(G<|g|\), whereas stronger interspecies repulsion induces splitting into dipole states, with quadrupole solitons emerging in certain regimes \cite{mineev1974theory, CGPP_2020}.
	The parameter \(q\) denotes the charge magnitude: larger \(q\) strengthens long-range electrostatic interactions, while \(q=0\) removes the Poisson coupling and reduces \eqref{TargetEq} to the coupled GPE system \eqref{coulped NLS equation} \cite{CGPP_2020}. 

	The development of efficient numerical methods for the GPP model \eqref{TargetEq} is particularly challenging because of the coexistence of multiple interacting species and the coupling with the Poisson equation, which together introduce several nonlinear interaction terms and charge-density terms. To the best of our knowledge, the only existing numerical study based on the imaginary time integration method \cite{GPandGPPantoine_computational_2013} was reported in \cite{CGPP_2020} to investigate the dynamics of the GPP model \eqref{TargetEq}. However, the implementation details of this approach were not provided.
	
	Motivated by the relaxation techniques in \cite{GPandGPPantoine_computational_2013, liu2026structure}, as well as the efficient treatment of Poisson coupling proposed in \cite{liu2026structure}, we develop a relaxation Crank–Nicolson finite element discretization for the coupled GPP model \eqref{TargetEq}. By introducing two auxiliary variables $Z_\pm = |\psi_\pm|^2$, the four nonlinear interaction terms and the two nonlinear charge-density terms in the original system are reformulated into equivalent forms.
	The resulting scheme is linear and fully decoupled, requiring only two independent linear systems for $\psi_\pm$ and one Poisson equation for $\phi$ at each time step. This significantly improves computational efficiency compared with up to second-order splitting-type methods, which require solving seven independent linear systems. Moreover, the scheme is well-posed and preserves the discrete mass of each species as well as a modified total energy, while remaining straightforward to implement. Moreover, the scheme is well-posed and preserves the discrete mass of each species, as well as a modified total energy, while remaining straightforward to implement.
	
	Error analyses for relaxation-type methods applied to the nonlinear Schrödinger equation and the GPP system \eqref{GPP1 equations} have been established in \cite{ShBhestimation, liu2026structure}. Building on some of the analytical tools in \cite{ShBhestimation, liu2026structure}, we derive rigorous error estimates for the proposed method. The analysis combines projection-based error splitting, induction arguments, uniform boundedness of the finite element approximations, and a careful handling of the interactions among multiple species as well as the coupling with the Poisson equation. This yields optimal second-order convergence in time and optimal $(k+1)$-th order convergence in space, where $k$ denotes the polynomial degree of the finite element space. The proposed relaxation Crank–Nicolson finite element method and its accompanying analysis extend naturally to the coupled GPE system \eqref{coulped NLS equation}.
	
	The contributions, innovations, and significance of this work are summarized as follows:
	\begin{itemize}
		\item For the coupled Gross--Pitaevskii--Poisson system, which involves multiple interacting species together with Poisson coupling, we develop a relaxation Crank--Nicolson finite element method by introducing only two auxiliary variables to reformulate the four nonlinear interaction terms and the two charge-density terms.
		
		\item The proposed method is fully decoupled and linear at each time step. It requires neither iterative procedures for nonlinear terms nor the solution of a coupled algebraic system, and thus is computationally efficient. Despite its linear and decoupled structure, the method preserves the discrete mass of each component and a modified discrete energy, thereby inheriting important structure-preserving properties of the continuous model.
		
		\item We establish the well-posedness and uniqueness of the fully discrete scheme and derive rigorous optimal error estimates. In particular, the method achieves second-order accuracy in time and optimal $(k+1)$-th order convergence in space in the $L^2$-norm, where $k$ is the polynomial degree of the finite element space.
		
		\item  While retaining second-order accuracy in time, the proposed method significantly reduces the computational cost. In contrast to existing splitting-type methods for the coupled GPE, which require up to seven groups of linear systems, the proposed approach involves only two groups of decoupled linear systems; for the coupled GPP model, only one additional Poisson equation needs to be solved.
		
		\item Numerical experiments are presented to verify the theoretical error estimates and the conservation properties, and to demonstrate the effectiveness of the proposed method in accurately capturing the dynamics of the coupled Gross--Pitaevskii--Poisson system.
	\end{itemize}
	
	The organization of this paper is as follows. In \Cref{RCN FEM}, we introduce an equivalent reformulation of the system and present a relaxation Crank–Nicolson finite element method for the coupled GPP model. We further establish the conservation properties. In \Cref{secErrorEstimate}, we derive optimal $L^2$-norm error estimates for the fully discrete solutions, achieving second-order accuracy in time and $(k+1)$-th order accuracy in space. The analysis is carried out via mathematical induction, with detailed proofs deferred to \Cref{Appendix}. \Cref{secNmecical} presents several numerical examples to verify the masses and energy conservation properties of the proposed method and to confirm the theoretical error estimates.

	\section{The Relaxation Crank--Nicolson Finite Element Method}\label{RCN FEM}
	
	The GPP model \eqref{TargetEq} is originally posed on $\mathbb{R}^d$. For numerical purposes, we restrict it to a bounded domain $\Omega \subset \mathbb{R}^d$ \cite{BDbian_almost_2016,BDszeftel_design_2004, CGPP_2020, liu2026structure} and impose periodic boundary conditions on $\psi_\pm$ and $\phi$ \cite{BDillner_quasi-linear_1997,sukumar2009classicalBoundaryC}. The proposed methods and analysis can be extended to other boundary conditions.
	The initial condition is prescribed as
	\begin{equation}\label{TargetEq5}
		\psi_\pm(\mathbf{x},0) = \psi_{\pm,0}(\mathbf{x}), \quad \text{in } \Omega,
	\end{equation}
	for given functions $\psi_{\pm,0}(\mathbf{x})$.
	The GPP model \eqref{TargetEq} conserves the following quantities \cite{CGPP_2020}:  \\
	\noindent (i) Masses conservation
	\begin{equation}
		M(\psi_{\pm}(t)) = M(\psi_{\pm}(0)), \quad \text{where the masses } M(\psi_{\pm}) = \int_{\Omega} \vert \psi_\pm(\mathbf{x},t) \vert^2 \, d\mathbf{x}; \label{PrimeMass}
	\end{equation}
	\noindent (ii) Energy conservation
	\begin{equation}
		E(\psi_{+}(t),\psi_{-}(t),\phi(t)) = E(\psi_{+}(0),\psi_{-}(0),\phi(0)), \label{PrimeEnergy}
	\end{equation}
	where the discrete energy
	\begin{equation*}
		E(\psi_{+},\psi_{-},\phi) = \int_{\Omega} \frac{1}{2} \Big( \vert \nabla \psi_+ \vert^2 + \vert \nabla \psi_- \vert^2 \Big)
		+ \frac{g}{2} \left( \vert \psi_+ \vert^4 + \vert \psi_- \vert^4 \right) 
		+ G \vert \psi_+ \vert^2 \vert \psi_- \vert^2 
		+ \frac{1}{8 \pi} \vert \nabla \phi \vert^2 \, d\mathbf{x}.
	\end{equation*}
	For the well-posedness of \eqref{TargetEq}, the compatibility condition $M(\psi_{+}(t)) = M(\psi_{-}(t))$
	must hold for all $t \ge 0$. 
	For well-posedness, the compatibility condition $M(\psi_{+}(t)) = M(\psi_{-}(t))$ must hold for all $t \ge 0$. By the masses conservation \eqref{PrimeMass}, it is sufficient to require that the initial data satisfy
	\begin{equation}\label{masscompat}
		M(\psi_{+}(0)) = M(\psi_{-}(0)).
	\end{equation}
	To ensure uniqueness of $\phi$, we impose a zero-average condition over $\Omega$.
	
	In this paper, we define the inner product and norm of the complex-valued Hilbert space $\mathbf{L}^2(\Omega)$ and the real-valued Hilbert space $L^2(\Omega)$, respectively, as
	\begin{align}
		& \langle u, v \rangle = \int_{\Omega}u v^* dx \quad \text{and} \quad  \|u\|=\sqrt{\langle u,u \rangle}, \label{inner}\\
		& \left(u, v \right) = \int_{\Omega}u v dx \quad \text{and} \quad  \|u\|=\sqrt{\left( u,u \right)}, \label{inner+}
	\end{align}
	where $v^* $ denotes the complex conjugate of $v$. 
	For any integer $s\ge 0$, we define 
	$$\mathbf{H}^s(\Omega)=\{ u \in \mathbf{L}^2(\Omega):D^{\alpha}u \in \mathbf{L}^2(\Omega), \forall|\alpha| \le s\}$$ as the conventional complex-valued Sobolev space, where $D^{\alpha} u$ represents the weak derivative of $u$ associated with the multi-index $\alpha$.
	This space is equipped with the semi-norm $| u|_s = \sqrt{\sum_{|\alpha| = s} \int_{\Omega} |D^{\alpha} u |^2 dx}$ and the norm $\| u\|_s = \sqrt{\sum_{|\alpha| \le s} \int_{\Omega} |D^{\alpha} u |^2 dx}$. 
	We define the space $\mathbf{H}^1_{\text{per}}(\Omega)$ as the subspace of $\mathbf{H}^1(\Omega)$ consisting of functions that are periodic in each coordinate direction. 
	For real-valued Sobolev spaces, we use $H^s(\Omega)$ and $H^1_{\text{per}}(\Omega)$ with analogous definitions and norms.
	To simplify the notation, we denote 
	\begin{equation*}
		\mathbf{V} := \mathbf{H}^1_{\text{per}}(\Omega), \quad V :=  H^1_{\text{per}}(\Omega) , \quad \mathring{V} := \left\{ v \in V \;\middle|\; \int_{\Omega} v(x) \, dx = 0 \right\}.
	\end{equation*}
	
	\subsection{Variational formulation}
	Consider a general Poisson equation
	\begin{equation}\label{PoissonPeriodic}
		-\Delta a = f \quad \text{in } \Omega,
	\end{equation}
	subject to a periodic boundary condition. The problem \eqref{PoissonPeriodic} is well-posed if and only if the source term $f$ satisfies the compatibility condition 
	\begin{equation}\label{compatibility}
		\int_\Omega f dx = 0.
	\end{equation}
	The variational formulation for \eqref{PoissonPeriodic} seeks $a \in \mathring{V}$ such that
	\begin{equation}\label{PoissonWeak}
		(\nabla a, \nabla w) = (f, w), \quad \forall w \in \mathring{V}.
	\end{equation}
	For any function $w\in V$, the shifted function $\tilde w =w - \int_\Omega w dx \in \mathring{V} $. Hence, under the compatibility condition \eqref{compatibility}, the variational problem is equivalent to finding $a \in \mathring{V}$ such that
	\begin{equation}\label{PoissonWeak+}
		(\nabla a, \nabla w) = (f, w), \quad \forall w \in V.
	\end{equation}
	
	
	
	Note that under initial masses condition \eqref{masscompat} and masses conservation \eqref{PrimeMass}, the source term of the Poisson equation \eqref{TargetEq3} satisfies the compatibility condition \eqref{compatibility}. Therefore, the Poisson equation \eqref{TargetEq3} with the periodic boundary condition is well-posed. Consequently, the variational formulation for problem \eqref{TargetEq} seeks
	$\psi_\pm \in C^1([0,T];\mathbf{V})$, and $\phi \in C([0,T];{\mathring{V}})$ such that
	\begin{subequations} \label{WeakF} 
		\begin{align}
			&\left\langle \mathbf{i} \partial_t \psi_+, v_+ \right\rangle = \frac{1}{2} A_0\left( \psi_+, v_+ \right) + \left \langle g \vert \psi_+ \vert^2 \psi_+ +  G \vert \psi_- \vert^2 \psi_+ +  q \phi \psi_+, v_+ \right\rangle, && \forall v_+ \in \mathbf{V}, \label{WeakF1}\\ 
			&\left\langle \mathbf{i} \partial_t \psi_-, v_- \right\rangle = \frac{1}{2} A_0 \left( \psi_-, v_- \right) + \left\langle g \vert \psi_- \vert^2 \psi_- +  G \vert \psi_+ \vert^2 \psi_-  -  q \phi \psi_-, v_- \right\rangle,&& \forall v_- \in \mathbf{V}, \label{WeakF2}\\
			&A_1 \left( \phi , w \right) = \left( 4 \pi q \left( \vert \psi_+ \vert^2 - \vert \psi_- \vert^2 \right), w \right), && \forall w \in V. \label{WeakF3}    
		\end{align} 
	\end{subequations}
	Here, the bilinear forms $A_0(\cdot,\cdot)$ and $A_1(\cdot ,\cdot )$ are defined as
	\begin{align}
		&A_0(w,v) = \langle \nabla w, \nabla v\rangle , \quad \forall w,v \in \mathbf{V}, \\
		&A_1(w,v) = (\nabla w, \nabla v), \quad \forall w,v \in \mathring{V}.
	\end{align}
	It is straightforward to verify for any $w, v \in \mathbf{V}$,
	\begin{equation}
		A_0(w,w) = |w|_1^2, \quad    A_0(w,v) \le | w |_1 | v |_1, \label{coercivity and boundedness-}
	\end{equation}
	and there exists some constants $\gamma_1, \gamma_2>0$ such that for any $w, v \in \mathring{V}$,
	\begin{equation}
		A_1(w,w) \geq \gamma_1 \|w\|_1^2, \quad  A_1(w,v) \le \gamma_2 \Vert w \Vert_1 \Vert v \Vert_1. \label{coercivity and boundedness}
	\end{equation}
	\subsection{Reformulated GPP model}\label{Equivalent System}
	
	By introducing these auxiliary variables, system \eqref{TargetEq} can be equivalently rewritten as
	\begin{subequations}\label{Equiv TargetEq}
		\begin{align}
			&\mathbf{i}\partial_t \psi_+ = -\frac{1}{2} \nabla ^2 \psi_+ + g Z_+ \psi_+ + G Z_- \psi_+  + q\phi \psi_+,  \label{Equiv TargetEq1} \\
			&\mathbf{i}\partial_t \psi_- = -\frac{1}{2} \nabla ^2 \psi_- + g Z_- \psi_- + G Z_+ \psi_- - q\phi \psi_-, \label{Equiv TargetEq2} \\ 
			&\nabla^2 \phi = -4 \pi q \left(Z_+ - Z_- \right), \\
			& Z_+ = |\psi_+|^2, \\
			& Z_- = |\psi_-|^2.
		\end{align}
	\end{subequations} 
	The variational formulation for the reformulated GPP model \eqref{Equiv TargetEq} is to find $\psi_\pm \in C^1([0,T]; \mathbf{V})$, $\phi\in C([0,T];\mathring{V})$, and $Z_\pm \in  C([0,T];V)$ such that
	\begin{subequations} \label{Equiv WeakF} 
		\begin{align}
			&\langle \mathbf{i} \partial_t \psi_+, v_+ \rangle = \frac{1}{2} A_0\left( \psi_+, v_+ \right) + \langle g Z_+ \psi_+, v_+ \rangle + \langle G Z_- \psi_+, v_+ \rangle + \langle q \phi \psi_+, v_+ \rangle, && \forall v_+ \in \mathbf{V}, \label{Equiv WeakF1}\\ 
			&\langle \mathbf{i} \partial_t \psi_-, v_- \rangle = \frac{1}{2} A_0 \left( \psi_-, v_- \right) + \langle g Z_- \psi_-, v_- \rangle + \langle G Z_+ \psi_-, v_- \rangle - \langle q \phi \psi_-, v_- \rangle,&& \forall v_- \in \mathbf{V}, \label{Equiv WeakF2}\\
			&A_1 \left( \phi , w \right) = \left( 4 \pi q \left( Z_+ - Z_- \right), w \right), && \forall w \in  V, \label{Equiv WeakF3}   \\
			&(Z_+, \chi_+) = (|\psi_+|^2, \chi_+), && \forall \chi_+ \in V, \label{Equiv WeakF4} \\
			&(Z_-, \chi_-) = (|\psi_-|^2, \chi_-), && \forall \chi_- \in V. \label{Equiv WeakF5} 
		\end{align} 
	\end{subequations}
	For this reformulated system, the following invariants are conserved.
	\begin{lemma}
		For any $t>0$, the reformulated GPP model \eqref{Equiv TargetEq} satisfy the masses conservation \eqref{PrimeMass} and the energy conservation:
		\begin{equation}
			E(\psi_{+}(t),\psi_{-}(t),\phi(t),Z_+(t),Z_-(t)) =  E(\psi_{+}(0),\psi_{-}(0),\phi(0),Z_+(0),Z_-(0)),
			\label{ModEnergy}
		\end{equation}
		where the energy
		\begin{equation}\label{energy}
			E(\psi_{+},\psi_{-},\phi,Z_+,Z_-) = \int _{\Omega} \frac{1}{2}  \left(\vert { \nabla \psi_+} \vert ^2 + \vert {\nabla \psi_-} \vert^2 \right) + \frac{g}{2} \left( Z_+ ^2 + Z_- ^2 \right) + G Z_+ Z_- + \frac{1}{8 \pi} \vert {\nabla \phi} \vert ^2 dx.
		\end{equation}
	\end{lemma}
	\begin{proof}
		Taking $v_+=\psi_+$ and $v_-=\psi_-$ in \eqref{Equiv WeakF1} and \eqref{Equiv WeakF2}, respectively, and comparing their imaginary parts of both equations yield
		\begin{equation} \label{RCGPPEMMAMASS}
			\frac{d}{dt} \int_\Omega |\psi_+|^2 dx = \frac{d}{dt} \int_\Omega |\psi_- |^2 dx =0,
		\end{equation}
		which directly implies the mass conservation property \eqref{PrimeMass}.
		
		Next, taking $v_+=2\partial_t \psi_+,\;v_-=2\partial_t \psi_-$ in \eqref{Equiv WeakF1} and \eqref{Equiv WeakF2}, respectively, the real parts yield
		\begin{subequations}
			\begin{align}
				\text{Re}\left(A_0\left(  \psi_+, \partial_t \psi_+ \right) + 2\langle g Z_+ \psi_+, \partial_t \psi_+ \rangle + 2\langle G Z_- \psi_+, \partial_t \psi_+ \rangle + 2\langle q \phi \psi_+, \partial_t \psi_+ \rangle \right) = 0, \label{PDEenergyF1}\\
				\text{Re}\left(A_0\left( \psi_-, \partial_t \psi_- \right) + 2\langle g Z_- \psi_-, \partial_t \psi_- \rangle + 2\langle G Z_+ \psi_-, \partial_t \psi_- \rangle - 2\langle q \phi \psi_-, \partial_t \psi_- \rangle \right)= 0.    \label{PDEenergyF2}
			\end{align}       
		\end{subequations}
		Taking the time derivative with respect to the first argument in \eqref{Equiv WeakF4} and \eqref{Equiv WeakF5} and setting $\chi_+ = gZ_+ + G Z_-$ and $\chi_- = gZ_- + G Z_+$ give
		\begin{subequations}
			\begin{align}
				&(\partial_t Z_+, gZ_+ + G Z_-) = \text{Re}\left(2\langle g \psi_+, Z_+ \partial_t \psi_+\rangle + 2\langle G\psi_+, Z_- \partial_t \psi_+ \rangle\right), \label{Equiv WeakF4 t} \\
				&(\partial_t Z_-, gZ_- + G Z_+) = \text{Re}\left(2\langle g \psi_-, Z_- \partial_t \psi_-\rangle + 2\langle G\psi_-, Z_+ \partial_t \psi_- \rangle\right).  \label{Equiv WeakF5 t} 
			\end{align}       
		\end{subequations}
		In addition, taking time derivative for the first argument in \eqref{Equiv WeakF3} and setting $w=\frac{\phi}{4\pi}$ give
		\begin{equation}
			\frac{1}{8 \pi} \frac{d}{dt} \int_{\Omega} \vert \nabla \phi \vert^2 dx = \int_{\Omega} q \phi \partial_t \left( \vert {\psi_+} \vert ^2 - \vert {\psi_-} \vert ^2 \right) dx =  \text{Re}\left(2\langle q \phi \psi_+, \partial_t \psi_+ \rangle-2\langle q \phi \psi_-, \partial_t \psi_- \rangle\right). \label{phi Z}
		\end{equation}
		By summing \eqref{PDEenergyF1} and \eqref{PDEenergyF2}, and applying \eqref{Equiv WeakF4 t}, \eqref{Equiv WeakF5 t} and \eqref{phi Z}, it follows
		\begin{align} \label{MediumEnergyF}
			\begin{aligned}
				\frac{d}{dt}E(\psi_{+},\psi_{-},\phi,Z_+,Z_-)  = 0,
			\end{aligned}
		\end{align}
		where $E$ is given in \eqref{energy}. Thus, \eqref{MediumEnergyF} implies the energy conservation \eqref{ModEnergy}.
	\end{proof}
	
	\subsection{Fully discrete scheme}
	In this paper, we discretize the reformulated GPP model \eqref{Equiv TargetEq} using the relaxation Crank–Nicolson finite element method, which combines the relaxation Crank–Nicolson scheme in time with the finite element method in space.
	
	Let $\Omega$ be partitioned into a collection of shape-regular elements $\mathcal{T}_h = \{K\}$, where each element $K$ has mesh size $h_K = \diam(K)$, and define $h = \max_{K \in \mathcal{T}_h} h_K$.
	We introduce the real-valued piecewise polynomial finite element spaces
	\[V_h = \{ v \in C(\Omega) {\cap V}: v \in \mathbb{P}^k(K),\quad \forall K \in \mathcal{T}_h  \},\quad \mathring{V}_h=V_h \cap \mathring{V},\]
	where $\mathbb{P}^k(K)$ denotes the space of real-valued polynomials of degree at most $k$ on $K$, with $k \ge 1$.
	
	The complex-valued finite element space is defined as
	\[
	\mathbf{V}_h = \{ v \in C(\Omega) {\cap \mathbf{V}}: v \in \mathbb{Q}^k(K),\quad \forall K \in \mathcal{T}_h  \},
	\]
	where $\mathbb{Q}^k(K)$ denotes the space of complex-valued polynomials of degree at most $k$ on $K$.
	
	We partition the time interval \([0, T]\) into \(N\) uniform subintervals with time step size \(\tau = {T}/{N}\). The resulting time grid is denoted by \(\{t_n : t_n = n\tau,\; 0 \le n \le N\}\).  
	For \(n \ge 0\) and any function \(v\), let \(v_h^n = v_h(x, t^n) \in V_h\) (or \(\mathbf{V}_h\)) be an approximation of \(v(x, t^n)\), where \(\tau > 0\) is the time step.  
	For simplicity, we introduce the notation
	\begin{equation*}
		D_{\tau}v^{n+1} := \frac{v^{n+1}-v^n}{\tau},\quad \bar v^{n+1/2}:=\frac{v^{n+1}+v^n}{2}.
	\end{equation*}
	
	Subsequently, the relaxation Crank–Nicolson finite element method for the reformulated GPP model \eqref{Equiv TargetEq} is to find $(\psi_{+,h}^{n+1},\psi_{-,h}^{n+1},\phi_h^{n+\frac{1}{2}},Z_{+,h}^{n+\frac{1}{2}},Z_{-,h}^{n+\frac{1}{2}}) \in \mathbf{V}_{h}\times \mathbf{V}_{h}\times \mathring{V}_h\times V_{h}\times V_{h}$ such that 
	\begin{subequations}\label{FullDS}
		\begin{align}
			&\left\langle \mathbf{i} D_{\tau} \psi_{+,h}^{n+1}, v_{+,h} \right\rangle =\frac{1}{2}A_0 \left( \bar\psi_{+,h}^{n+\frac{1}{2}}, v_{+,h} \right)  + \left\langle \left(g Z_{+,h}^{n+\frac{1}{2}} +  G Z_{-,h}^{n+\frac{1}{2}} + q \phi_h^{n+\frac{1}{2}} \right) \bar  \psi_{+,h}^{n+\frac{1}{2}} , v_{+,h} \right\rangle,  \label{FullDS1} \\
			&\left\langle \mathbf{i} D_{\tau} \psi_{-,h}^{n+1}, v_{-,h} \right\rangle = \frac{1}{2}A_0 \left(  \bar\psi_{-,h}^{n+\frac{1}{2}}, v_{-,h} \right) + \left\langle \left( g Z_{-,h}^{n+\frac{1}{2}} + G Z_{+,h}^{n+\frac{1}{2}} - q \phi_h^{n+\frac{1}{2}} \right) \bar \psi_{-,h}^{n+\frac{1}{2}} , v_{-,h} \right\rangle,  \label{FullDS2}\\
			&A_1\left( \phi_h^{n+\frac{1}{2}} , w_h \right) = \left( 4 \pi q \left( Z_{+,h}^{n+\frac{1}{2}} - Z_{-,h}^{n+\frac{1}{2}} \right), w_h \right),  \label{FullDS3} \\ 
			&\left( Z_{+,h}^{n+\frac{1}{2}} + Z_{+,h}^{n-\frac{1}{2}} , \chi_{+,h} \right) = \left( 2 \vert {\psi_{+,h}^n} \vert ^2, \chi_{+,h} \right),  \label{FullDS4} \\
			&\left( Z_{-,h}^{n+\frac{1}{2}} + Z_{-,h}^{n-\frac{1}{2}} , \chi_{-,h} \right) = \left( 2 \vert {\psi_{-,h}^n} \vert ^2, \chi_{-,h} \right),  \label{FullDS5}
		\end{align}
	\end{subequations}
	for any $(v_{+,h},v_{-,h},w_h,\chi_{+,h},\chi_{-,h}) \in \mathbf{V}_{h}\times \mathbf{V}_{h}\times V_{h}\times V_{h}\times V_{h}$. The initial values $\psi_{\pm,h}^{0} \in \mathbf{V}_h$, $Z_{\pm,h}^{-\frac{1}{2}} \in V_h$ are computed by 
	\begin{equation}\label{feminit}
		\langle \psi_{\pm,h}^{0}, v_{\pm,h}\rangle = \langle \psi_{\pm,0}(\mathbf{x}), v_{\pm,h} \rangle, \qquad ( Z_{\pm,h}^{-\frac{1}{2}},\chi_{\pm, h}) = ( |\psi_{\pm,0}(\mathbf{x})|^2, \chi_{\pm, h}),
	\end{equation}
	for any $v_{\pm,h} \in \mathbf{V}_{h}$, and $\chi_{\pm,,h} \in V_h$.
	
	\begin{lemma}\label{lem:ex}
		For any $\tau>0$, given $(\psi_{+,h}^{n},\psi_{-,h}^{n},\phi_h^{n-\frac{1}{2}},Z_{+,h}^{n-\frac{1}{2}},Z_{-,h}^{n-\frac{1}{2}}) \in \mathbf{V}_{h}\times \mathbf{V}_{h}\times \mathring{V}_{h}\times V_{h}\times V_{h}$, and assuming that
		\begin{equation}\label{preassume}
			\int_{\Omega}|\psi_{+,h}^n(\mathbf{x})|^2 - |\psi_{-,h}^n(\mathbf{x})|^2d\mathbf{x}=0, \quad \text{and} \quad \int_{\Omega}Z_{+,h}^{n-\frac{1}{2}} - Z_{-,h}^{n-\frac{1}{2}}d\mathbf{x}=0
		\end{equation}
		the relaxation Crank--Nicolson finite element scheme \eqref{FullDS} admits a unique solution 
		\begin{equation}
			(\psi_{+,h}^{n+1},\psi_{-,h}^{n+1},\phi_h^{n+\frac{1}{2}},Z_{+,h}^{n+\frac{1}{2}},Z_{-,h}^{n+\frac{1}{2}}) \in \mathbf{V}_{h}\times \mathbf{V}_{h}\times \mathring{V}_{h}\times V_{h}\times V_{h}. 
		\end{equation}
	\end{lemma}
	\begin{proof}
		Under the the assumptions in \eqref{preassume}, the compatibility condition for equation \eqref{FullDS3} is satisfied, namely,
		\begin{equation} \label{compability condition on n+}
			\int_{\Omega}Z_{+,h}^{n+\frac{1}{2}} - Z_{-,h}^{n+\frac{1}{2}}d\mathbf{x}=0,
		\end{equation}
		which is obtained by testing \eqref{FullDS5} and \eqref{FullDS4} with $\chi_{+,h}=\chi_{-,h}=1$, respectively, and subtracting \eqref{FullDS5} from \eqref{FullDS4}.
		
		Since the scheme \eqref{FullDS} is a finite-dimensional algebraic system, the existence of a solution is equivalent to its uniqueness. Therefore, it suffices to prove the uniqueness. Suppose there exists another solution $(\tilde\psi_{+,h}^{n+1},\tilde\psi_{-,h}^{n+1},\tilde\phi_h^{n+\frac{1}{2}},\tilde Z_{+,h}^{n+\frac{1}{2}},\tilde Z_{-,h}^{n+\frac{1}{2}})$, and let $(\delta\psi_{+,h}^{n+1},\delta\psi_{-,h}^{n+1},\delta\phi_h^{n+\frac{1}{2}},\delta Z_{+,h}^{n+\frac{1}{2}},\delta Z_{-,h}^{n+\frac{1}{2}})$ be the difference between the two solutions, i.e., $\delta f = f - \tilde f$. Then it satisfies:   
		\begin{subequations}
			\begin{align}
				&\frac{1}{\tau}\left\langle \mathbf{i} \delta \psi_{+,h}^{n+1}, v_{+,h} \right\rangle = \frac{1}{4}A_0\left( \delta \psi_{+,h}^{n+1}, v_{+,h} \right) + \frac{1}{2}\left\langle \left( g Z_{+,h}^{n+\frac{1}{2}} + G Z_{-,h}^{n+\frac{1}{2}} + q \phi_h^{n+\frac{1}{2}}\right) \delta  \psi_{+,h}^{n+1} , v_{+,h} \right\rangle,  \label{DFullDS1} \\
				&\frac{1}{\tau}\left\langle \mathbf{i} \delta \psi_{-,h}^{n+1}, v_{-,h} \right\rangle = \frac{1}{4} A_0 \left( \delta \psi_{-,h}^{n+1}, v_{-,h} \right) + \frac{1}{2}\left\langle \left( g Z_{-,h}^{n+\frac{1}{2}} +G Z_{+,h}^{n+\frac{1}{2}} - q \phi_h^{n+\frac{1}{2}}\right) \delta  \psi_{-,h}^{n+1} , v_{-,h} \right\rangle,  \label{DFullDS2}\\
				&A_1\left( \delta \phi_h^{n+\frac{1}{2}} , w_h \right) = \left( 4 \pi q \left( \delta Z_{+,h}^{n+\frac{1}{2}} - \delta Z_{-,h}^{n+\frac{1}{2}} \right), w_h \right),  \label{DFullDS3} \\ 
				&\left( \delta Z_{+,h}^{n+\frac{1}{2}} , \chi_{+,h} \right) = 0,  \label{DFullDS4} \\
				&\left( \delta Z_{-,h}^{n+\frac{1}{2}} , \chi_{-,h} \right) = 0.  \label{DFullDS5}
			\end{align} 
		\end{subequations}
		Taking $\chi_{\pm,h} = \delta Z_{\pm,h}^{n+\frac{1}{2}}$ separately in \eqref{DFullDS4} and \eqref{DFullDS5} implies $\Vert \delta Z_{\pm,h}^{n+\frac{1}{2}} \Vert = 0$, namely $\delta Z_{\pm,h}^{n+\frac{1}{2}}=0$. Then \eqref{DFullDS3} becomes 
		\begin{equation}
			A_1\left( \delta \phi_h^{n+\frac{1}{2}} , w_h \right) = 0. \label{DFullDS3_1}
		\end{equation}
		By taking $w_h = \delta \phi_h^{n+\frac{1}{2}}$ in \eqref{DFullDS3_1} and applying \eqref{coercivity and boundedness}, we have $\Vert \delta \phi_h^{n+\frac{1}{2}} \Vert_1 \le 0$, which implies $\delta \phi_h^{n+\frac{1}{2}} = 0$. 
		Finally, we take $v_{+,h} = \tau \delta \psi_{+,h}^{n+1}$ in \eqref{DFullDS1} to get
		\begin{equation}
			\mathbf{i} \Vert \delta \psi_{+,h}^{n+1} \Vert^2  = \frac{\tau}{4}A_0 \left( \delta \psi_{+,h}^{n+1}, \delta \psi_{+,h}^{n+1}  \right)+ \frac{\tau}{2} \langle (g Z_{+,h}^{n+\frac{1}{2}} + G Z_{-,h}^{n+\frac{1}{2}} + q \phi_h^{n+\frac{1}{2}} ) \delta \psi_{+,h}^{n+1} , \delta \psi_{+,h}^{n+1} \rangle. \label{DFullDS1_1}
		\end{equation}
		Note that the left side of the equation \eqref{DFullDS1_1} is purely imaginary, while the right side is a real number by \eqref{inner}, which implies that $\Vert \delta \psi_{+,h}^{n+1} \Vert = 0$, namely $\delta \psi_{+,h}^{n+1}=0$. Similarly, $\delta \psi_{-,h}^{n+1}=0$. Thus, the proof is complete.
	\end{proof}
	We postpone the discussion of the assumptions in \eqref{preassume} until after the following conservation result.
	\begin{lemma} \label{lemma: Des Mass and Energy}
		The relaxation Crank--Nicolson finite element scheme \eqref{FullDS} satisfies the discrete masses conservation law 
		\begin{equation} \label{full Des Mass}
			M_{\pm,h}^{n} = M_{\pm,h}^0,
		\end{equation}  
		and the discrete energy conservation law
		\begin{equation} \label{full Des Energy}
			E_h^{n} = E_h^0,
		\end{equation}    
		where the masses
		\begin{align*}
			&M_{\pm,h}^{n} = \int _{\Omega} \vert {\psi_{\pm,h}^{n}} \vert ^2 dx,
		\end{align*}
		and the discrete energy
		\begin{align*}
			E_{h}^{n} =& \frac{1}{2} \int _{\Omega} \vert {\nabla \psi_{+,h}^{n}} \vert ^2 + \vert {\nabla \psi_{-,h}^{n}} \vert ^2 dx + \frac{1}{2} \int _{\Omega} g\left( Z_{+,h}^{n+\frac{1}{2}}  Z_{+,h}^{n-\frac{1}{2}} + Z_{-,h}^{n+\frac{1}{2}}  Z_{-,h}^{n-\frac{1}{2}} \right) dx \\
			& + \frac{1}{2} \int _{\Omega} G\left( Z_{-,h}^{n+\frac{1}{2}}  Z_{+,h}^{n-\frac{1}{2}} + Z_{+,h}^{n+\frac{1}{2}}  Z_{-,h}^{n-\frac{1}{2}} \right) dx + \frac{1}{8\pi} \int _{\Omega} {\nabla \phi_{h}^{n-\frac{1}{2}}} {\nabla \phi_{h}^{n+\frac{1}{2}}} dx.
		\end{align*}
	\end{lemma}
	\begin{proof}
		By taking $v_{+,h} = \bar\psi_{+,h}^{n+\frac{1}{2}}$ and $v_{-,h} = \bar\psi_{-,h}^{n+\frac{1}{2}}$ in \eqref{FullDS1} and \eqref{FullDS2}, respectively, and comparing their imaginary parts, we obtain 
		\begin{equation} \label{DMASS des}
			D_\tau \int_\Omega |\psi_{+,h}^{n+1}|^2 dx =D_\tau \int
			_\Omega | \psi_{-,h}^{n+1}|^2 dx =0,
		\end{equation} 
		for $n=0,1,2,...$, which directly implies the discretized masses conservation  \eqref{full Des Mass}. 
		
		Taking $v_{+,h} = D_\tau \psi_{+,h}^{n+1}$ and 
		$v_{-,h} = D_\tau \psi_{-,h}^{n+1}$ in \eqref{FullDS1} and \eqref{FullDS2}, respectively, and comparing their real parts give
		\begin{align}
			\frac{1}{4\tau} \int _{\Omega} \vert {\nabla \psi_{+,h}^{n+1}} \vert ^2 - \vert {\nabla \psi_{+,h}^{n}} \vert ^2 dx + \frac{1}{2\tau} \int _{\Omega} \left(g Z_{+,h}^{n+\frac{1}{2}} + G Z_{-,h}^{n+\frac{1}{2}} + q \phi_h^{n+\frac{1}{2}}\right) \left( \vert { \psi_{+,h}^{n+1}} \vert ^2 - \vert { \psi_{+,h}^{n}} \vert ^2 \right) dx = 0, \label{descrationpart1}
			\\
			\frac{1}{4\tau} \int _{\Omega} \vert {\nabla \psi_{-,h}^{n+1}} \vert ^2 - \vert {\nabla \psi_{-,h}^{n}} \vert ^2 dx + \frac{1}{2\tau} \int _{\Omega} \left(g Z_{-,h}^{n+\frac{1}{2}} + G Z_{+,h}^{n+\frac{1}{2}} - q \phi_h^{n+\frac{1}{2}}\right) \left( \vert { \psi_{-,h}^{n+1}} \vert ^2 - \vert { \psi_{-,h}^{n}} \vert ^2 \right) dx = 0. \label{descrationpart2}
		\end{align}
		Summing \eqref{descrationpart1} and \eqref{descrationpart2} gives
		\begin{equation}\label{Dsum}
			D_1^n + D_2^n + D_3^n = 0,
		\end{equation}
		where
		\begin{align*}
			&D_1^n = \frac{1}{4\tau} \int _{\Omega} \vert {\nabla \psi_{+,h}^{n+1}} \vert ^2 - \vert {\nabla \psi_{+,h}^{n}} \vert ^2 + \vert {\nabla \psi_{-,h}^{n+1}} \vert ^2 - \vert {\nabla \psi_{-,h}^{n}} \vert ^2 dx, \\ &
			D_2^n = \frac{1}{2\tau} \int _{\Omega} \left(g Z_{+,h}^{n+\frac{1}{2}} + G Z_{-,h}^{n+\frac{1}{2}} \right) \left( \vert { \psi_{+,h}^{n+1}} \vert ^2 - \vert { \psi_{+,h}^{n}} \vert ^2 \right) + \left(g Z_{-,h}^{n+\frac{1}{2}} + G Z_{+,h}^{n+\frac{1}{2}} \right) \left( \vert { \psi_{-,h}^{n+1}} \vert ^2 - \vert { \psi_{-,h}^{n}} \vert ^2 \right) dx, \\ &
			D_3^n = \frac{1}{2\tau} \int _{\Omega} q \phi_h^{n+\frac{1}{2}} \left( \vert { \psi_{+,h}^{n+1}} \vert ^2 - \vert { \psi_{+,h}^{n}} \vert ^2 \right) - q \phi_h^{n+\frac{1}{2}} \left( \vert { \psi_{-,h}^{n+1}} \vert ^2 - \vert { \psi_{-,h}^{n}} \vert ^2 \right) dx.
		\end{align*}
		
		We now proceed to verify that
		\begin{equation}
			D_\tau E_h^{n+1} = 0.
		\end{equation}
		We begin by simplifying $D_2^n$ using \eqref{FullDS4} and \eqref{FullDS5}. Specifically, we subtract the $n$-th layer of \eqref{FullDS4} from its $(n+1)$-th layer, and perform an analogous subtraction for \eqref{FullDS5}. This yields 
		\begin{align}
			&\left( Z_{+,h}^{n+\frac{3}{2}} - Z_{+,h}^{n-\frac{1}{2}} , \chi_{+,h} \right) = \left( 2 \vert {\psi_{+,h}^{n+1}} \vert ^2-2 \vert {\psi_{+,h}^n} \vert ^2 , \chi_{+,h} \right),  \label{D FullDS4} \\
			&\left( Z_{-,h}^{n+\frac{3}{2}} - Z_{-,h}^{n-\frac{1}{2}} , \chi_{-,h} \right) = \left( 2 \vert {\psi_{-,h}^{n+1}} \vert ^2 - 2 \vert {\psi_{-,h}^n} \vert ^2, \chi_{-,h} \right).  \label{D FullDS5}
		\end{align}
		Taking $\chi_{+,h} = g Z_{+,h}^{n+\frac{1}{2}} + G Z_{-,h}^{n+\frac{1}{2}}$ and $\chi_{-,h} = g Z_{-,h}^{n+\frac{1}{2}} + G Z_{+,h}^{n+\frac{1}{2}}$ in \eqref{D FullDS4} and \eqref{D FullDS5}, respectively, and then substituting the resulting expressions into $D_2^n$ yield
		\begin{align}
			\begin{aligned}
				2\tau D_2^n &= \frac{1}{2}\int _{\Omega} \left(g Z_{+,h}^{n+\frac{1}{2}} + G Z_{-,h}^{n+\frac{1}{2}}\right) \left( Z_{+,h}^{n+\frac{3}{2}} - Z_{+,h}^{n-\frac{1}{2}} \right)  +  \left(g Z_{-,h}^{n+\frac{1}{2}} + G Z_{+,h}^{n+\frac{1}{2}} \right) \left( Z_{-,h}^{n+\frac{3}{2}} - Z_{-,h}^{n-\frac{1}{2}} \right) dx \\ 
				= & \frac{1}{2}\int _{\Omega} g\left(\left( Z_{+,h}^{n+\frac{1}{2}}Z_{+,h}^{n+\frac{3}{2}} +Z_{-,h}^{n+\frac{1}{2}} Z_{-,h}^{n+\frac{3}{2}} \right) - \left( Z_{+,h}^{n+\frac{1}{2}} Z_{+,h}^{n-\frac{1}{2}} + Z_{-,h}^{n+\frac{1}{2}} Z_{-,h}^{n-\frac{1}{2}} \right) \right) dx \\ 
				& + \frac{1}{2}\int _{\Omega} G\left(\left( Z_{-,h}^{n+\frac{1}{2}}Z_{+,h}^{n+\frac{3}{2}} +Z_{+,h}^{n+\frac{1}{2}} Z_{-,h}^{n+\frac{3}{2}} \right) - \left( Z_{-,h}^{n+\frac{1}{2}} Z_{+,h}^{n-\frac{1}{2}} + Z_{+,h}^{n+\frac{1}{2}} Z_{-,h}^{n-\frac{1}{2}} \right) \right) dx.
			\end{aligned} \label{D2^n}
		\end{align}
		For the term $D_3^n$, we consider the $(n+1)$-th and $(n-1)$-th layers of \eqref{FullDS3}. By taking 
		$w_h = \phi_h^{n+\frac{1}{2}}$, we have
		\begin{align}
			&\int_{\Omega} \nabla \phi_h^{n+\frac{3}{2}} \cdot \nabla \phi_h^{n+\frac{1}{2}} dx = 4 \pi \int_{\Omega} q \left( Z_{+,h}^{n+\frac{3}{2}} - Z_{-,h}^{n+\frac{3}{2}}\right) \phi_h^{n+\frac{1}{2}} dx, \label{phiphi1}
			\\
			&\int_{\Omega} \nabla \phi_h^{n-\frac{1}{2}} \cdot \nabla \phi_h^{n+\frac{1}{2}} dx = 4 \pi \int_{\Omega} q \left( Z_{+,h}^{n-\frac{1}{2}} - Z_{-,h}^{n-\frac{1}{2}}\right) \phi_h^{n+\frac{1}{2}} dx. \label{phiphi2}
		\end{align}
		Subtracting \eqref{phiphi2} from \eqref{phiphi1} and performing algebraic transformations yield
		\begin{align}
			\begin{aligned}
				&\frac{1}{8\pi}\int_{\Omega} \nabla \phi_h^{n+\frac{3}{2}} \cdot \nabla \phi_h^{n+\frac{1}{2}} - \nabla \phi_h^{n-\frac{1}{2}} \cdot \nabla \phi_h^{n+\frac{1}{2}} dx \\ &= \frac{1}{2} \int_{\Omega} q \phi_h^{n+\frac{1}{2}} \left( \left( Z_{+,h}^{n+\frac{3}{2}} - Z_{+,h}^{n-\frac{1}{2}} \right) - \left( Z_{-,h}^{n+\frac{3}{2}} - Z_{-,h}^{n-\frac{1}{2}}\right) \right) dx  \\
				&= \int_{\Omega} q \phi_h^{n+\frac{1}{2}} \left( \left( \vert { \psi_{+,h}^{n+1}} \vert ^2 - \vert { \psi_{+,h}^{n}} \vert ^2 \right) - \left( \vert { \psi_{-,h}^{n+1}} \vert ^2 - \vert { \psi_{-,h}^{n}} \vert ^2 \right) \right) dx = 2\tau D_3^n,
			\end{aligned} \label{phi and Z}
		\end{align}
		where in the second equality, we have used \eqref{D FullDS4} and \eqref{D FullDS5} with $\chi_{+,h}= \chi_{-,h} =  q \phi_h^{n+\frac{1}{2}}$.
		
		Finally, summing the expressions for $2\tau D_1^n$, $2\tau D_2^n$ from \eqref{D2^n}, and $2\tau D_3^n$ from \eqref{phi and Z}, together with \eqref{Dsum}, yields
		\begin{equation}
			D_{\tau} E_h^{n+1} = 2(D_1^n + D_2^n + D_3^n) = 0,
		\end{equation}
		which completes the proof.
	\end{proof}
	
	To this end, we revisit the assumptions in \eqref{preassume}. 
	\begin{remark}
		For $n = 0$, by taking $\chi_{\pm,h} = 1$ in \eqref{feminit} and using the assumption of equal initial masses \eqref{masscompat}, it is straightforward to verify that the second equality in the assumptions of \eqref{preassume} holds, namely,
		\[
		\int_{\Omega} \left( Z_{+,h}^{-\frac{1}{2}} - Z_{-,h}^{-\frac{1}{2}} \right) \, d\mathbf{x} = 0.
		\]  
	\end{remark}
	\subsection{Implementation details}
	We finally state the algorithm for the relaxation Crank--Nicolson finite element scheme \eqref{FullDS} in Algorithm \ref{alg:CNFEM}.
	
	\begin{algorithm}[ht]
		\caption{The relaxation Crank--Nicolson finite element scheme for the CGPP system}
		\label{alg:CNFEM}
		\KwIn{Initial data $\psi_{\pm,h}^{0}$}
		\KwOut{Wave function $\psi_{\pm,h}^{N}$ and $\phi_{h}^{N-\frac{1}{2}}$}  
		Initialization for $Z_{\pm,h}^{-\frac{1}{2}}$ as formulated in \eqref{feminit}.\\
		\For{$n=0$ \KwTo $\lfloor T/\Delta t \rfloor$}{
			Compute $Z_{+,h}^{n+\frac{1}{2}}$ from \eqref{FullDS4}, and $Z_{-,h}^{n+\frac{1}{2}}$ from \eqref{FullDS5} in parallel; \\
			Compute $\phi_h^{n+\frac{1}{2}}$ from \eqref{FullDS3};\\
			Compute \( \psi_{+,h}^{n+1} \) from \eqref{FullDS1}, and \( \psi_{-,h}^{n+1} \) from \eqref{FullDS2} in parallel.\\
		}
	\end{algorithm}
	
	The equations in Algorithm \ref{alg:CNFEM} are fully decoupled and linear, and therefore require no iterations. Moreover, the scheme advances sequentially without the need to solve a coupled system.

	\section{Error estimates for the fully discrete system} \label{secErrorEstimate}
	We define the nodal interpolation operator $\Pi_h:V \rightarrow V_h$ as $(\Pi_h u)(\mathbf{x}_i) = u(\mathbf{x}_i)$ with predefined nodes $\{\mathbf{x}_i\}_{i=1}^N$. The Ritz projection operator $R_h:\mathring{V} \rightarrow \mathring{V}_h$ is defined by the following problem: for any $v \in \mathring{V}$, find $R_h v \in \mathring{V}_h$ such that
	\begin{equation}\label{Galerkin orthgonal}
		A_1(R_hv- v, w) = 0,\quad \forall w \in V_h.
	\end{equation} 
	The following lemmas describe the properties of the Ritz projection operator $R_h$.
	\begin{lemma}
		The operators $\Pi_h$ and $R_h$ satisfy the following properties 
		\begin{align}
			&\Vert v - \Pi_h v \Vert + h\Vert \nabla(v - \Pi_h v) \Vert + h \Vert v - \Pi_h v \Vert_{\infty} \le Ch^{k+1} \Vert v \Vert_{k+1},\quad \forall v \in H^{k+1}(\Omega), \label{Interpolation error}\\ 
			&\Vert v - R_h v \Vert + h\Vert \nabla(v - R_h v) \Vert \le Ch^{k+1} \Vert v \Vert_{k+1},\quad \forall v \in \mathring{V} \cap H^{k+1}(\Omega). \label{Ritz projection error}
		\end{align}
	\end{lemma}
	\begin{proof}
		The result of the interpolation error estimate \eqref{Interpolation error} can be found in \cite{brenner2008mathematical}. We now proceed to prove the Ritz projection error estimate \eqref{Ritz projection error}. 
		\begin{equation}
			\Vert \nabla(v - R_h v) \Vert^2 =   A_1(v-R_h v, v-R_h v) =  A_1(v-R_h v, v-\Pi_h v) + A_1(v-R_h v, \Pi_h v-R_h v).
		\end{equation}
		Since $\Pi_h v-R_h v \in V_h$, by applying \eqref{Galerkin orthgonal}, we have $A_1(v-R_h v, \Pi_h v-R_h v)=0$. Hence, using \eqref{coercivity and boundedness}, we have
		\begin{equation}
			\Vert \nabla(v - R_h v) \Vert^2 =   A_1(v-R_h v, v-\Pi_h v) \le \gamma_2\Vert \nabla(v - R_h v) \Vert\Vert \nabla(v - \Pi_h v) \Vert,
		\end{equation}
		which upon applying the interpolation error estimate \eqref{Interpolation error} for $\nabla(v - \Pi_h v)$, we obtain
		\begin{equation} \label{H1 error real number}
			\Vert \nabla(v - R_h v) \Vert \le Ch^k \Vert v \Vert_{k+1}.
		\end{equation}
		To derive the $L^2$-norm error estimate, we introduce the dual problem
		\begin{equation}
			\begin{aligned}
				-\Delta &\varphi = e:=v-R_hv\;\text{in}\; \Omega,  \\
				& \varphi\; \text{is periodic on}\;\partial \Omega.
			\end{aligned}
		\end{equation}
		The corresponding weak formulation is to find $\varphi \in \mathring{V}_h$ such that
		\begin{equation}
			A_1(\varphi,\nu) = (e,\nu),\; \forall \nu \in \mathring{V},
		\end{equation} 
		by taking $\nu=e$, we obtain
		\begin{equation}
			\begin{aligned}
				\|e\|^2=A_1(\varphi,e) = A_1(\varphi-\Pi_h \varphi,e)+A_1(\Pi_h \varphi,e) \le \gamma_2\|\nabla(\varphi-\Pi_h\varphi)\|\|\nabla e\| ,     \end{aligned}
		\end{equation}
		where we apply \eqref{Galerkin orthgonal} and \eqref{coercivity and boundedness} again.
		
		Because of the quasi-optimal interpolant error in $H^1$, and the elliptic regularity \cite{brenner2008mathematical}, combined with the $H^1$-error \eqref{H1 error real number}, we have
		\begin{equation} \label{L2 error real number}
			\|e\|^2 \le \gamma_2\|\nabla(\varphi-\Pi_h\varphi)\|\|\nabla e\| \le Ch\|\varphi\|_2 \cdot h^k\|v\|_{k+1} \le Ch^{k+1}\|e\| \cdot \|v\|_{k+1},
		\end{equation}
		which together with \eqref{H1 error real number} imply \eqref{Ritz projection error}.
	\end{proof}
	
	Similarly, we define the complex interpolation operator $\mathbf{\Pi}_h:\mathbf{V} \rightarrow \mathbf{V}_h$ as $(\mathbf{\Pi}_h u)(\mathbf{x}_i) = u(\mathbf{x}_i)$ and complex Ritz projection $\mathbf{R}_h: \mathbf{V} \rightarrow \mathbf{V}_h$: for given $v \in \mathbf{V}$, find $\mathbf{R}_h v \in \mathbf{V}_h$ such that
	\begin{align}
		A_0(v-\mathbf{R}_h v, w) &= 0,\quad \forall w \in \mathbf{V}_h, \\
		\langle v-\mathbf{R}_h v, 1\rangle &= 0,\label{aver0}
	\end{align}
	where \eqref{aver0} ensures the uniqueness of the Ritz projection $\mathbf{R}_h v$. Then, it holds the following result.
	\begin{lemma}
		The operators $\mathbf{\Pi}_h$ and $\mathbf{R}_h$ satisfy the following properties 
		\begin{align}
			&\Vert v - \mathbf{\Pi}_h v \Vert + h\Vert \nabla(v - \mathbf{\Pi}_h v) \Vert + h \Vert v - \mathbf{\Pi}_h v \Vert_{\infty} \le Ch^{k+1} \Vert v \Vert_{k+1},\quad \forall v \in \mathbf{H}^{k+1}(\Omega), \label{Interpolation error complex}\\ 
			&\Vert v - \mathbf{R}_h v \Vert + h\Vert \nabla(v -  \mathbf{R}_h v) \Vert \le Ch^{k+1} \left\| v-\frac{1}{|\Omega|}\int_{\Omega}vd\mathbf{x} \right\|_{k+1},\quad \forall v \in \mathbf{V} \cap \mathbf{H}^{k+1}(\Omega). \label{Ritz projection error complex}
		\end{align}
	\end{lemma}
	\begin{proof}
		We assume that $v=a+b\mathbf{i} \in \mathbf{V}\cap\mathbf{H}^{k+1}(\Omega)$, which implies that $\mathbf{\Pi}_hv=\Pi_ha+\Pi_hb\mathbf{i}$. Then we have $\|v\|^2_{\mathbf{H}^s(\Omega)}=\|a\|^2_{{H}^s(\Omega)}+\|b\|^2_{{H}^s(\Omega)}$. Note that $a,b\in H^{k+1}(\Omega)$, then
		\begin{equation*}
			\begin{aligned}
				& \Vert v - \mathbf{\Pi}_h v \Vert + h\Vert \nabla(v - \mathbf{\Pi}_h v) \Vert + h \Vert v - \mathbf{\Pi}_h v \Vert_{\infty} \\ & \le (\Vert a - \Pi_h a \Vert + h\Vert \nabla(a - \Pi_h a) \Vert + h \Vert a - \Pi_h a \Vert_{\infty}) + (\Vert b - \Pi_h b \Vert + h\Vert \nabla(b - \Pi_h b) \Vert + h \Vert b - \Pi_h b \Vert_{\infty}) \\ & \le Ch^{k+1} (\Vert a \Vert_{k+1}+\Vert b \Vert_{k+1}) \le  Ch^{k+1}\|v\|_{k+1},
			\end{aligned}
		\end{equation*}
		where we have applied \eqref{Interpolation error}. This completes the proof of \eqref{Interpolation error complex}. 
		
		To verify \eqref{Ritz projection error complex}, we assume, $v=a+b\mathbf{i}$, and define $v_0 = v-\frac{1}{|\Omega|}\int_{\Omega}vd\mathbf{x}=a_0+b_0\mathbf{i}$. Then it follows that $a_0=a-\frac{1}{|\Omega|}\int_{\Omega}ad\mathbf{x}$, $b_0=b-\frac{1}{|\Omega|}\int_{\Omega}bd\mathbf{x}$, and $a_0,b_0 \in \mathring{V} \cap H^{k+1}(\Omega)$. \\
		It is easy to verify that $\mathbf{R}_hv=R_hv,\;\forall v \in {\mathring{V}} \cap {H}^{k+1}(\Omega)$, and $\mathbf{R}_hc=c$, where $c$ is a constant. Then,
		\begin{equation} 
			\Vert v_0 - \mathbf{R}_h v_0 \Vert = \left\| v - \frac{1}{|\Omega|}\int_{\Omega}vd\mathbf{x} - \mathbf{R}_h v + \mathbf{R}_h \left(\frac{1}{|\Omega|}\int_{\Omega}vd\mathbf{x}\right)\right\|=\Vert v - \mathbf{R}_h v \Vert.
		\end{equation}
		Therefore,
		\begin{equation}\label{v-RCv}
			\begin{aligned}
				\Vert v - \mathbf{R}_h v \Vert^2 &= \Vert v_0 - \mathbf{R}_h v_0 \Vert^2 = \Vert a_0 - \mathbf{R}_h a_0 \Vert^2 + \Vert b_0 - \mathbf{R}_h b_0 \Vert^2 \\
				&= \Vert a_0 - {R}_h a_0 \Vert^2 + \Vert b_0 - {R}_h b_0 \Vert^2 \le (Ch^{k+1})^2(\|a_0\|_{k+1}^2+\|b_0\|_{k+1}^2),
			\end{aligned}    
		\end{equation}
		where the last inequality follows from \eqref{Ritz projection error}. Similarly, we can derive
		\begin{equation}\label{nablav-RCv}
			(h\Vert \nabla(v - \mathbf{R}_h v )\Vert )^2 \le (Ch^{k+1})^2(\|a_0\|_{k+1}^2+\|b_0\|_{k+1}^2) .
		\end{equation}
		Combining \eqref{v-RCv} and \eqref{nablav-RCv}, we obtain the estimate \eqref{Ritz projection error complex}.
	\end{proof}
	\begin{lemma}\label{lem:inftol2}
		For $\forall v \in V_h$, the following inverse inequality holds
		\begin{equation}
			\Vert v \Vert_{\infty}\le Ch^{-\frac{d}{2}}\Vert v \Vert. \label{inverse inequality}
		\end{equation}
		Moreover, for all $v \in \mathring{V}
		\cap H^{k+1}(\Omega)$, there exists a constant $h_0>0$ such that, $0<h\leq h_0$, the Ritz projection satisfies
		\begin{equation} \label{boundedness of Ritz}
			\Vert R_h v \Vert_{\infty}\le C,
		\end{equation}
		where $C$ in \eqref{boundedness of Ritz} depends on $\Vert v \Vert_{k+1}$ and $\Vert v \Vert_{\infty}$, but is independent of $h$.
	\end{lemma}
	\begin{proof}
		The conclusion of \eqref{inverse inequality} can be found in \cite{bookjin_numerical_2023}. To prove \eqref{boundedness of Ritz}, we apply the triangle inequality together with \eqref{inverse inequality} and \eqref{Ritz projection error} to obtain
		\begin{align*}
			\Vert v - R_h v \Vert_{\infty} &\le \Vert v - \Pi_h v \Vert_{\infty} + \Vert \Pi_h v - R_h v \Vert_{\infty} \le \Vert v - \Pi_h v \Vert_{\infty} + Ch^{-\frac{d}{2}} \Vert \Pi_h v - R_h v \Vert \\ & \le \Vert v - \Pi_h v \Vert_{\infty} + Ch^{-\frac{d}{2}} (\Vert v - R_h v \Vert + \Vert v - \Pi_h v\Vert) \le Ch^{(k+1-\frac{d}{2})}\Vert v \Vert_{k+1}. 
		\end{align*}
		Since $d\leq 3$, applying the triangle inequality once more, we obtain for $h\leq h_0$,
		\begin{equation*}
			\Vert R_h v \Vert_{\infty}\le \Vert v - R_h v \Vert_{\infty} + \Vert v \Vert_{\infty} \le C,
		\end{equation*}
		which establishes the boundedness of the Ritz projection.
	\end{proof}
	
	\begin{lemma} \label{Gronwall’s inequality}
		(Discrete Gronwall’s inequality \cite{GrownwallInequality}). Let $\tau$, $B$, and $a_k, b_k, c_k, \gamma_k$ for $k \ge 0$ be nonnegative numbers, satisfying
		\begin{equation}
			a_n + \tau \sum_{k=0}^n b_k \le \tau \sum_{k=0}^n \gamma_k a_k + \tau \sum_{k=0}^n c_k + B, \quad n \ge 0.
		\end{equation}
		Suppose that $\tau \gamma_k <1$, for $k \ge 0$, and $\sigma_k = (1-\tau \gamma_k)^{-1}$. Then
		\begin{equation}
			a_n + \tau \sum_{k=0}^n b_k \le \exp \{\tau \sum_{k=0}^n \sigma_k \gamma_k\} (\tau \sum_{k=0}^n c_k + B).
		\end{equation}
	\end{lemma}
	\begin{lemma} \label{propties of A1} \cite{liu2026structure} For given $f \in L^2(\Omega)$, if $a \in \mathring{V}$ satisfies
		\begin{equation}
			A_1(a,w_h)=(f,w_h), \quad \forall w_h \in \mathring{V}_h,
		\end{equation}
		there exists a constant $C>0$ such that 
		\begin{equation}
			\Vert a \Vert \le C\left(\Vert f\Vert + h \min_{a_h \in \mathring{V}_h} \Vert a_h - a \Vert_1 \right). \label{BD of A1}
		\end{equation}
	\end{lemma}

	We define the discrete Laplacian operator $\Delta_h: \mathbf{H}_0^1(\Omega) \rightarrow \mathbf{V}_h$ as
	\begin{equation}
		\langle -\Delta_h u, v_h \rangle = \langle \nabla u, \nabla v_h \rangle, \quad \forall v_h \in \mathbf{V}_h.
	\end{equation}
	Subsequently, we define the linear operators ${S}_h,{T}_h: \mathbf{V}_h \rightarrow \mathbf{V}_h$ as follows
	\begin{align}
		&\langle {S}_h u_h, v_h \rangle = \langle (\mathbf{I}_h - \mathbf{i}\frac{\tau}{4}\Delta_h)u_h, v_h\rangle, \quad \forall v_h \in \mathbf{V}_h, \label{definitionS}\\
		&\langle {T}_h u_h, v_h \rangle = \langle (\mathbf{I}_h + \mathbf{i}\frac{\tau}{4}\Delta_h)u_h, v_h \rangle, \quad \forall v_h \in \mathbf{V}_h, \label{definitionT}
	\end{align}
	where $\mathbf{I}_h: \mathbf{V}_h \rightarrow \mathbf{V}_h$ is an identity operator. We take $v_h = u_h$ in \eqref{definitionS}, then taking the real part of both sides of the equation
	\begin{equation}
		\text{Re}({S}_h v_h,v_h) = \Vert v_h \Vert^2,\quad \forall v_h \in \mathbf{V}_h,
	\end{equation}
	means that $\text{Ker}({S}_h)=\{0\}$. Similarly $\text{Ker}({T}_h)=\{0\}$. Therefore, the operator ${S}_h$ and ${T}_h$ are invertible.
	\begin{lemma} \label{lemma: B}
		\cite{ShBhestimation} The operators ${S}_h$ and ${T}_h$ fulfill
		\begin{align}
			&\Vert {S}_h^{-1} (v_h) \Vert \le \Vert v_h \Vert, \quad \forall v_h \in \mathbf{V}_h, \\ 
			& \Vert {B}_h (v_h) \Vert \le \Vert v_h \Vert, \quad \forall v_h \in \mathbf{V}_h,
		\end{align}
		where ${B}_h = {S}_h^{-1} {T}_h$ is linear operator.
	\end{lemma}
	\begin{lemma} \label{lemma: sequence lemma}
		\cite{liu2026structure} Let $\{y^n\}_{n=1}^N$ be a sequence in $ \mathbf{V}_h$ satisfying
		\begin{equation}
			y^{n+1} = ({B}_h - \mathbf{I}_h)y^n + {B}_h y^{n-1} + {S}_h^{-1} \Gamma^{n+1}, \quad n \geq 2,
		\end{equation}
		where $\{\Gamma^{n+1}\}_{n=2}^{N-1}$ are given functions in $ \mathbf{V}_h$, then for $n\geq 2$ it follows
		\begin{equation}
			\Vert y^{n+1} \Vert + \Vert y^{n} \Vert \le 2 \Vert {S}_h (y^2) \Vert + 2 \Vert {S}_h (y^1) \Vert + {2}\sum_{l=2}^n \Vert \Gamma^{l+1} \Vert.
		\end{equation}
	\end{lemma}
	\begin{lemma} \cite{liu2026structure} \label{lemma2}
		Let $v_a, v_b, z_a, z_b: \Omega \rightarrow \mathbb{C}$ and $S(v_a,v_b,z_a,z_b):=\vert v_a \vert^2 - \vert v_b \vert^2 - \vert z_a \vert^2 + \vert z_b \vert^2$. Then it holds that
		\begin{align}
			\begin{aligned}
				\Vert S(v_a, v_b, z_a, z_b) \Vert \le &2 \Vert z_a - z_b \Vert_{\infty} \Vert v_b - z_b \Vert  + K(v_a, v_b, z_a, z_b)\Vert v_a - v_b -z_a + z_b \Vert,
			\end{aligned}
		\end{align}
		where $K(v_a, v_b, z_a, z_b)=\Vert v_a \Vert_{\infty} + \Vert v_b \Vert_{\infty} + \Vert z_a - z_b \Vert_{\infty}$. 
	\end{lemma}
	To facilitate error estimation, we impose the following regularity assumptions on the exact solutions $\psi_+,\psi_-,\phi,Z_+ $ and $Z_-$.
	\begin{align} \label{regularity assumption}
		\begin{aligned}
			& \psi_{\pm}, \psi_{\pm,t} \in L^{\infty}(0,T;\mathbf{H}^{k+1}(\Omega)), \quad 
			& Z_{\pm}, Z_{\pm,t}, \phi \in L^{\infty}(0,T;H^{k+1}(\Omega)),  \\ & \psi_{\pm,tt} \in L^{\infty}(0,T;\mathbf{H}^{2}(\Omega)), \quad & Z_{\pm,tt} \in L^{\infty}(0,T;H^{2}(\Omega)), \\ 
			& \psi_{\pm,ttt},\psi_{\pm,tttt} \in L^{\infty}(0,T;\mathbf{L}^{2}(\Omega)), \quad & Z_{\pm,ttt} \in L^{\infty}(0,T;L^{2}(\Omega)).
		\end{aligned} 
	\end{align}
	Based on the regularity assumption \eqref{regularity assumption} and the property \eqref{boundedness of Ritz}, we can demonstrate the boundedness of the exact solution and its Ritz projection
	\begin{align}
		\begin{aligned}
			&\Vert \psi_+^n \Vert_{\infty}, \Vert \psi_-^n \Vert_{\infty} \le C_{\psi}, \quad \Vert Z_+^{n-\frac12} \Vert_{\infty}, \Vert Z_-^{n-\frac12} \Vert_{\infty} \le C_{Z}, \quad  \Vert \phi^{n-\frac12} \Vert_{\infty} \le C_{\phi}, \\
			&\Vert R_h \psi_+^n \Vert_{\infty}, \Vert R_h \psi_-^n \Vert_{\infty} \le D_{\psi}, \quad \Vert R_h Z_+^{n-\frac12} \Vert_{\infty}, \Vert R_h Z_-^{n-\frac12} \Vert_{\infty} \le D_{Z}, \quad  \Vert R_h \phi^{n-\frac12} \Vert_{\infty} \le D_{\phi}, \\        
		\end{aligned} \label{boundedness of the true solution}
	\end{align}
	where the constants 
	\begin{align}
		\begin{aligned}
			&C_{\psi} = \sup_{0\le n \le N} \Vert \psi_{\pm}^n \Vert_{\infty}, \quad C_{Z} = \sup_{0\le n \le N} \Vert Z_{\pm}^{n-1/2} \Vert_{\infty}, \quad  C_{\phi} = \sup_{0\le n \le N} \Vert \phi^{n-1/2} \Vert_{\infty}, \\
			&D_{\psi} = \sup_{0\le n \le N} \Vert R_h \psi_{\pm}^n \Vert_{\infty}, \quad D_{Z} = \sup_{0\le n \le N} \Vert R_h Z_{\pm}^{n-1/2} \Vert_{\infty}, \quad  D_{\phi} = \sup_{0\le n \le N} \Vert R_h \phi^{n-1/2} \Vert_{\infty}. \\        
		\end{aligned}
	\end{align}
	where $\sup \Vert f_{\pm} \Vert_{\infty}:=\max \left\{\sup \Vert f_{+} \Vert_{\infty}, \sup \Vert f_{-} \Vert_{\infty}\right\}$. \par
	Note that the exact solution of \eqref{Equiv TargetEq} satisfies 
	\begin{subequations}
		\begin{align}
			&{\left\langle {\mathbf{i}{D_\tau }\psi _ + ^{n + 1},v_+} \right\rangle = \frac{1}{2}{A_0}\left( {\bar \psi _ + ^{n + \frac{1}{2}},v_+} \right) + \left\langle \left( {gZ_ + ^{n + \frac{1}{2}} + GZ_ - ^{n + \frac{1}{2}} + q\phi _{}^{n + \frac{1}{2}}} \right)\bar \psi _ + ^{n + \frac{1}{2}}+ {R_{1, + }^{n + 1},v_+} \right\rangle},\\
			&{\left\langle {\mathbf{i}{D_\tau }\psi _ - ^{n + 1},v_-} \right\rangle = \frac{1}{2}{A_0}\left( {\bar \psi _ - ^{n + \frac{1}{2}},v_-} \right) + \left\langle \left( {gZ_ - ^{n + \frac{1}{2}} + GZ_ + ^{n + \frac{1}{2}} - q\phi _{}^{n + \frac{1}{2}}} \right)\bar \psi _ - ^{n + \frac{1}{2}}+ {R_{1, - }^{n + 1},v_-} \right\rangle},\\
			&{{A_1}\left( {\phi _{}^{n + \frac{1}{2}},w} \right) = \left( {4\pi qZ_ + ^{n + \frac{1}{2}},w} \right) - \left( {4\pi qZ_ - ^{n + \frac{1}{2}},w} \right)},\\
			&{\left( {Z_ + ^{n + \frac{1}{2}} + Z_ + ^{n - \frac{1}{2}},{\chi _ + }} \right) = \left( {S_{1, + }^n,{\chi _ + }} \right) + 2\left( {{{\left| {\psi _ + ^n} \right|}^2},{\chi _ + }} \right)},\\
			&{\left( {Z_ - ^{n + \frac{1}{2}} + Z_ + ^{n - \frac{1}{2}},{\chi _ - }} \right) = \left( {S_{1, - }^n,{\chi _ - }} \right) + 2\left( {{{\left| {\psi _ - ^n} \right|}^2},{\chi _ - }} \right)},
		\end{align}\label{HalfDSwithTylor}
	\end{subequations}
	with remainder terms resulting from the Temporal discretization
	\begin{align}
		&S_{1, + }^n = Z_ + ^{n + \frac{1}{2}} + Z_ + ^{n - \frac{1}{2}} - 2Z_ + ^n,
		\\
		&S_{1, - }^n = Z_ - ^{n + \frac{1}{2}} + Z_ - ^{n - \frac{1}{2}} - 2Z_ - ^n,\\
		&\begin{aligned}
			R_{1, + }^{n + 1} &=  - \mathbf{i}\left( {{\partial _t}\psi _ + ^{n + \frac{1}{2}} - {D_\tau }\psi _ + ^{n + 1}} \right) + \frac{1}{2}\Delta \left( {\bar \psi _ + ^{n + \frac{1}{2}} - \psi _ + ^{n + \frac{1}{2}}} \right) \\ &+ \left( {gZ_ + ^{n + \frac{1}{2}} + GZ_ - ^{n + \frac{1}{2}} + q\phi _{}^{n + \frac{1}{2}}} \right)\left( {\psi _ + ^{n + \frac{1}{2}} - \bar \psi _ + ^{n + \frac{1}{2}}} \right),
		\end{aligned}\\
		&\begin{aligned}
			R_{1, - }^{n + 1} &=  - \mathbf{i}\left( {{\partial _t}\psi _ - ^{n + \frac{1}{2}} - {D_\tau }\psi _ - ^{n + 1}} \right) + \frac{1}{2}\Delta \left( {\bar \psi _ - ^{n + \frac{1}{2}} - \psi _ - ^{n + \frac{1}{2}}} \right) \\ &+ \left( {gZ_ - ^{n + \frac{1}{2}} + GZ_ + ^{n + \frac{1}{2}} - q\phi _{}^{n + \frac{1}{2}}} \right)\left( {\psi _ - ^{n + \frac{1}{2}} - \bar \psi _ - ^{n + \frac{1}{2}}} \right).
		\end{aligned}
	\end{align}    
	We define the following notations
	\begin{align*}
		&e_{\psi , + }^{n + 1} = \psi _ + ^{n + 1} - \psi _{+,h }^{n + 1},\quad e_{Z, + }^{n + \frac{1}{2}} = Z_ + ^{n + \frac{1}{2}} - Z_{+,h }^{n + \frac{1}{2}},\quad e_\phi ^{n + \frac{1}{2}} = \phi ^{n + \frac{1}{2}} - \phi _{h }^{n + \frac{1}{2}}, \\ 
		&e_{\psi , - }^{n + 1} = \psi _ - ^{n + 1} - \psi _{-,h }^{n + 1},\quad e_{Z, - }^{n + \frac{1}{2}} = Z_ - ^{n + \frac{1}{2}} - Z_{-,h }^{n + \frac{1}{2}}.
	\end{align*}
	By subtracting the fully discrete scheme \eqref{FullDS} from \eqref{HalfDSwithTylor}, we obtain
	\begin{subequations}
		{\begin{align}
				&{\left\langle {\mathbf{i}{D_\tau }e_{\psi , + }^{n + 1},v_{+,h}} \right\rangle = \frac{1}{2}{A_0}\left( {\bar e_{\psi , + }^{n + \frac{1}{2}},v_{+,h}} \right) + \left\langle {J_{1, + }^{n + 1},v_{+,h}} \right\rangle + \left\langle {R_{1, + }^{n + 1},v_{+,h}} \right\rangle}, \label{ErrorEq1}\\
				&{\left\langle {\mathbf{i}{D_\tau }e_{\psi , - }^{n + 1},v_{-,h}} \right\rangle = \frac{1}{2}{A_0}\left( {\bar e_{\psi , - }^{n + \frac{1}{2}},v_{-,h}} \right) + \left\langle {J_{1, - }^{n + 1},v_{-,h}} \right\rangle + \left\langle {R_{1, - }^{n + 1},v_{-,h}} \right\rangle}, \label{ErrorEq2}\\
				&{{A_1}\left( {e_\phi ^{n + \frac{1}{2}},{w_h}} \right) = \left( {4\pi qe_{Z, + }^{n + \frac{1}{2}},{w_h}} \right) - \left( {4\pi qe_{Z, - }^{n + \frac{1}{2}},{w_h}} \right)}, \label{ErrorEq3}\\
				&{\left( {e_{Z, + }^{n + \frac{1}{2}} + e_{Z, + }^{n - \frac{1}{2}},{\chi _{+,h }}} \right) = \left( {S_{1, + }^n,{\chi _{+,h }}} \right) + \left( {T_{1, + }^n,{\chi _{+,h }}} \right)}, \label{ErrorEq4}\\
				&{\left( {e_{Z, - }^{n + \frac{1}{2}} + e_{Z, - }^{n - \frac{1}{2}},{\chi _{-,h }}} \right) = \left( {S_{1, - }^n,{\chi _{-,h }}} \right) + \left( {T_{1, - }^n,{\chi _{-,h }}} \right)}, \label{ErrorEq5}
		\end{align}}\label{ErrorEq}
	\end{subequations}
	where 
	\begin{align}
		&T_{1, + }^n = 2{\left| {\psi _ + ^n} \right|^2} - 2{\left| {\psi _{+,h }^n} \right|^2},\label{defT+} \\
		&T_{1, - }^n = 2{\left| {\psi _ - ^n} \right|^2} - 2{\left| {\psi _{-,h }^n} \right|^2},\\
		&J_{1, + }^{n + 1} = \left( {gZ_ + ^{n + \frac{1}{2}} + GZ_ - ^{n + \frac{1}{2}} + q\phi _{}^{n + \frac{1}{2}}} \right)\bar \psi _ + ^{n + \frac{1}{2}} - \left( {gZ_{+,h }^{n + \frac{1}{2}} + GZ_{-,h }^{n + \frac{1}{2}} + q\phi _h^{n + \frac{1}{2}}} \right)\bar \psi _{+,h }^{n + \frac{1}{2}},\\
		&J_{1, - }^{n + 1} = \left( {gZ_ - ^{n + \frac{1}{2}} + GZ_ + ^{n + \frac{1}{2}} - q\phi _{}^{n + \frac{1}{2}}} \right)\bar \psi _ - ^{n + \frac{1}{2}} - \left( {gZ_{-,h }^{n + \frac{1}{2}} + GZ_{+,h }^{n + \frac{1}{2}} - q\phi _h^{n + \frac{1}{2}}} \right)\bar \psi _{-,h }^{n + \frac{1}{2}}.
	\end{align}
	We further split the errors as
	\begin{align}
		&e_{\psi , \pm }^{n + 1} = \left( {\psi _ \pm ^{n + 1} - {R_h}\psi _ \pm ^{n + 1}} \right) + \left( {{R_h}\psi _ \pm ^{n + 1} - \psi _{h, \pm }^{n + 1}} \right) := \xi _{\psi , \pm }^{n + 1} + \eta _{\psi , \pm }^{n + 1},\\
		&e_{Z, \pm }^{n + \frac{1}{2}} = \left( {Z_ \pm ^{n + \frac{1}{2}} - {R_h}Z_ \pm ^{n + \frac{1}{2}}} \right) + \left( {{R_h}Z_ \pm ^{n + \frac{1}{2}} - Z_{h, \pm }^{n + \frac{1}{2}}} \right) := \xi _{Z, \pm }^{n + \frac{1}{2}} + \eta _{Z, \pm }^{n + \frac{1}{2}},\\
		&e_\phi ^{n + \frac{1}{2}} = \left( {\phi ^{n + \frac{1}{2}} - {R_h}\phi ^{n + \frac{1}{2}}} \right) + \left( {{R_h}\phi ^{n + \frac{1}{2}} - \phi _{h}^{n + \frac{1}{2}}} \right) := \xi _\phi ^{n + \frac{1}{2}} + \eta _\phi ^{n + \frac{1}{2}}.
	\end{align}
	By the decomposition of errors and the Ritz projection, we write \eqref{ErrorEq} as follows:
	\begin{subequations}
		{\begin{align}
				&{\left\langle {\mathbf{i}{D_\tau }\eta _{\psi , + }^{n + 1},v_{+,h}} \right\rangle = \frac{1}{2}{A_0}\left( {\bar \eta _{\psi , + }^{n + \frac{1}{2}},v_{+,h}} \right) + \left\langle {J_{1, + }^{n + 1},v_{+,h}} \right\rangle + \left\langle {R_{2, + }^{n + 1},v_{+,h}} \right\rangle}, \label{deErrorEq1} \\
				&{\left\langle {\mathbf{i}{D_\tau }\eta _{\psi , - }^{n + 1},v_{-,h}} \right\rangle = \frac{1}{2}{A_0}\left( {\bar \eta _{\psi , - }^{n + \frac{1}{2}},v_{-,h}} \right) + \left\langle {J_{1, - }^{n + 1},v_{-,h}} \right\rangle + \left\langle {R_{2, - }^{n + 1},v_{-,h}} \right\rangle}, \label{deErrorEq2} \\
				&{{A_1}\left( {\eta _\phi ^{n + \frac{1}{2}},{w_h}} \right) = \left( {4\pi q\eta _{Z, + }^{n + \frac{1}{2}},{w_h}} \right) - \left( {4\pi q\eta _{Z, - }^{n + \frac{1}{2}},{w_h}} \right) + \left( {{R_3},{w_h}} \right)}, \label{deErrorEq3} \\
				&{\left( {\eta _{Z, + }^{n + \frac{1}{2}} + \eta _{Z, + }^{n - \frac{1}{2}},{\chi _{+,h }}} \right) = \left( {S_{2, + }^n,{\chi _{+,h }}} \right) + \left( {T_{1, + }^n,{\chi _{+,h }}} \right)}, \label{deErrorEq4} \\
				&{\left( {\eta _{Z, - }^{n + \frac{1}{2}} + \eta _{Z, - }^{n - \frac{1}{2}},{\chi _{-,h }}} \right) = \left( {S_{2, - }^n,{\chi _{-,h }}} \right) + \left( {T_{1, - }^n,{\chi _{-,h }}} \right)}, \label{deErrorEq5}
		\end{align}} \label{deErrorEq}
	\end{subequations}
	where
	\begin{equation*}
		S_{2, \pm }^n = S_{1, \pm }^n - \left( {\xi _{Z, \pm }^{n + \frac{1}{2}} + \xi _{Z, \pm }^{n - \frac{1}{2}}} \right),   
		\quad   R_{2, \pm }^{n + 1} = R_{1, \pm }^{n + 1} - \mathbf{i}{D_\tau }\xi _{\psi , \pm }^{n + 1},
		\quad   R_3^{n + \frac{1}{2}} = 4\pi q\left( {\xi _{Z, + }^{n + \frac{1}{2}} - \xi _{Z, - }^{n + \frac{1}{2}}} \right). 
	\end{equation*}
	By applying the Taylor expansion and the error estimate for interpolation, it holds 
	\begin{align}
		&S_{2, + }^n \leq C(\tau^2 + h^{k+1}), \label{S2+}\\
		&S_{2, - }^n \leq C(\tau^2 + h^{k+1}), \label{S2-}\\    
		&R_{2, + }^{n + 1} \leq C(\tau^2 + h^{k+1}),\label{R2+} \\ 
		&R_{2, - }^{n + 1} \leq C(\tau^2 + h^{k+1}),\label{R2-} \\
		&R_3^{n + \frac{1}{2}} \leq Ch^{k+1}.
	\end{align}
	
	\begin{theorem}\label{TH error estimate}
		Suppose that $\psi_\pm$, $Z_\pm$, and $\phi$ satisfy the regularity assumptions stated in \eqref{regularity assumption}. Assume that $\tau \le C h$. Then there exist positive constants $\tau_0$ and $h_0$ such that, for $\tau \le \tau_0$ and $h \le h_0$, the solution of the fully discrete scheme \eqref{FullDS} satisfies
		\begin{align}
			&\mathop {\max }\limits_{0 \le n \le N} \left\| {e_{\psi , \pm }^{n}} \right\|  \le C\left( {{\tau ^2} + {h^{k + 1}}} \right), \label{TH3.7 psi}\\
			&\mathop {\max }\limits_{0 \le n \le N-1} \left\| {e_{Z, \pm }^{n + \frac{1}{2}}} \right\|  \le C\left( {{\tau ^2} + {h^{k + 1}}} \right),\label{TH3.7 Z}\\
			&\mathop {\max }\limits_{0 \le n \le N-1} \left\| {e_\phi ^{n + \frac{1}{2}}} \right\| \le C\left( {{\tau ^2} + {h^{k + 1}}} \right).\label{TH3.7 phi}
		\end{align}
	\end{theorem}
	The proof of \Cref{TH error estimate} will be presented in \Cref{Appendix}.
	
	\begin{remark}
		For the coupled GPE \eqref{coulped NLS equation}, the proposed relaxation Crank–Nicolson finite element method reduces to find $(\psi_{+,h}^{n+1},\psi_{-,h}^{n+1},Z_{+,h}^{n+\frac{1}{2}},Z_{-,h}^{n+\frac{1}{2}}) \in \mathbf{V}_{h}\times \mathbf{V}_{h}\times V_{h}\times V_{h}$ such that 
		\begin{subequations}\label{FullDS-}
			\begin{align}
				&\left\langle \mathbf{i} D_{\tau} \psi_{+,h}^{n+1}, v_{+,h} \right\rangle =\frac{1}{2}A_0 \left( \bar\psi_{+,h}^{n+\frac{1}{2}}, v_{+,h} \right)  + \left\langle \left(g Z_{+,h}^{n+\frac{1}{2}} +  G Z_{-,h}^{n+\frac{1}{2}}  \right) \bar  \psi_{+,h}^{n+\frac{1}{2}} , v_{+,h} \right\rangle,  \label{FullDS1-} \\
				&\left\langle \mathbf{i} D_{\tau} \psi_{-,h}^{n+1}, v_{-,h} \right\rangle = \frac{1}{2}A_0 \left(  \bar\psi_{-,h}^{n+\frac{1}{2}}, v_{-,h} \right) + \left\langle \left( g Z_{-,h}^{n+\frac{1}{2}} + G Z_{+,h}^{n+\frac{1}{2}}  \right) \bar \psi_{-,h}^{n+\frac{1}{2}} , v_{-,h} \right\rangle,  \label{FullDS2-}\\ 
				&\left( Z_{+,h}^{n+\frac{1}{2}} + Z_{+,h}^{n-\frac{1}{2}} , \chi_{+,h} \right) = \left( 2 \vert {\psi_{+,h}^n} \vert ^2, \chi_{+,h} \right),  \label{FullDS4-} \\
				&\left( Z_{-,h}^{n+\frac{1}{2}} + Z_{-,h}^{n-\frac{1}{2}} , \chi_{-,h} \right) = \left( 2 \vert {\psi_{-,h}^n} \vert ^2, \chi_{-,h} \right),  \label{FullDS5-}
			\end{align}
		\end{subequations}
		for any $(v_{+,h},v_{-,h},\chi_{+,h},\chi_{-,h}) \in \mathbf{V}_{h}\times \mathbf{V}_{h}\times V_{h}\times V_{h}$.
		
		Algorithm \ref{alg:CNFEM} works correctly only after removing line 4. Furthermore, the results in \Cref{lem:ex}, \eqref{lemma: Des Mass and Energy}, and \Cref{TH error estimate} remain valid if the approximation of $q$ is omitted.
	\end{remark}

	\section{Numerical experiments}\label{secNmecical}
	In this section, we present numerical experiments to validate the theoretical results of the proposed method. The effectiveness of the method will be evaluated through its optimal convergence rates and its ability to preserve mass and energy. Moreover, the well-known instability phenomenon will also be observed under certain parameter settings.
	
	\begin{exmp}
		(A one-dimensional density wave) We consider the one-dimensional GPP equation on $\Omega = [0,L]$ with parameters $q=1$, $G=2$, $g=1$, and 
		$$l_0 = \frac{\sqrt{2\pi (G-g)}}{q},\quad L=8l_0.$$
		The problem is subject to the initial conditions
		\[
		\psi_{+}(x,0) = U_0 \cos\!\left(\frac{2\pi x}{l_0}\right),
		\qquad
		\psi_{-}(x,0) = U_0 \sin\!\left(\frac{2\pi x}{l_0}\right),
		\]
		and periodic boundary conditions
		\[
		\psi_\pm(x,t) = \psi_\pm(x+L,t),
		\qquad t \in [0,T].
		\]
		The analytical solution is given by 
		\begin{align*}
			&\psi _ +=U_0\exp(-\mathbf{i}\mu t) \text{cos}(\frac{2 \pi x}{l_0}),\\
			&\psi _ -=U_0\exp(-\mathbf{i}\mu t) \text{sin}(\frac{2 \pi x}{l_0}),\\
			&\phi = \Phi_0 \text{cos}(\frac{4 \pi x}{l_0}),
		\end{align*}
		where $U_0 = 2\sqrt{5},\;\Phi_0 = \frac{(G-g)U_0^2}{2q},\;\mu = \frac{2\pi^2}{l_0^2}+\frac{1}{2}(G+g)U_0^2$.

		The spatial discretization errors based on the $P^k$ ($k=1,2$) elements are reported in \Cref{1DP1 space} and \Cref{1DP2 space}, respectively. To eliminate the influence of temporal discretization errors, we set $\tau = 1 \times 10^{-4}$ and $T = 1 \times 10^{-2}$. The results show that the spatial errors exhibit $(k+1)$-th order convergence, in good agreement with the theoretical predictions.
		The temporal discretization errors based on $P^2$ element are presented in \Cref{1DP2 time}, where $T = 1 \times 10^{-1}$. To suppress the influence of spatial discretization errors, we take $h = L/8000$. The results demonstrate that the temporal errors also achieve second-order convergence, again consistent with the theoretical analysis.
		The conservation of masses and energy is examined in Figure~\ref{conserve1D}, with $\tau = 1 \times 10^{-3}$, $h = L/1000$, and $T = 5$, using $P^2$ finite elements. As illustrated in the figure, the changes in masses and energy remain at the level of round-off accuracy, confirming the conservation properties of the proposed method. 
		
		The instability of the one-dimensional density wave has been reported in \cite{GPandGPPantoine_computational_2013} and is an intrinsic feature of the system. Similar behavior is observed in our simulations. Figure~\ref{ex1D_psi+ 1} shows snapshots of $\psi_+$ at $t = 0$, $0.8$, and $2$ with $\tau = 1 \times 10^{-3}$ and $h = L/1000$. The discrepancy between the numerical and exact solutions becomes increasingly pronounced after $t = 2.5$, eventually leading to an unstable state, as depicted in Figure~\ref{ex1D_psi+ 2}. A similar phenomenon is observed for $\psi_-$ and $\phi$.
		
		\begin{table}[H]
			\centering
			\begin{tabular}{ccccccc}
				\hline
				& $e_{\psi,+}$ & rate &
				$e_{\psi,-}$ & rate &
				$e_{\phi}$ & rate  \\
				\hline    
				$h = L/100$ & 3.95E-01 & --
				& 3.95E-01
				& -- & 5.00E-00 & --
				\\
				$h = L/200$ & 1.00E-01 & 1.98
				& 1.00E-01
				& 1.98 & 1.33E-00 & 1.91
				\\
				$h = L/400$ & 2.52E-02 & 1.99
				& 2.52E-02
				& 1.99 &3.37E-01 & 1.98 \\
				$h = L/800$ & 6.30E-03 & 2.00
				& 6.30E-03
				& 2.00 & 8.48E-02 & 1.99\\
				\hline    
			\end{tabular}
			\caption{Spatial discretization errors of $P^1$ element.}
			\label{1DP1 space}
		\end{table}
		
		\begin{table}[H]
			\centering
			\begin{tabular}{ccccccc}
				\hline
				& $e_{\psi,+}$ & rate &
				$e_{\psi,-}$ & rate &
				$e_{\phi}$ & rate  \\
				\hline    
				$h = L/100$ & 1.04E-02 & --
				& 1.04E-02
				& -- & 1.93E-01 & --
				\\
				$h = L/200$ & 1.29E-03 & 3.01
				& 1.29E-03
				& 3.01 & 2.34E-02 & 3.04
				\\
				$h = L/400$ & 1.62E-04 & 3.00
				& 1.62E-04
				& 3.00 & 2.90E-03 & 3.01 \\
				$h = L/800$ & 2.07E-05 & 2.97
				& 2.07E-05
				& 2.97 & 3.62E-04 & 3.00\\
				\hline    
			\end{tabular}
			\caption{Spatial discretization errors of $P^2$ element.}
			\label{1DP2 space}
		\end{table}
		
		
		\begin{table}[H]
			\centering
			\begin{tabular}{lcccccc}
				\hline
				& $e_{\psi,+}$ & rate &
				$e_{\psi,-}$ & rate \\
				\hline    
				$\tau = 5 \times 10^{-3}$ & 1.07E-01 & --
				& 1.07E-01
				& -- 
				\\
				$\tau = 2.5 \times 10^{-3}$ & 2.68E-02 & 2.00
				& 2.68E-02
				& 2.00 
				\\
				$\tau = 1.25 \times 10^{-3}$ & 6.71E-03 & 2.00
				& 6.71E-03
				& 2.00 
				\\
				$\tau = 6.25 \times 10^{-4}$ & 1.68E-03 & 2.00
				& 1.68E-03
				& 2.00 \\
				\hline    
			\end{tabular}
			\caption{Temporal discretization errors of $P^2$ element.}
			\label{1DP2 time}
		\end{table}
		
		
		\begin{figure}[!htbp]
			\centering
			$\begin{array}{c}
				\includegraphics[width=5.5cm,height=4.5cm]{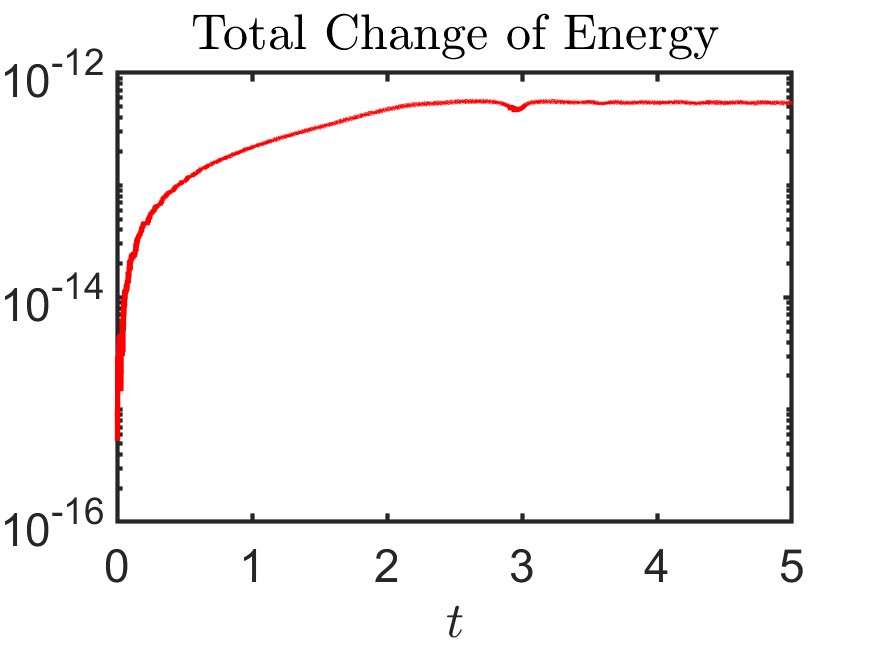} \; 
				\includegraphics[width=5.5cm,height=4.5cm]{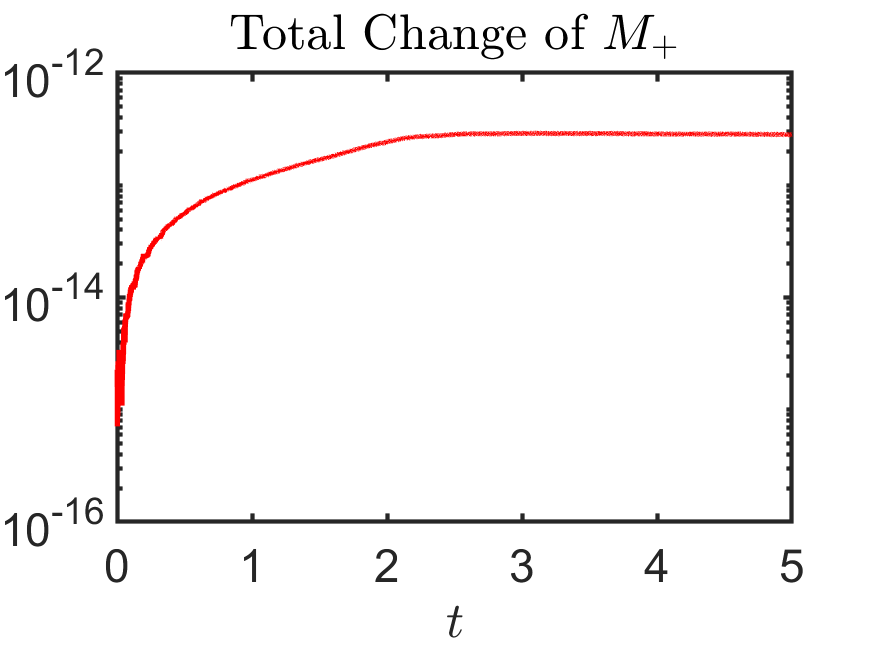} \; \includegraphics[width=5.5cm,height=4.5cm]{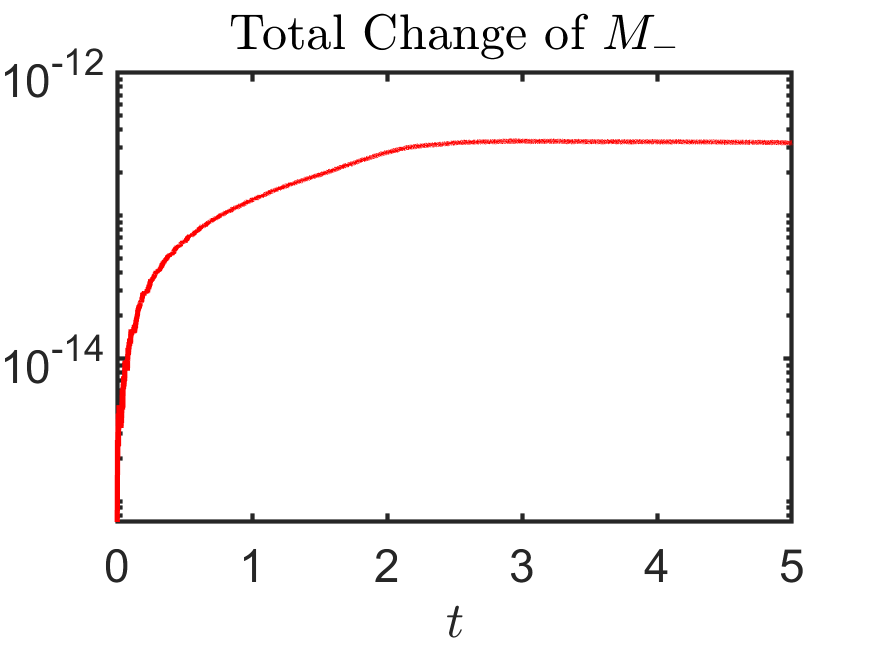}
			\end{array}$\vspace{-0.2cm}
			\caption{Total changes of the energy and masses.} \label{conserve1D}
		\end{figure}
		
		\begin{figure}[!htbp]
			$\begin{array}{c}
				\includegraphics[width=5.5cm,height=4.5cm]{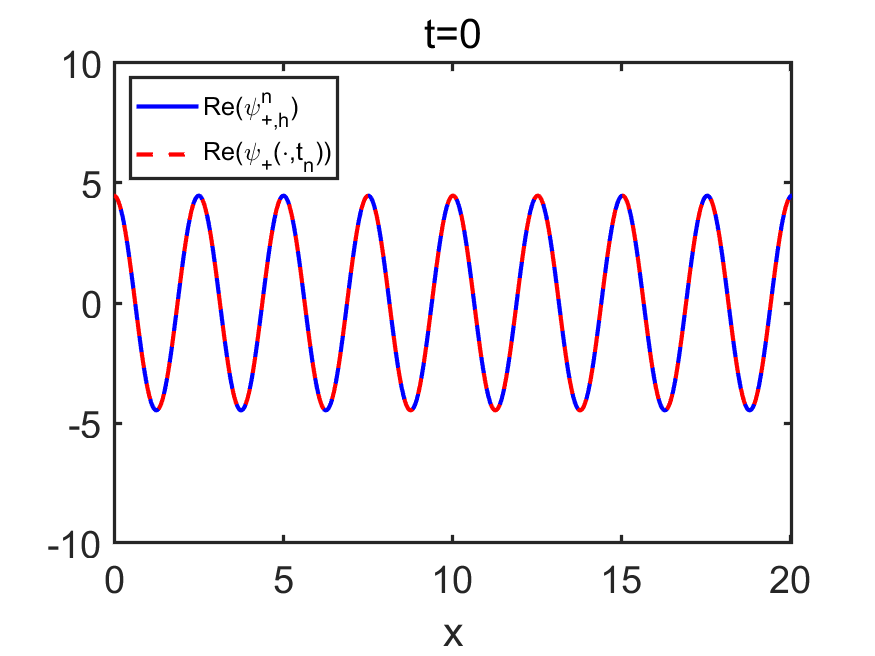} \;
				\includegraphics[width=5.5cm,height=4.5cm]{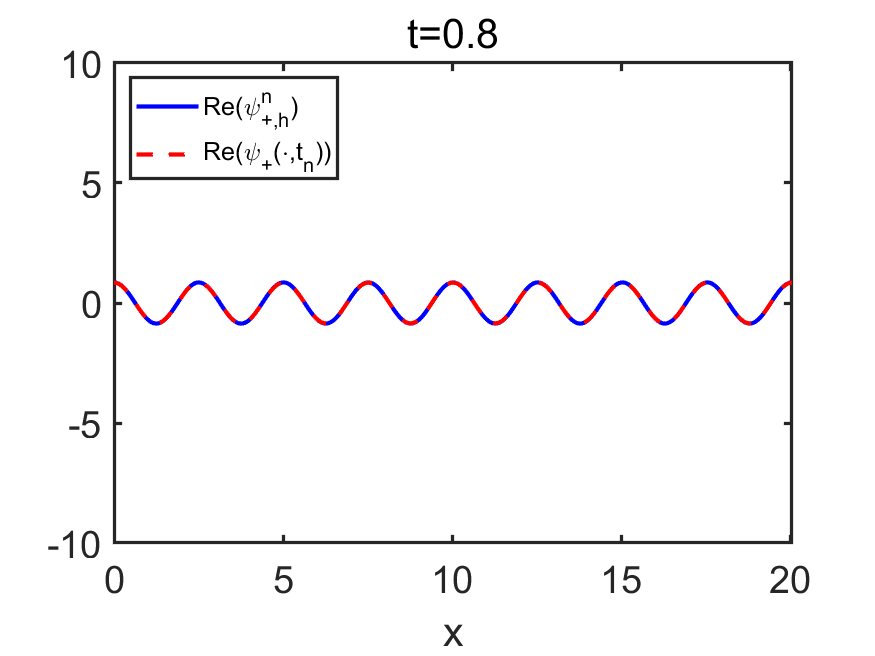} \;
				\includegraphics[width=5.5cm,height=4.5cm]{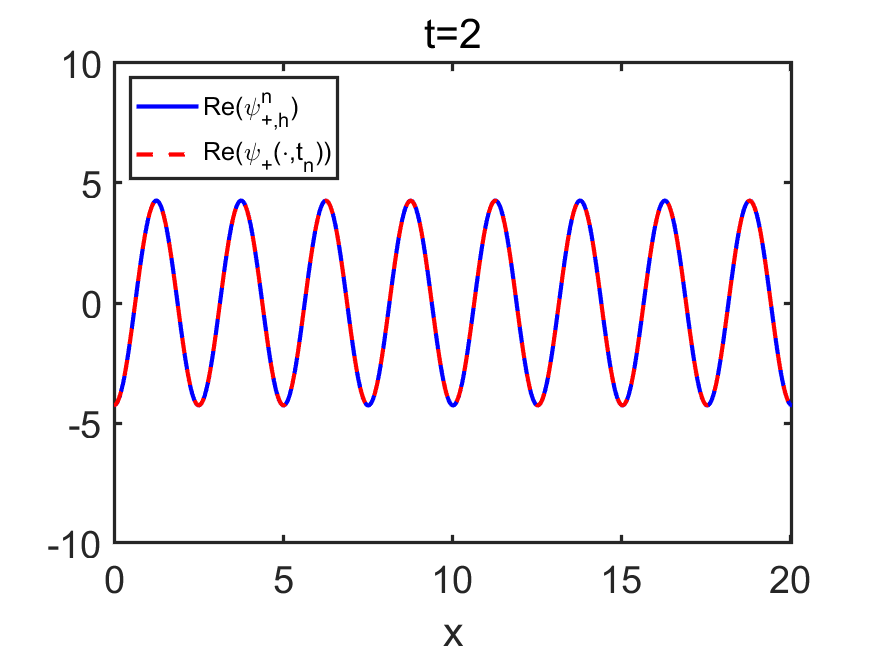} 
			\end{array}$\vspace{-0.2cm}
			\caption{Patterns before $t \le 2$ of Re$(\psi_{+,h})$.} 
			\label{ex1D_psi+ 1}
		\end{figure}
		
		\begin{figure}[!htbp]
			$\begin{array}{c}
				\includegraphics[width=5.5cm,height=4.5cm]{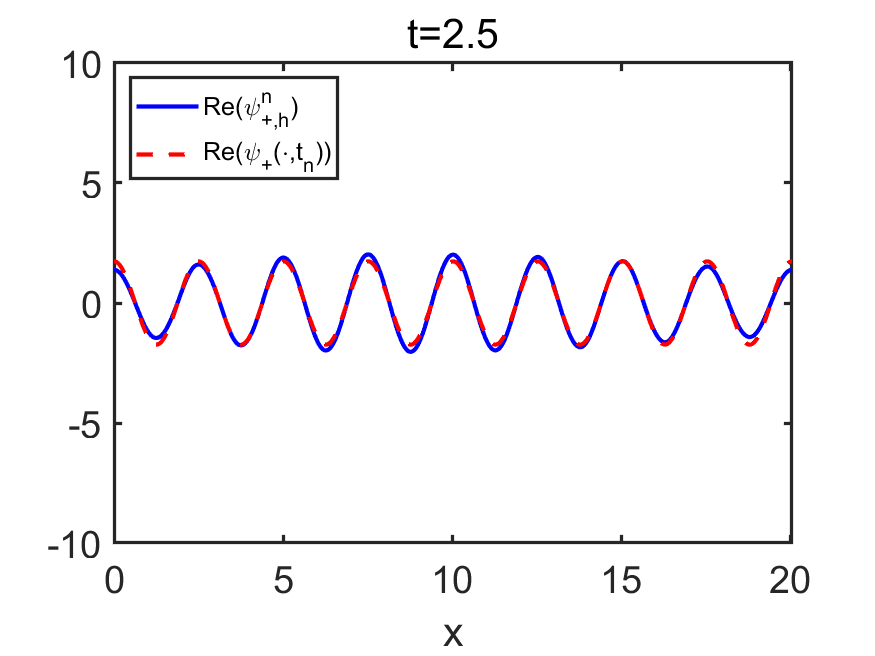}\; 
				\includegraphics[width=5.5cm,height=4.5cm]{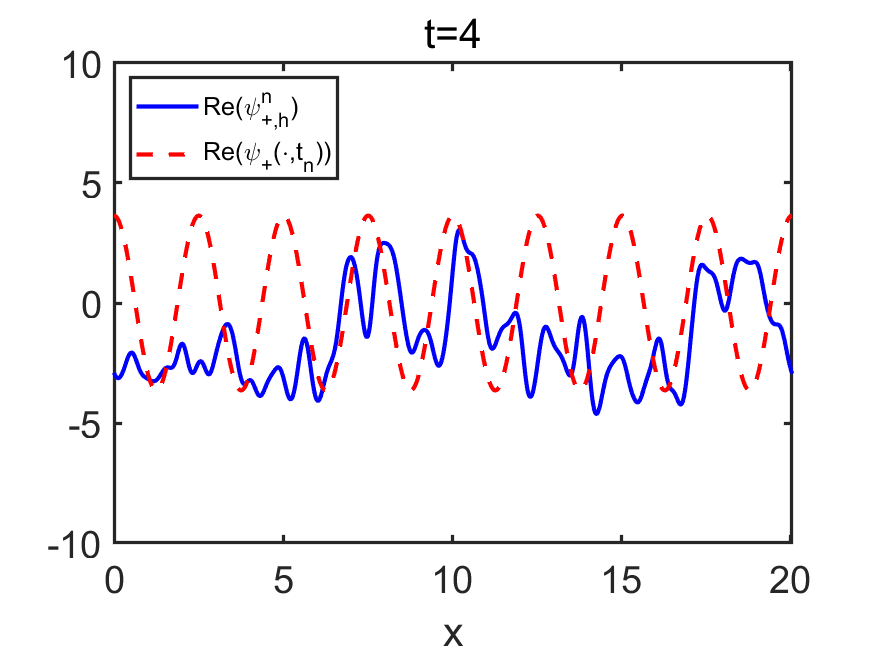}\;
				\includegraphics[width=5.5cm,height=4.5cm]{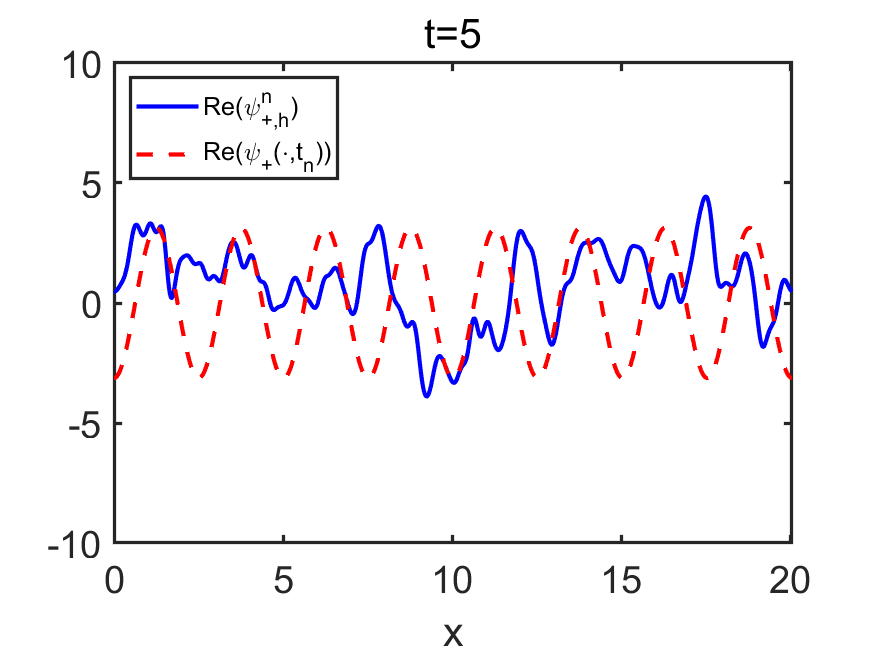}
			\end{array}$\vspace{-0.2cm}
			\caption{The instability of $\psi_{+,h}$.} 
			\label{ex1D_psi+ 2}
		\end{figure}
	\end{exmp}
	
	\begin{exmp}
		(A two-dimensional continuous wave of Gross-Pitaevskii equations) \cite{Handbook}
		Consider the Gross-Pitaevskii (GP) equations, corresponding to the case $q = 0$ in \eqref{TargetEq}. In this setting, the system \eqref{TargetEq} reduces to
		\begin{subequations}\label{TargetEqGP}
			\begin{align}
				&\mathbf{i}\,\partial_t \psi_+ = \left[ -\frac{1}{2} \nabla^2 + \left( g \, |\psi_+|^2 + G \, |\psi_-|^2 \right) \right] \psi_+, 
				& \text{in } \Omega \times (0,T], \label{TargetEq1GP} \\
				&\mathbf{i}\,\partial_t \psi_- = \left[ -\frac{1}{2} \nabla^2 + \left( g \, |\psi_-|^2 + G \, |\psi_+|^2 \right) \right] \psi_-, 
				& \text{in } \Omega \times (0,T], \label{TargetEq2GP} \\
				&\psi_\pm(x,y,0) = \psi_{\pm,0}(x,y), & \text{in } \Omega,
			\end{align}
		\end{subequations} 
		where $\Omega = [0,L]^2$ with $L=5$.
		
		Given the initial conditions
		\[
		\psi_{+,0} = A_0 \exp\!\big(\mathbf{i}( A_2 (x + y) + X_1 )\big), 
		\quad 
		\psi_{-,0} = B_0 \exp\!\big(\mathbf{i}( B_2 (x + y) + X_2 )\big),
		\]
		together with periodic boundary conditions for $\psi_{\pm}$, the exact solution takes the form
		\[
		\psi_{+}(x,y,t) = A_0 \exp\!\big(\mathbf{i}( A_1 t + A_2 (x + y) + X_1 )\big), 
		\quad 
		\psi_{-}(x,y,t) = B_0 \exp\!\big(\mathbf{i}( B_1 t + B_2 (x + y) + X_2 )\big),
		\]
		which preserves constant amplitudes.
		
		The parameters are given by
		\[
		A_0 = B_0 = 4.5, 
		\quad 
		A_2 = B_2 = \frac{4\pi}{5},
		\]
		\[
		A_1 = -\big(A_2^2 + g A_0^2 + G B_0^2\big), 
		\quad 
		B_1 = -\big(B_2^2 + G A_0^2 + g B_0^2\big),
		\]
		\[
		X_1 = 0, 
		\quad 
		X_2 = \frac{\pi}{5},
		\]
		with the model parameters
		\(G = 2, \ g = 1.\)
		The spatial discretization errors based on the $P^k$ ($k=1,2$) elements are reported in \Cref{2DP1 space GP} and \Cref{2DP2 space GP}, respectively. To eliminate the influence of temporal discretization errors, we set $\tau = 1 \times 10^{-5}$ and $T = 1 \times 10^{-3}$. The results show that the spatial errors exhibit $(k+1)$-th order convergence, in good agreement with the theoretical predictions.
		The temporal discretization errors are presented in Table~\ref{2DP2 time GP}, with $T = 4 \times 10^{-2}$. To suppress the influence of spatial discretization errors, we take $h = L/200$. The results in Table~\ref{2DP2 time GP} demonstrate that the temporal errors achieve second-order convergence, also consistent with theoretical expectations.  
		The conservation of masses and energy is examined in Figure~\ref{ex2DGP total change}, with $\tau = 1 \times 10^{-3}$, $h = L/160$, and $T = 1$. From the result, it is evident that the changes in masses and energy remain at the level of machine precision, confirming the conservation properties of the proposed scheme.  
		
		In this example, because $A_1 = B_1$, the approximations $\psi_{+,h}$ and $\psi_{-,h}$ share a time period of $2\pi |A_1^{-1}| \approx 0.0937$. Several patterns within one time period are shown in Figures~\ref{ex2DGP psip} and \ref{ex2DGP psin}. The snapshots of $\psi_{+,h}$ and $\psi_{-,h}$ exhibit a phase difference of $X_2 - X_1$, and their patterns propagate in the northeast direction over time at speeds $A_2$ and $B_2$, respectively.  
		The computed convergence orders and the energy-mass transformation plots demonstrate that the proposed numerical scheme remains accurate and robust for the couple GPE ($q = 0$ in \eqref{TargetEq}).
		
		\begin{table}[H]
			\centering
			\begin{tabular}{ccccc}
				\hline
				& $e_{\psi,+}$ & rate & $e_{\psi,-}$ & rate \\
				\hline    
				$h = L/20$  & 2.36E+00 & --   & 2.36E+00 & --   \\
				$h = L/40$  & 6.06E-01 & 1.96 & 6.06E-01 & 1.96 \\
				$h = L/80$  & 1.52E-01 & 1.99 & 1.52E-01 & 1.99 \\
				$h = L/160$ & 3.82E-02 & 2.00 & 3.82E-02 & 2.00 \\
				\hline    
			\end{tabular}
			\caption{Spatial discretization errors of $P^1$ element.}
			\label{2DP1 space GP}
		\end{table}
		
		\begin{table}[H]
			\centering
			\begin{tabular}{cccccc}
				\hline
				& $e_{\psi,+}$ & rate & 
				$e_{\psi,-}$ & rate \\
				\hline    
				$h = L/20$ & 1.37E-01 & --
				& 1.37E-01
				& --
				\\
				$h = L/40$ & 1.71E-02 & 3.00
				& 1.71E-02
				& 3.00
				\\
				$h = L/80$ & 2.18E-03 & 2.97
				& 2.18E-03
				& 2.97\\
				$h = L/160$ & 2.73E-04 & 2.99
				& 2.73E-04
				& 2.99 \\
				\hline    
			\end{tabular}
			\caption{Spatial discretization errors of $P^2$ element.}
			\label{2DP2 space GP}
		\end{table}
		
		\begin{table}[H]
			\centering
			\begin{tabular}{cccccc}
				\hline
				& $e_{\psi,+}$ & rate & 
				$e_{\psi,-}$ & rate \\
				\hline    
				$\tau = 4 \times 10 ^{-3}$ & 3.58E-01 & --
				& 3.58E-01
				& --
				\\
				$\tau = 2 \times 10 ^{-3}$ & 9.03E-02 & 1.99
				& 9.03E-02
				& 1.99
				\\
				$\tau = 1 \times 10 ^{-3}$ & 2.26E-02 & 2.00
				& 2.26E-02
				& 2.00\\
				$\tau = 5 \times 10 ^{-4}$ & 5.67E-03 & 2.00
				& 5.67E-03
				& 2.00\\
				\hline    
			\end{tabular}
			\caption{Temporal discretization error of $P^2$ element.}
			\label{2DP2 time GP}
		\end{table}
		\begin{figure}[!htbp]
			\centering
			$\begin{array}{c}
				\includegraphics[width=5.5cm,height=4.5cm]{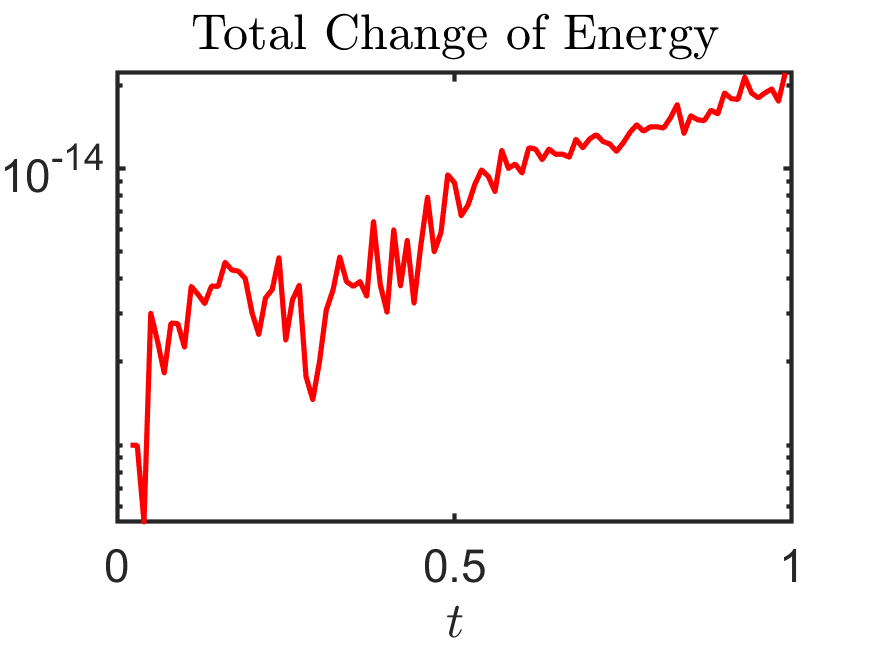} \;
				\includegraphics[width=5.5cm,height=4.5cm]{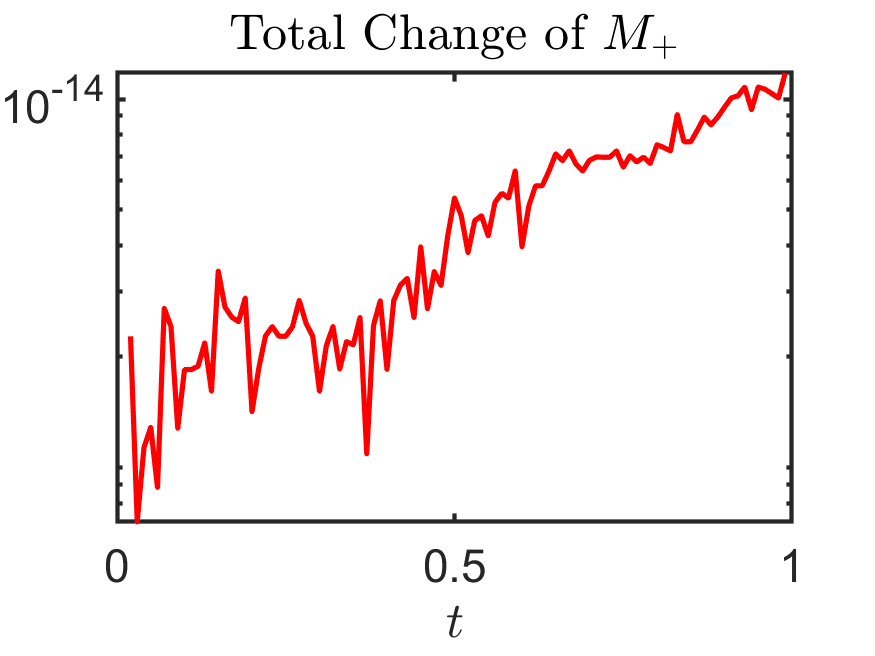} \; \includegraphics[width=5.5cm,height=4.5cm]{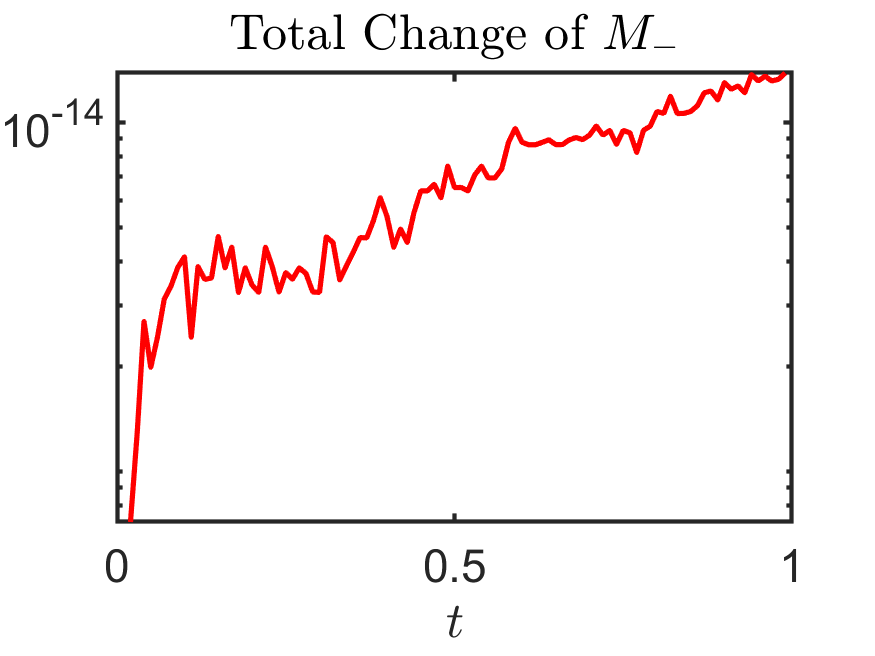}
			\end{array}$\vspace{-0.2cm}
			\caption{Total changes of the energy and masses.} \label{ex2DGP total change}
		\end{figure}
		
		\begin{figure}[!htbp]
			\centering
			$\begin{array}{c}
				\includegraphics[width=5.5cm,height=4.5cm]{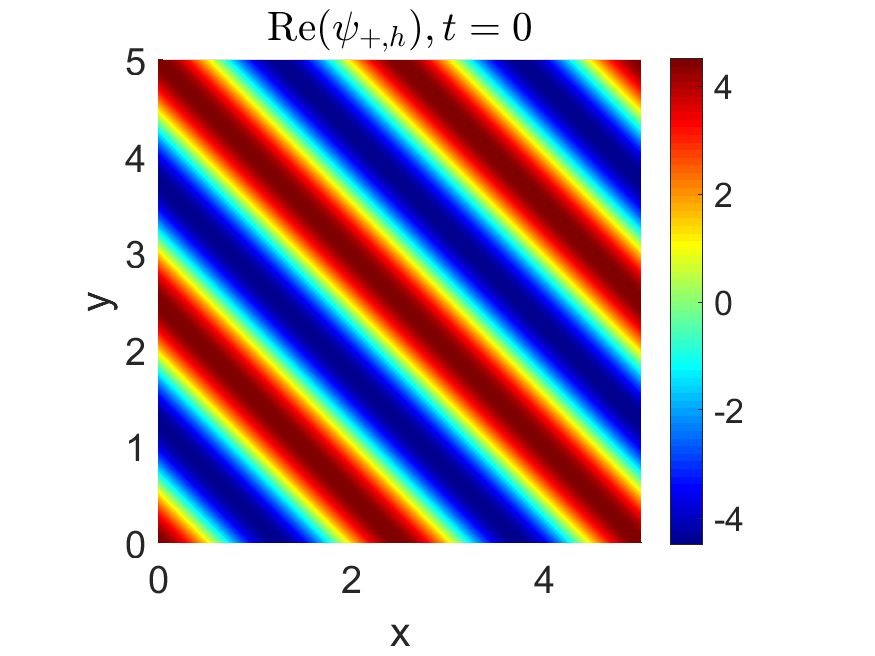} \;
				\includegraphics[width=5.5cm,height=4.5cm]{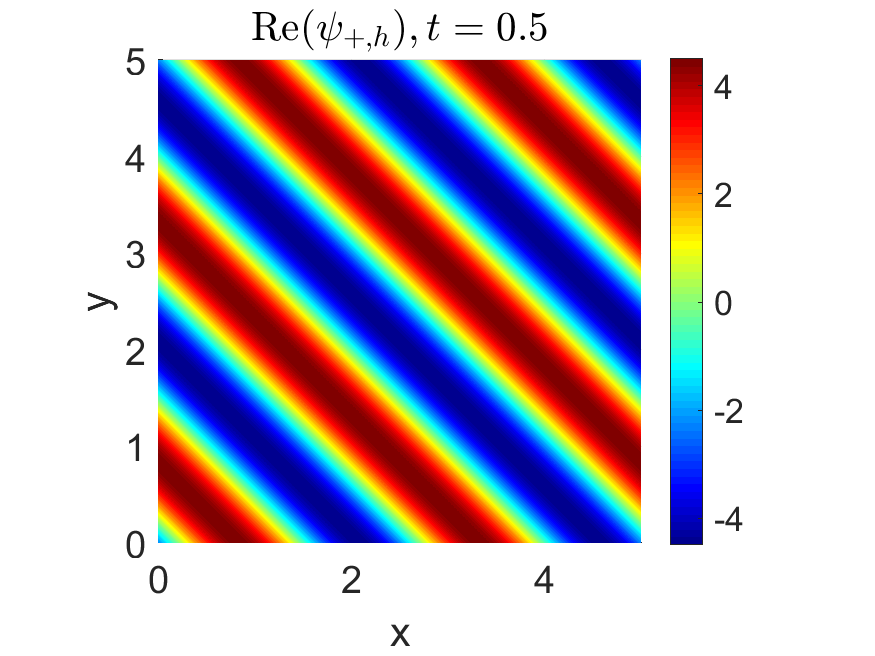} \;
				\includegraphics[width=5.5cm,height=4.5cm]{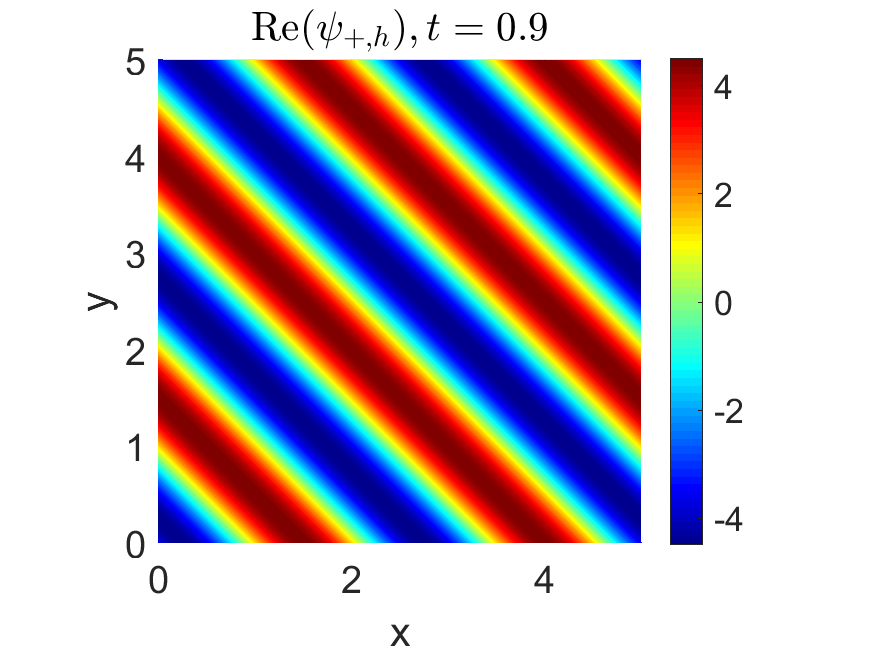} \;
			\end{array}$\vspace{-0.2cm}
			\caption{Snapshots of $\text{Re}(\psi_{+,h})$} \label{ex2DGP psip}
		\end{figure}
		
		\begin{figure}[!htbp]
			\centering
			$\begin{array}{c}
				\includegraphics[width=5.5cm,height=4.5cm]{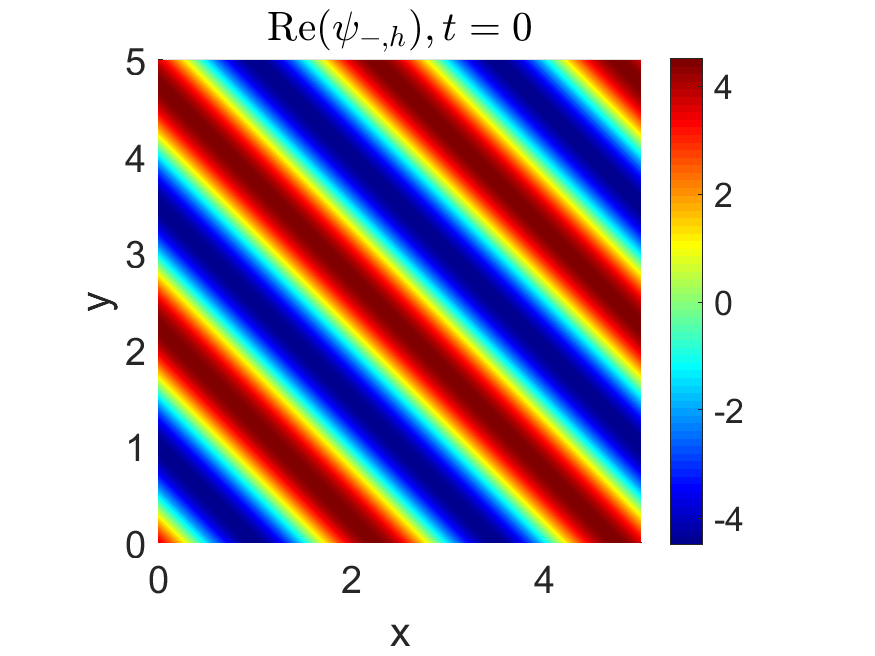} \;
				\includegraphics[width=5.5cm,height=4.5cm]{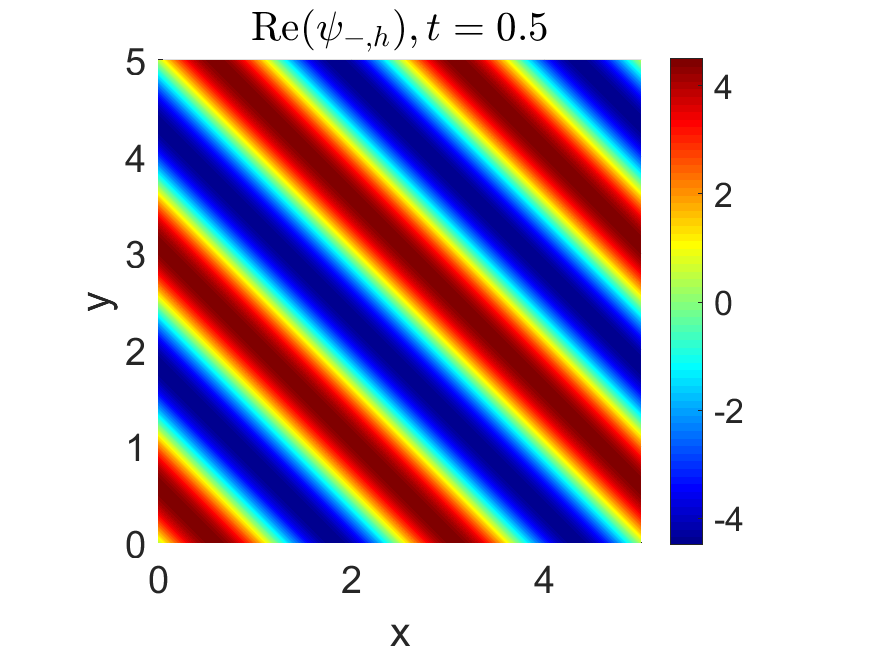} \;
				\includegraphics[width=5.5cm,height=4.5cm]{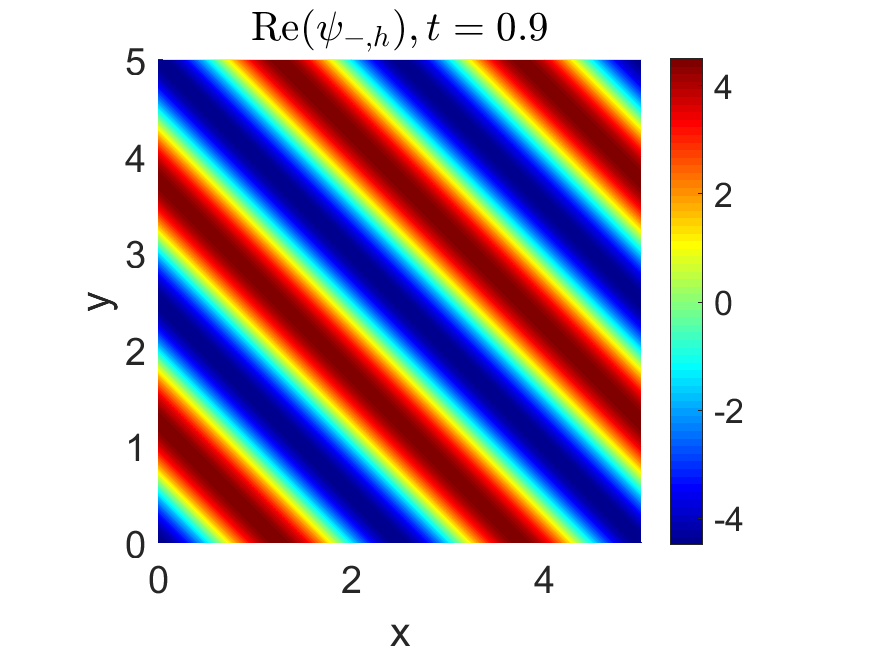} \;
			\end{array}$\vspace{-0.2cm}
			\caption{Snapshots of $\text{Re}(\psi_{-,h})$} \label{ex2DGP psin}
		\end{figure}
	\end{exmp}
	
	\begin{exmp} 
		(A two-dimensional density wave) 
		We consider the parameters 
		\[
		q = 1, \quad G = 2, \quad g = 1,
		\] 
		with initial conditions  
		\[
		\psi_+ = U_{0x} \cos\Big(\frac{2 \pi x}{l_0}\Big) + \mathbf{i} U_{0y} \cos\Big(\frac{2 \pi y}{l_0}\Big), 
		\quad 
		\psi_- = U_{0x} \sin\Big(\frac{2 \pi x}{l_0}\Big) + \mathbf{i} U_{0y} \sin\Big(\frac{2 \pi y}{l_0}\Big),
		\] 
		and periodic boundary conditions for $\psi_{\pm}$ and $\phi$.  
		
		For equation \eqref{TargetEq}, the exact solution is given by  
		\[
		\psi_+ = U_{0x} \, e^{-\mathbf{i} \mu t} \cos\Big(\frac{2 \pi x}{l_0}\Big) + \mathbf{i} U_{0y} \, e^{-\mathbf{i} \mu t} \cos\Big(\frac{2 \pi y}{l_0}\Big),
		\] 
		\[
		\psi_- = U_{0x} \, e^{-\mathbf{i} \mu t} \sin\Big(\frac{2 \pi x}{l_0}\Big) + \mathbf{i} U_{0y} \, e^{-\mathbf{i} \mu t} \sin\Big(\frac{2 \pi y}{l_0}\Big),
		\] 
		\[
		\phi = \Phi_{0x} \cos\Big(\frac{4 \pi x}{l_0}\Big) + \Phi_{0y} \cos\Big(\frac{4 \pi y}{l_0}\Big),
		\] 
		where 
		\[
		U_{0x} = U_{0y} = 2 \sqrt{5}, \quad 
		\Phi_{0x} = \Phi_{0y} = \frac{(G-g) U_{0x}^2}{2q}, \quad
		l_0 = \frac{\sqrt{2\pi (G-g)}}{q}, \quad
		\mu = \frac{2\pi^2}{l_0^2} + \frac{1}{2} (G+g)(U_{0x}^2 + U_{0y}^2),
		\] 
		and the computational domain \(\Omega = [0,L]^2\) with $L=l_0$. 
		
		The spatial discretization errors based on the $P^k$ ($k=1,2$) elements are reported in \Cref{2DP1 space with phi} and \Cref{2DP2 space}, respectively. To eliminate the influence of temporal discretization errors, we set $\tau = 1 \times 10^{-5}$ and $T = 1 \times 10^{-3}$. The results show that the spatial errors exhibit $(k+1)$-th order convergence, in good agreement with the theoretical predictions. 
		The temporal discretization errors are reported in Table~\ref{2DP2 time}, with $T = 4 \times 10^{-2}$. To suppress the influence of spatial discretization errors, we take $h = L/200$. The results in Table~\ref{2DP2 time} demonstrate that the temporal errors exhibit second-order convergence, again consistent with the theoretical analysis.  
		The conservation of masses and energy is examined in Figure~\ref{ex2D total change}, with $\tau = 1 \times 10^{-3}$, $T = 1$, and $h = L/160$. From the result, the changes in masses and energy remain at the level of machine precision, confirming the conservation properties of the proposed scheme.  
		
		Figure~\ref{ex2D psip} presents snapshots of the numerical solution $\mathrm{Re}(\psi_+)$ at $t = 0, 0.01, 0.02, 0.03, 0.04,$ and $0.05$. As indicated by the error results in Tables~\ref{2DP2 space} and \ref{2DP2 time}, the numerical solution converges to the exact solution, which is periodic in time with period $2\pi \mu^{-1} \approx 0.0995$.
		
		\begin{table}[H]
			\centering
			\begin{tabular}{ccccccc}
				\hline
				& $e_{\psi,+}$ & rate & $e_{\psi,-}$ & rate & $e_{\phi}$ & rate \\
				\hline    
				$h = L/20$  & 1.02E-01 & --   & 1.02E-01 & --   & 1.64E+00 & --   \\
				$h = L/40$  & 2.54E-02 & 2.00 & 2.54E-02 & 2.00 & 4.19E-01 & 1.97 \\
				$h = L/80$  & 6.36E-03 & 2.00 & 6.36E-03 & 2.00 & 1.05E-01 & 1.99 \\
				$h = L/160$ & 1.59E-03 & 2.00 & 1.59E-03 & 2.00 & 2.64E-02 & 2.00 \\
				\hline    
			\end{tabular}
			\caption{Spatial discretization errors of the $P^1$ element.}
			\label{2DP1 space with phi}
		\end{table}
		
		\begin{table}[H]
			\centering
			\begin{tabular}{ccccccc}
				\hline
				& $e_{\psi,+}$ & rate & 
				$e_{\psi,-}$ & rate & 
				$e_{\phi}$ & rate \\
				\hline    
				$h = L/20$ & 1.99E-03 & --
				& 1.99E-03
				& -- & 3.63E-02
				& --
				\\
				$h = L/40$ & 2.50E-04 & 3.00
				& 2.50E-04
				& 3.00 & 4.49E-03
				& 3.02
				\\
				$h = L/80$ & 3.12E-05 & 3.00
				& 3.12E-05
				& 3.00 & 2.50E-04
				& 3.00\\
				$h = L/160$ & 3.90E-06 & 3.00
				& 3.90E-06
				& 3.00 & 6.99E-05
				& 3.00 \\
				\hline    
			\end{tabular}
			\caption{Spatial discretization errors of the $P^2$ element.}
			\label{2DP2 space}
		\end{table}
		
		
		\begin{table}[H]
			\centering
			\begin{tabular}{cccccc}
				\hline
				& $e_{\psi,+}$ & rate & 
				$e_{\psi,-}$ & rate \\
				\hline    
				$\tau = 4 \times 10 ^{-3}$ & 1.49E-01 & --
				& 1.49E-01
				& --
				\\
				$\tau = 2 \times 10 ^{-3}$ & 3.75E-02 & 1.99
				& 3.75E-02
				& 1.99
				\\
				$\tau = 1 \times 10 ^{-3}$ & 9.40E-03 & 2.00
				& 9.40E-03
				& 2.00\\
				$\tau = 5 \times 10 ^{-4}$ & 2.35E-03 & 2.00
				& 2.35E-03
				& 2.00\\
				\hline    
			\end{tabular}
			\caption{Temporal discretization error with $T = 4\times 10^{-2}$}
			\label{2DP2 time}
		\end{table}
		
		\begin{figure}[!htbp]
			\centering
			$\begin{array}{c}
				\includegraphics[width=5.5cm,height=4.5cm]{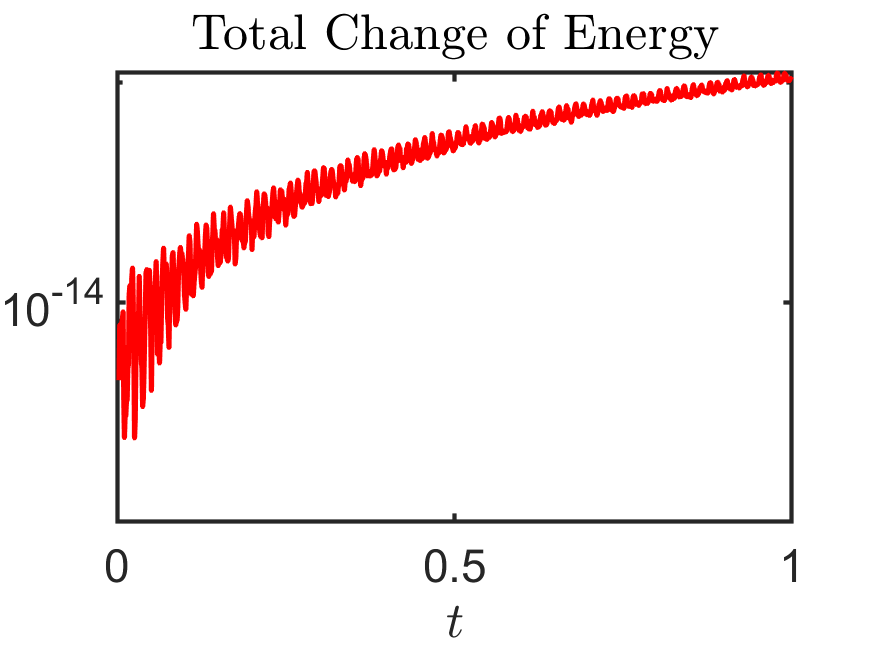} \;
				\includegraphics[width=5.5cm,height=4.5cm]{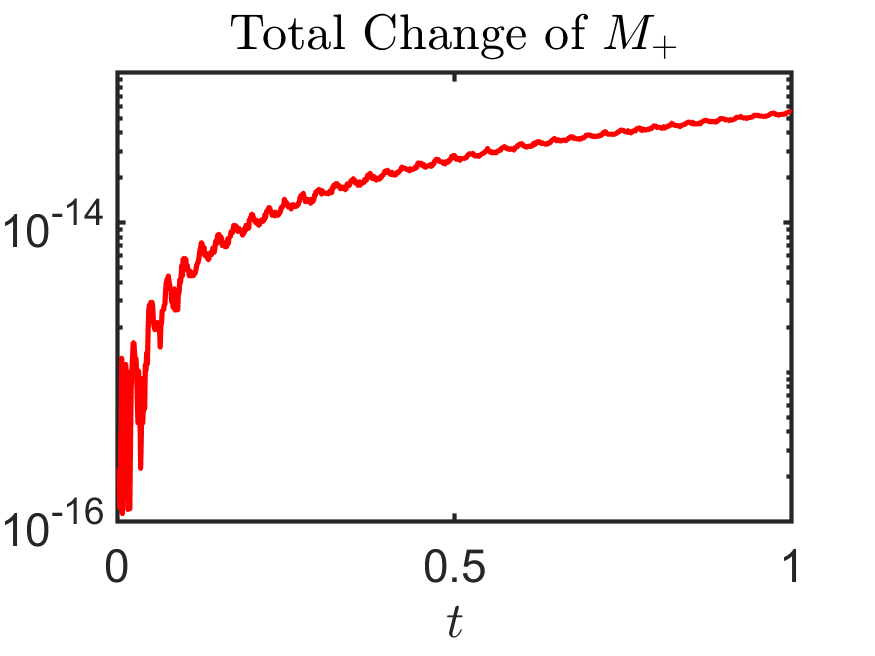} \;\includegraphics[width=5.5cm,height=4.5cm]{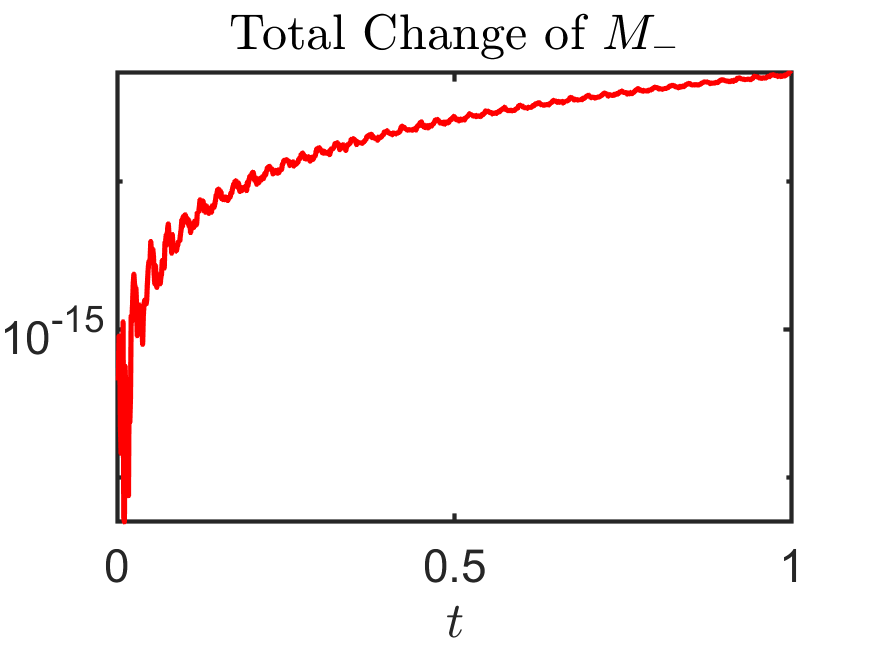} \;
			\end{array}$\vspace{-0.2cm}
			\caption{Total changes of the energy and masses.} \label{ex2D total change}
		\end{figure}
		
		\begin{figure}[!htbp]
			\centering
			$\begin{array}{c}
				\includegraphics[width=5.5cm,height=4.5cm]{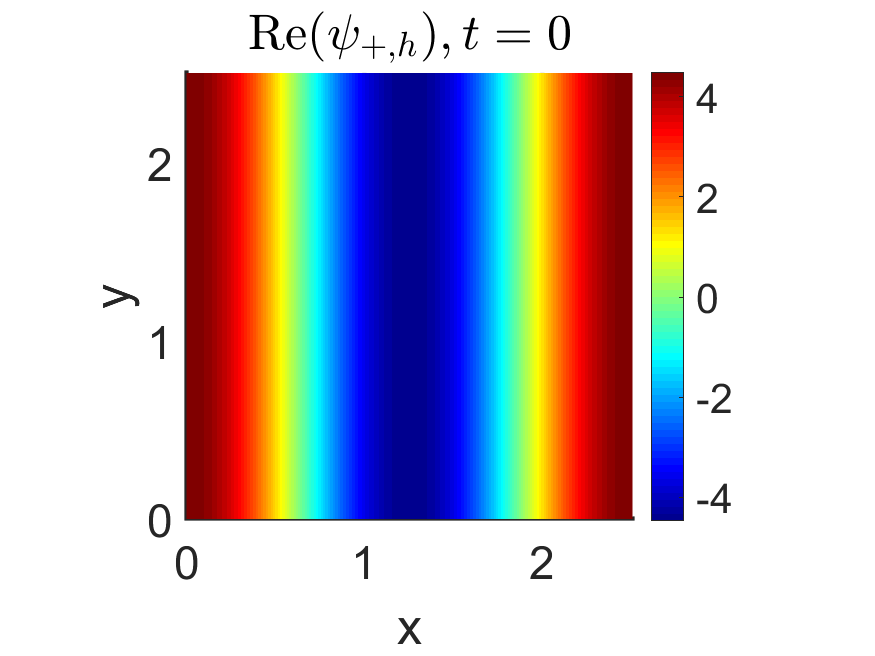} \;
				\includegraphics[width=5.5cm,height=4.5cm]{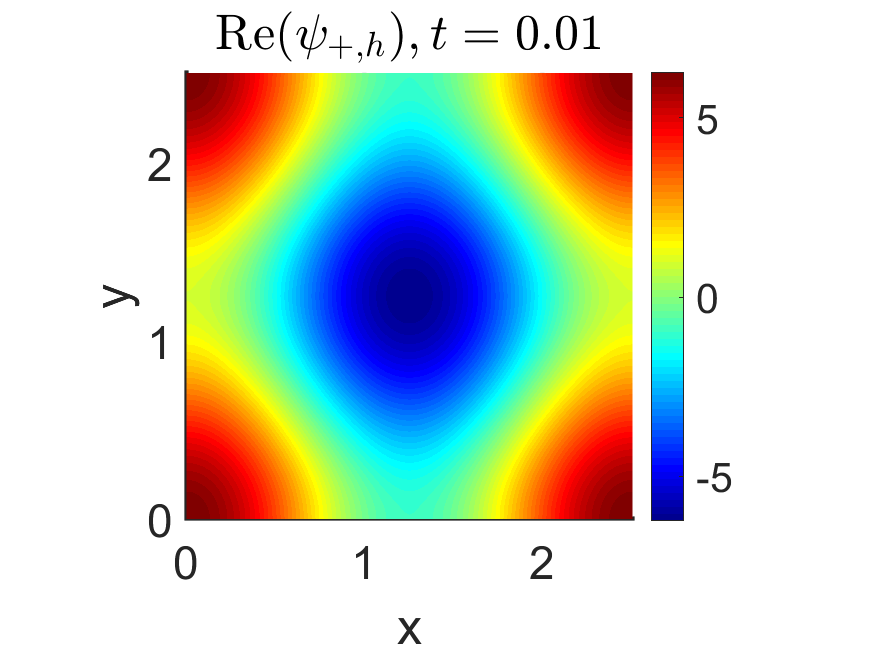} \;
				\includegraphics[width=5.5cm,height=4.5cm]{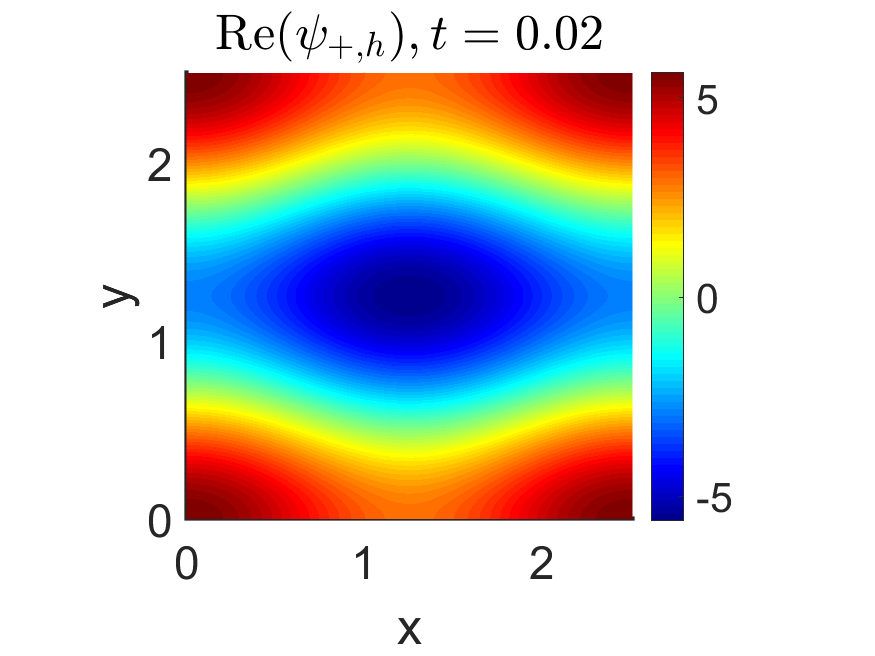} \; \\
				\includegraphics[width=5.5cm,height=4.5cm]{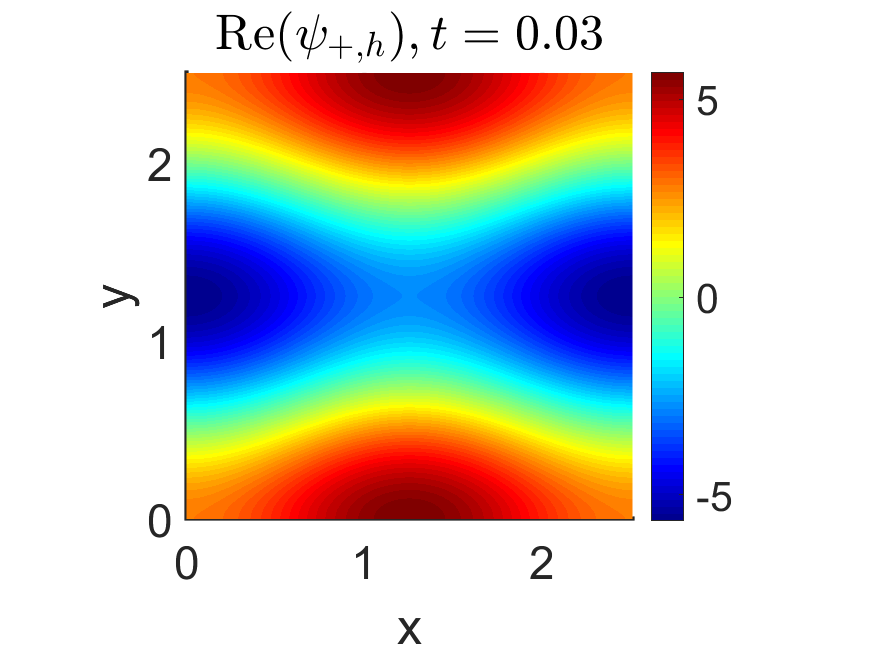} \;
				\includegraphics[width=5.5cm,height=4.5cm]{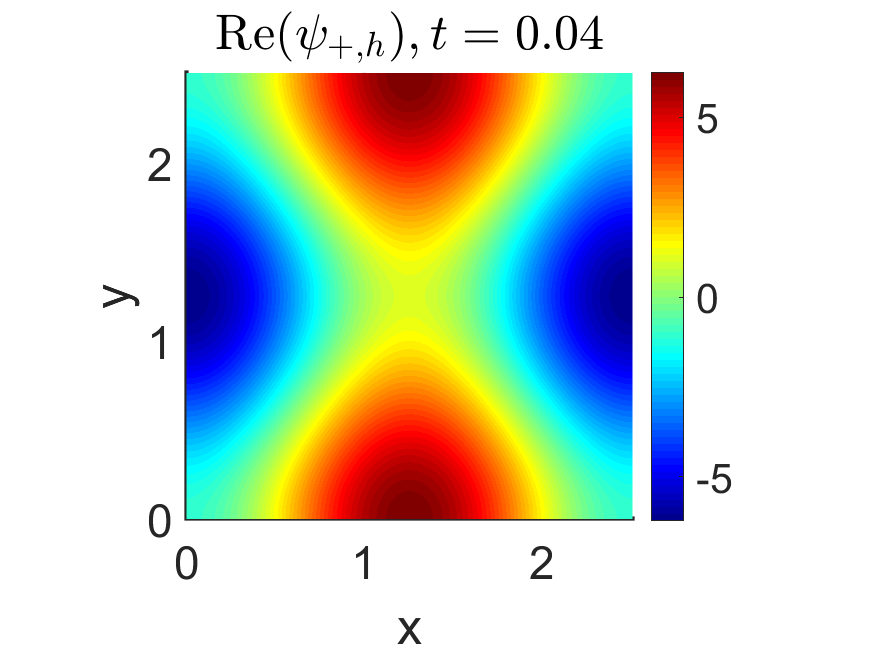} \;
				\includegraphics[width=5.5cm,height=4.5cm]{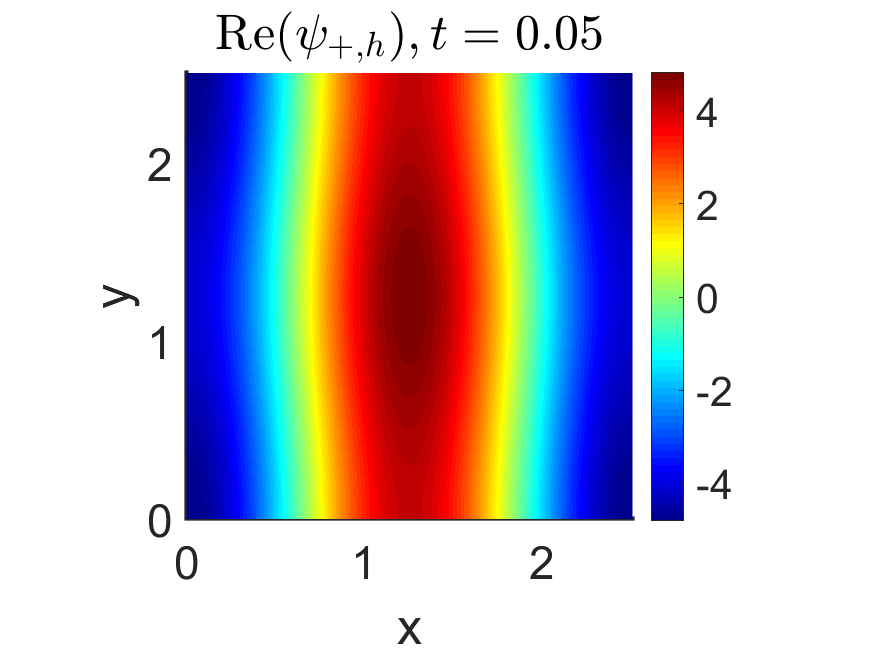} \; \\
			\end{array}$\vspace{-0.2cm}
			\caption{Snapshots of $\text{Re}(\psi_{+,h})$} \label{ex2D psip}
		\end{figure}
		
	\end{exmp}

	\appendix
	\section{Proof of \Cref{TH error estimate}}\label{Appendix}
	In this section, we present the proof of \Cref{TH error estimate} based on the method of induction.
	\begin{proof}[Proof of \Cref{TH error estimate}]
		\par
		\textbf{Step 1.} We begin by proving the following estimates
		\begin{align}
			&\left\| {e_{Z, \pm }^{\frac{1}{2}}} \right\| \le C\left( {{\tau ^2} + {h^{k + 1}}} \right),\\
			&\left\| {e_\phi ^{\frac{1}{2}}} \right\| \le C\left( {{\tau ^2} + {h^{k + 1}}} \right),\\
			&\left\| {e_{\psi , \pm }^1} \right\| \le C\left( {{\tau ^2} + {h^{k + 1}}} \right), \label{epsi1}\\
			&\left\| {{D_\tau }\eta _{\psi , \pm }^1} \right\| \le C\left( {{\tau ^2} + {h^{k + 1}}} \right).
		\end{align}
		We now take ${\chi _{+,h }} = \eta _{Z, + }^{\frac{1}{2}} - \eta _{Z, + }^{ - \frac{1}{2}}$
		in \eqref{deErrorEq4}, and from this, we obtain
		\begin{align}
			\begin{aligned}
				{\left\| {\eta _{Z, + }^{\frac{1}{2}}} \right\|^2} - {\left\| {\eta _{Z, + }^{ - \frac{1}{2}}} \right\|^2} &= \left( {S_{2, + }^0,\eta _{Z, + }^{\frac{1}{2}} - \eta _{Z, + }^{ - \frac{1}{2}}} \right) + \left( {T_{1, + }^0,\eta _{Z, + }^{\frac{1}{2}} - \eta _{Z, + }^{ - \frac{1}{2}}} \right)\\
				&\le 2{\left\| {S_{2, + }^0} \right\|^2} + 2{\left\| {T_{1, + }^0} \right\|^2} + \frac{1}{2}{\left\| {\eta _{Z, + }^{\frac{1}{2}}} \right\|^2} + \frac{1}{2}{\left\| {\eta _{Z, + }^{ - \frac{1}{2}}} \right\|^2} .   
			\end{aligned}
		\end{align}
		Because of \eqref{S2+} and the fact that $\left\|T_{1,+}^0 \right\|= 2\left\| \psi_{+}^0  -  \psi_{+,h}^0 \right\| \left\| \psi_{+}^0  +  \psi_{+,h}^0 \right\|_{\infty} \le Ch^{k+1}$, along with the inequality $ {\left\| {\eta _{Z, + }^{ - \frac{1}{2}}} \right\|^2} \le {\left\| {e _{Z, + }^{ - \frac{1}{2}}} \right\|^2} + {\left\| {\xi _{Z, + }^{ - \frac{1}{2}}} \right\|^2} \le Ch^{k+1}$, it follows that
		\begin{equation}
			{\left\| {\eta _{Z, + }^{\frac{1}{2}}} \right\|^2} \le 4{\left\| {S_{2, + }^0} \right\|^2} + 4{\left\| {T_{1, + }^0} \right\|^2} + 3{\left\| {\eta _{Z, + }^{ - \frac{1}{2}}} \right\|^2} \le C{\left( {{\tau ^2} + {h^{k + 1}}} \right)^2}. \label{etaZ0}    
		\end{equation}
		We then apply the projection error \eqref{Ritz projection error} to \eqref{etaZ0}, yielding
		\begin{equation}
			\left\| {e_{Z , + }^{\frac{1}{2}}} \right\| \le \left\| {\eta_{Z , + }^{\frac{1}{2}}} \right\| + \left\| {\xi_{Z , + }^{\frac{1}{2}}} \right\| \le C\left( {{\tau ^2} + {h^{k + 1}}} \right).     \label{eZp}
		\end{equation}
		Similarly, it holds
		\begin{align}
			&  {\left\| {\eta _{Z, - }^{\frac{1}{2}}} \right\|} \le  C{\left( {{\tau ^2} + {h^{k + 1}}} \right)}, \label{etaZ0+} \\
			&  \left\| {e_{Z , - }^{\frac{1}{2}}} \right\| \le \left\| {\eta_{Z , - }^{\frac{1}{2}}} \right\| + \left\| {\xi_{Z , - }^{\frac{1}{2}}} \right\| \le C\left( {{\tau ^2} + {h^{k + 1}}} \right).   \label{eZn}  
		\end{align}
		Next, by using Lemma \ref{propties of A1} in \eqref{ErrorEq3} at $n=0$, and incorporating \eqref{eZp} and \eqref{eZn}, 
		\begin{equation}
			\left\| {e_\phi ^{\frac{1}{2}}} \right\| \le C\left( {\left\| {e_{Z, + }^{\frac{1}{2}}} \right\| + \left\| {e_{Z, - }^{\frac{1}{2}}} \right\|} \right) + C{h^{k + 1}} \le C\left( {{\tau ^2} + {h^{k + 1}}} \right).   \label{ephi}
		\end{equation}
		By \eqref{etaZ0}, \eqref{etaZ0+}, and \eqref{ephi},  \Cref{lem:inftol2} and the assumption $\tau<Ch$ imply that there exists $h_1>0$ such that when $h<h_1$,
		\begin{align}
			&\|Z_{h, \pm }^{\frac{1}{2}}\| \leq \|R_hZ_{\pm }^{\frac{1}{2}}\|_{\infty} + \| \eta_{Z,\pm}^{\frac{1}{2}}\|_{\infty} \leq \|R_hZ_{\pm }^{\frac{1}{2}}\|_{\infty} + Ch^{-\frac{d}{2}}\| \eta_{Z,\pm}^{\frac{1}{2}}\| \leq D_Z+ Ch^{k+1-\frac{d}{2}} \leq D_Z+1, \\
			&\left\| {\phi_h^{\frac{1}{2}}} \right\|_\infty \leq \|R_h\phi^{\frac{1}{2}}\|_{\infty} + Ch^{-\frac{d}{2}}\| \eta_{\phi}^{\frac{1}{2}}\| \leq D_\phi+ Ch^{k+1-\frac{d}{2}} \leq D_\phi+1.
		\end{align}
		Then, by choosing $v_{+,h} = \bar \eta _{\psi , + }^{\frac{1}{2}}$ in \eqref{deErrorEq1} with $n=0$, we derive the following inequality
		\begin{align}
			\begin{aligned}
				\frac{1}{{2\tau }}\left( {{{\left\| {\eta _{\psi , + }^1} \right\|}^2} - {{\left\| {\eta _{\psi , + }^0} \right\|}^2}} \right) &= {\mathop{\rm Im}\nolimits} \left( {J_{1, + }^1,\bar \eta _{\psi , + }^{\frac{1}{2}}} \right) + {\mathop{\rm Im}\nolimits} \left( {R_{2, + }^1,\bar \eta _{\psi , + }^{\frac{1}{2}}} \right)\\
				&\le {\left\| {J_{1, + }^1} \right\|^2} + {\left\| {R_{2, + }^1} \right\|^2} + \frac{1}{4}\left( {{{\left\| {\eta _{\psi , + }^1} \right\|}^2} + {{\left\| {\eta _{\psi , + }^0} \right\|}^2}} \right).
			\end{aligned} \label{eta_psi+}
		\end{align} 
		In view of \eqref{eZp} and \eqref{ephi}, we have
		\begin{align}
			\begin{aligned}
				\left\| {J_{1, + }^1} \right\| &= \left\| {\left( {gZ_ + ^{\frac{1}{2}} + GZ_ - ^{\frac{1}{2}} + q\phi _{}^{\frac{1}{2}}} \right)\bar \psi _ + ^{\frac{1}{2}} - \left( {gZ_{+,h }^{\frac{1}{2}} + GZ_{-,h }^{\frac{1}{2}} + q\phi _h^{\frac{1}{2}}} \right)\bar \psi _{+,h }^{\frac{1}{2}}} \right\|\\
				&\le \left\| {\left( {gZ_ + ^{\frac{1}{2}} + GZ_ - ^{\frac{1}{2}} + q\phi _{}^{\frac{1}{2}}} \right)\bar \psi _ + ^{\frac{1}{2}} - \left( {gZ_{+,h }^{\frac{1}{2}} + GZ_{-,h }^{\frac{1}{2}} + q\phi _h^{\frac{1}{2}}} \right)\bar \psi _ + ^{\frac{1}{2}}} \right\|\\
				& \quad + \left\| {\left( {gZ_{+,h }^{\frac{1}{2}} + GZ_{-,h }^{\frac{1}{2}} + q\phi _h^{\frac{1}{2}}} \right)\bar \psi _ + ^{\frac{1}{2}} - \left( {gZ_{+,h }^{\frac{1}{2}} + GZ_{-,h }^{\frac{1}{2}} + q\phi _h^{\frac{1}{2}}} \right)\bar \psi _{+,h }^{\frac{1}{2}}} \right\|\\
				&\le {\left\| {\bar \psi _ + ^{\frac{1}{2}}} \right\|_\infty }\left( {g\left\| {e_{Z, + }^{\frac{1}{2}}} \right\| + G\left\| {e_{Z, - }^{\frac{1}{2}}} \right\| + q\left\| {e_\phi ^{\frac{1}{2}}} \right\|} \right) + {\left\| {gZ_{+,h }^{\frac{1}{2}} + GZ_{-,h }^{\frac{1}{2}} + q\phi _h^{\frac{1}{2}}} \right\|_\infty }\left\| {e_{\psi , + }^{\frac{1}{2}}} \right\|\\
				&\le C\left( {{\tau ^2} + {h^{k + 1}}} \right) + C\left( {\left\| {\eta _{\psi , + }^1} \right\| + \left\| {\eta _{\psi , + }^0} \right\|} \right).
			\end{aligned} \label{G1+}
		\end{align} 
		Substituting \eqref{G1+} and \eqref{R2+} into \eqref{eta_psi+}, we obtain
		\begin{equation} \label{e psi+ M}
			\frac{1}{{2\tau }}\left( {{{\left\| {\eta _{\psi , + }^1} \right\|}^2} - {{\left\| {\eta _{\psi , + }^0} \right\|}^2}} \right) \le {C_1}\left( {{{\left\| {\eta _{\psi , + }^1} \right\|}^2} + {{\left\| {\eta _{\psi , + }^0} \right\|}^2}} \right) + C{\left( {{\tau ^2} + {h^{k + 1}}} \right)^2}.
		\end{equation}
		Since $\eta _{\psi , + }^0 = 0 $, we deduce that for $\tau\le \tau_1:=1/(2C_1)$, 
		\begin{equation}
			{\left\| {\eta _{\psi , + }^1} \right\|} \le C\tau \left( {{\tau ^2} + {h^{k + 1}}} \right). \label{e1 psi+}
		\end{equation}
		By combining the interpolation error estimation, we obtain the following bound 
		\begin{equation}
			\left\| {e_{\psi , + }^1} \right\| \le C\tau \left( {{\tau ^2} + {h^{k + 1}}} \right) + C\left( {{\tau ^2} + {h^{k + 1}}} \right) \le C\left( {{\tau ^2} + {h^{k + 1}}} \right).
		\end{equation}
		Additionally, we have
		\begin{equation}
			\left\| {{D_\tau }\eta _{\psi , + }^1} \right\| = \frac{1}{\tau }\left\| {\eta _{\psi , + }^1} \right\| \le C\left( {{\tau ^2} + {h^{k + 1}}} \right). \label{De1 psi+}
		\end{equation}
		Similarly, by taking $v_{-,h} = \bar \eta _{\psi , - }^{\frac{1}{2}}$ in \eqref{ErrorEq2} with $n=0$, it follows
		\begin{align}
			& \left\| {e_{\psi , - }^1} \right\| \le C\left( {{\tau ^2} + {h^{k + 1}}} \right), \\
			& \left\| {{D_\tau }\eta _{\psi , - }^1} \right\| \le C\left( {{\tau ^2} + {h^{k + 1}}} \right). 
		\end{align}
		By using \eqref{inverse inequality}, \eqref{boundedness of the true solution}, \eqref{e1 psi+}, and the assumption, there exists a constant $h_{2}>0$, such that for $h<h_2$, it holds
		\begin{align}
			& {\left\| {\psi _{\pm,h }^1} \right\|_\infty } \le {\left\| {{R_h}\psi _ \pm ^1} \right\|_\infty } + {\left\| {{R_h}\psi _ \pm ^1 - \psi _{\pm,h }^1} \right\|_\infty } \le D_\psi + Ch^{-\frac{d}{2}} \left\| \eta _{\psi , \pm }^1 \right\| \le D_\psi + 1.\label{Dpsi0+}
		\end{align}
		
		\textbf{Step 2.} We now show that the second step for the intermediate variable $Z$ and potential $\phi$ similarly satisfy the error estimate
		\begin{align}
			& \|e_{\psi,\pm}^2\| \le C(\tau^2 + h^{k+1}), \label{e_psipm2}\\
			& \left\| {{D_\tau }\eta _{\psi , \pm }^2} \right\|  \le C\left( {{\tau ^2} + {h^{k + 1}}} \right), \label{Dtaueta_psipm2}\\
			&\mathop {\max }\limits_{1 \le n \le 2} \left\| {e_{Z, \pm }^{n+\frac{1}{2}}} \right\| \le C\left( {{\tau ^2} + {h^{k + 1}}} \right),\label{e_Zpm3/2}\\
			&\mathop {\max }\limits_{1 \le n \le 2} \left\| {e_\phi ^{n+\frac{1}{2}}} \right\| \le C\left( {{\tau ^2} + {h^{k + 1}}} \right). \label{e_phi3/2}
		\end{align}
		We take $n=1,0$ in \eqref{deErrorEq4}, respectively, and subtract the resulting equations to obtain
		\begin{equation}
			\left( {\eta _{Z, + }^{\frac{3}{2}} - \eta _{Z, + }^{ - \frac{1}{2}},{\chi _{+,h }}} \right) = \left( {S_{2, + }^1 - S_{2, + }^0,{\chi _{+,h }}} \right) + \left( {T_{1, + }^1 - T_{1, + }^0,{\chi _{+,h }}} \right). \label{eta3/2-eta-1/2}
		\end{equation}
		By using the projection error \eqref{Ritz projection error} and the mean value theorem: there exists $t^* \in (t_{-\frac{1}{2}},t_{\frac{3}{2}})$ such that
		\begin{equation} \label{xi Z 1}
			\begin{aligned}
				\left\| {\xi _{Z, + }^{\frac{3}{2}} - \xi _{Z, + }^{ - \frac{1}{2}}} \right\| & = 2\tau \left\| \left(\frac{Z_+^{\frac{3}{2}}-Z_+^{-\frac{1}{2}}}{2\tau} \right) - \text{R}_h \left(\frac{Z_+^{\frac{3}{2}}-Z_+^{-\frac{1}{2}}}{2\tau} \right)\right\| \le C\tau h^{k+1} \left\| \left(\frac{Z_+^{\frac{3}{2}}-Z_+^{-\frac{1}{2}}}{2\tau} \right)\right\|_{k+1} \\ &  \le C\tau h^{k+1} \left\| \partial_t Z_+(x,t^*)\right\|_{k+1}.    
			\end{aligned}
		\end{equation}
		Next, applying a Taylor expansion at $t_1$ gives
		\begin{equation}
			\begin{aligned}
				Z_ + ^{ \frac{3}{2}} -2 Z_ + ^{ 1} + 2 Z_ + ^{ 0} - Z_ + ^{ -\frac{1}{2}} & = \frac{1}{2} \int_{t_1}^{t_{\frac{3}{2}}} \partial_{ttt}Z_+(x,t)(t_{\frac{3}{2}}-t)^2 dx + \int_{t_1}^{t_{0}} \partial_{ttt}Z_+(x,t)(t_{0}-t)^2 dx \\ & \quad + \frac{1}{2} \int_{t_1}^{t_{-\frac{1}{2}}} \partial_{ttt}Z_+(x,t)(t_{-\frac{1}{2}}-t)^2 dx.
			\end{aligned} \label{taylor expansion Z+}
		\end{equation}
		By using the regularity assumption \eqref{regularity assumption} and the expansion in \eqref{taylor expansion Z+}, 
		\begin{equation}\label{S1 1}
			\left\| {S_{1, + }^1 - S_{1, + }^0} \right\| = \left\| Z_ + ^{ \frac{3}{2}} -2 Z_ + ^{ 1} + 2 Z_ + ^{ 0} - Z_ + ^{ -\frac{1}{2}}\right\|  \le C\tau^3.
		\end{equation}
		Combining \eqref{xi Z 1} and \eqref{S1 1}, we have
		\begin{equation} 
			\begin{aligned}
				\left\| {S_{2, + }^1 - S_{2, + }^0} \right\| &\le \left\| {S_{1, + }^1 - S_{1, + }^0} \right\| + \left\| {\xi _{Z, + }^{\frac{3}{2}} - \xi _{Z, + }^{ - \frac{1}{2}}} \right\| \\ & \le C{\tau ^3} + C\tau {h^{k + 1}}{\left\| {{D_\tau }{Z_ + }\left( {x,{t^*}} \right)} \right\|_{k + 1}} \le C\tau \left( {{\tau ^2} + {h^{k + 1}}} \right). \label{S21-S20}
			\end{aligned}
		\end{equation}
		Next, applying \Cref{lemma2}, \eqref{Ritz projection error}, \eqref{boundedness of the true solution} and \eqref{Dpsi0+} yields
		\begin{align}\label{lastT11-T10}
			\begin{aligned}
				&\left\| {T_{1, + }^1 -  T_{1, + }^0} \right\|  = 2\left\|{\left| {\psi _ + ^1} \right|^2} - {\left| {\psi _{+,h }^1} \right|^2 } - {\left| {\psi _ + ^0} \right|^2} + {\left| {\psi _{+,h }^0} \right|^2 } \right\| \\
				& \qquad \le 4 \left\| \psi_+^1-\psi_+^0 \right\|_{\infty} \left\| e_{\psi,+}^1 \right\| + 2\left(\|\psi_{+,h}^0\|_{\infty}+\|\psi_{+,h}^1\|_{\infty} + \|\psi_+^1-\psi_+^0\|_{\infty}\right) \left\| e_{\psi,+}^1 -  e_{\psi,+}^0\right\|,  \\
				& \qquad \le C\tau\left\| D_\tau \eta_{\psi,+}^1\right\| + C\tau\left\| \eta_{\psi,+}^1\right\| + C\tau h^{k+1}, 
			\end{aligned}
		\end{align}
		where we have used the Taylor expansion and the projection error estimate \eqref{Ritz projection error}.
		
		Taking ${\chi _{+,h }} = \eta _{Z, + }^{\frac{3}{2}} + \eta _{Z, + }^{ - \frac{1}{2}}$ in \eqref{eta3/2-eta-1/2}, and combining \eqref{e1 psi+}, \eqref{De1 psi+}, \eqref{S21-S20} and \eqref{lastT11-T10}, we obtain
		\begin{equation}
			\left\| {\eta _{Z, + }^{\frac{3}{2}}} \right\| \le \left\| {\eta _{Z, + }^{ - \frac{1}{2}}} \right\| + \left\| {S_{2, + }^1 - S_{2, + }^0} \right\| + \left\| {T_{1, + }^1 - T_{1, + }^0} \right\| \le C \tau \left( {{\tau ^2} + {h^{k + 1}}} \right).
		\end{equation}
		Together with the projection error estimate \eqref{Ritz projection error}, this implies
		\begin{equation}
			\left\| {e_{Z, + }^{\frac{3}{2}}} \right\| \le \left\| {\eta _{Z, + }^{\frac{3}{2}}} \right\| + \left\| {\xi _{Z, + }^{\frac{3}{2}}} \right\| \le C\left( {{\tau ^2} + {h^{k + 1}}} \right).
			\label{eZ3/2+}
		\end{equation}
		Similarly, applying the same argument to \eqref{deErrorEq5}, we obtain
		\begin{equation}
			\left\| {e_{Z, - }^{\frac{3}{2}}} \right\| \le C\left( {{\tau ^2} + {h^{k + 1}}} \right).  \label{eZ3/2-}
		\end{equation}
		By applying \Cref{propties of A1} to \eqref{ErrorEq3} with $n=1$, and combining the estimates in \eqref{eZ3/2+} and \eqref{eZ3/2-},
		\begin{equation} \label{ephi3/2}
			\left\| {e_\phi ^{\frac{3}{2}}} \right\| \le 4\pi \left| q \right|\left( {\left\| {e_{Z, + }^{\frac{3}{2}}} \right\| + \left\| {e_{Z, - }^{\frac{3}{2}}} \right\|} \right) + C{h^{k + 1}} \le C\left( {{\tau ^2} + {h^{k + 1}}} \right).
		\end{equation}
		
		By \Cref{lem:inftol2}, \eqref{boundedness of the true solution}, \eqref{eZ3/2+}, \eqref{eZ3/2-}, \eqref{ephi3/2} and the assumption, there exists $h_3>0$, such that for $h<h_3$, it holds
		\begin{align}
			&\|Z_{h, \pm }^{\frac{3}{2}}\| \leq \|R_hZ_{\pm }^{\frac{3}{2}}\|_{\infty} + \| \eta_{Z,\pm}^{\frac{3}{2}}\|_{\infty} \leq \|R_hZ_{\pm }^{\frac{3}{2}}\|_{\infty} + Ch^{-\frac{d}{2}}\| \eta_{Z,\pm}^{\frac{3}{2}}\| \leq D_Z+ Ch^{k+1-\frac{d}{2}} \leq D_Z+1, \\
			&\left\| {\phi_h^{\frac{3}{2}}} \right\|_\infty \leq \|R_h\phi^{\frac{3}{2}}\|_{\infty} + Ch^{-\frac{d}{2}}\| \eta_{\phi}^{\frac{3}{2}}\| \leq D_\phi+ Ch^{k+1-\frac{d}{2}} \leq D_\phi+1.
		\end{align}
		When $n =2$, by applying arguments similar to those used from \eqref{eta_psi+} to \eqref{e1 psi+}, with the index changed from $n=1$ to $n=2$. There exists $\tau_2>0$ such that for $\tau\le \tau_2$,
		\begin{equation}
			{\left\| {\eta _{\psi , + }^2} \right\|} \le C\tau \left( {{\tau ^2} + {h^{k + 1}}} \right), \label{e2 psi+}
		\end{equation}
		where the estimate for $\|R_{2,+}^2\|$ follows directly from \eqref{R2+}, while the estimate for $\|J_{2,+}^2\|$ relies on the intermediate results \eqref{eZ3/2+}, \eqref{eZ3/2-}, and \eqref{ephi3/2}.
		Consequently, we obtain
		\begin{align}
			& \|e_{\psi,\pm}^2\| \le C(\tau^2 + h^{k+1}) ,\\
			& \left\| {{D_\tau }\eta _{\psi , \pm }^2} \right\| = \frac{1}{\tau }\left\| {\eta _{\psi , \pm }^2-\eta _{\psi , \pm }^1} \right\| \le C\left( {{\tau ^2} + {h^{k + 1}}} \right).
		\end{align}
		Then there exists a constant $h_4>0$ such that, when $h< h_4$, it holds that
		\begin{equation} \label{BD of psi+2}
			{\left\| {\psi _{\pm,h }^2} \right\|_\infty }  \le D_\psi + 1.
		\end{equation}
		Similarly, following the steps from \eqref{eta3/2-eta-1/2} to \eqref{ephi3/2} and updating the index accordingly, we obtain
		\begin{align}
			\left\| {e_{Z, \pm }^{\frac{5}{2}}} \right\| \le C\left( {{\tau ^2} + {h^{k + 1}}} \right), \\
			\left\| {e_{\phi}^{\frac{5}{2}}} \right\| \le C\left( {{\tau ^2} + {h^{k + 1}}} \right).
		\end{align}
		Therefore, the estimates \eqref{intro psi}–\eqref{intro phi} hold for $n=2$.
		
		\textbf{Step 3.} Assume that for $0 \le n \le m$ with $m > 2$, the following error estimates hold
		\begin{align}
			&\mathop {\max }\limits_{0 \le n \le m} \left\| {e_{\psi , \pm }^{n }} \right\| \le C\left( {{\tau ^2} + {h^{k + 1}}} \right), \label{intro psi}\\
			&\mathop {\max }\limits_{0 \le n \le m} \left\| {{D_\tau }\eta_{\psi , \pm }^{n }} \right\| \le C\left( {{\tau ^2} + {h^{k + 1}}} \right), \label{intro Dpsi}\\
			&\mathop {\max }\limits_{0 \le n \le m} \left\| {e_{Z, \pm }^{n + \frac{1}{2}}} \right\| \le C\left( {{\tau ^2} + {h^{k + 1}}} \right), \label{intro Z}\\
			&\mathop {\max }\limits_{0 \le n \le m} \left\| {e_\phi ^{n + \frac{1}{2}}} \right\| \le C\left( {{\tau ^2} + {h^{k + 1}}} \right). \label{intro phi}
		\end{align}
		
		In general, for $n>2$, arguments similar to those in Step 2 show that the following boundedness conditions hold as a consequence of the induction hypotheses \eqref{intro psi}, \eqref{intro Z}, and \eqref{intro phi}. In particular, there exists a constant $h_5>0$ such that, when $h<h_5$, it follows for $0\le n \leq m$
		\begin{align}
			&{\left\| {\psi _{h, \pm }^n} \right\|_\infty } \le {\left\| {{R_h}\psi _ \pm ^n} \right\|_\infty } + {\left\| {{R_h}\psi _ \pm ^n - \psi _{h, \pm }^n} \right\|_\infty } \le D_\psi+1, \label{Boud psi}\\
			&{\left\| {Z_{h, \pm }^{n + \frac{1}{2}}} \right\|_\infty } \le {\left\| {{R_h}Z_ \pm ^{n + \frac{1}{2}}} \right\|_\infty } + {\left\| {{R_h}Z_ \pm ^{n + \frac{1}{2}} - Z_{h, \pm }^{n + \frac{1}{2}}} \right\|_\infty } \le D_Z+1, \label{Boud Z}\\
			&{\left\| {\phi _h^{n + \frac{1}{2}}} \right\|_\infty } \le {\left\| {{R_h}\phi _{}^{n + \frac{1}{2}}} \right\|_\infty } + {\left\| {{R_h}\phi _{}^{n + \frac{1}{2}} - \phi _h^{n + \frac{1}{2}}} \right\|_\infty } \le D_\phi+1. \label{Boud phi}
		\end{align}
		
		Next, we consider the equations at $t_{n+1}$ and $t_{n-1}$ with $2<n\le m$ for \eqref{deErrorEq1}. Subtracting the two resulting equations yields
		\begin{equation}
			\begin{aligned}
				\left\langle \mathbf{i}{{D_\tau }\eta _{\psi , + }^{n + 1} - \mathbf{i}{D_\tau }\eta _{\psi , + }^{n - 1},{u_h}} \right\rangle &= \frac{1}{2}{A_0}\left( {\bar \eta _{\psi , + }^{n + \frac{1}{2}} - \bar \eta _{\psi , + }^{n - \frac{3}{2}},{u_h}} \right) \\ & \quad + \left\langle {J_{1, + }^{n + 1} - J_{1, + }^{n - 1},{u_h}} \right\rangle + \left\langle {R_{2, + }^{n + 1} - R_{2, + }^{n - 1},{u_h}} \right\rangle\\
				&= \frac{\tau }{4}{A_0}\left( {{D_\tau }\eta _{\psi , + }^{n + 1} + 2{D_\tau }\eta _{\psi , + }^n + {D_\tau }\eta _{\psi , + }^{n - 1},{u_h}} \right) \\ & \quad + \left\langle {J_{1, + }^{n + 1} - J_{1, + }^{n - 1},{u_h}} \right\rangle + \left\langle {R_{2, + }^{n + 1} - R_{2, + }^{n - 1},{u_h}} \right\rangle.
			\end{aligned}
		\end{equation}
		This relation can be written in pointwise form as
		\begin{equation}
			{D_\tau }\eta _{\psi , + }^{n + 1} - {D_\tau }\eta _{\psi , + }^{n - 1} = \mathbf{i}\frac{\tau }{4}{\Delta _h}\left( {{D_\tau }\eta _{\psi , + }^{n + 1} + 2{D_\tau }\eta _{\psi , + }^n + {D_\tau }\eta _{\psi , + }^{n - 1}} \right) + \Gamma _{1, + }^{n + 1} + \Gamma _{2, + }^{n + 1}, \label{pointwiseform}
		\end{equation}
		where
		\begin{equation*}
			\Gamma _{1, + }^{n + 1} =  - \mathbf{i}{P_h}\left( {R_{2, + }^{n + 1} - R_{2, + }^{n - 1}} \right),\quad \Gamma _{2, + }^{n + 1} =  - \mathbf{i}{P_h}\left( {J_{1, + }^{n + 1} - J_{1, + }^{n - 1}} \right),
		\end{equation*}
		and $P_h: \mathbf{L}^2(\Omega) \rightarrow \mathbf{V}$ denotes the $L^2$ projection.
		
		According to definition of $S_h$ and $T_h$ in \eqref{definitionS} and \eqref{definitionT}, equation \eqref{pointwiseform} can be written as
		\begin{equation} \label{D eta S}
			{S_h}\left( {{D_\tau }\eta _{\psi , + }^{n + 1}} \right) = \left( {{T_h} - {S_h}} \right)\left( {{D_\tau }\eta _{\psi , + }^n} \right) + {T_h}\left( {{D_\tau }\eta _{\psi , + }^{n - 1}} \right) + \Gamma _{1, + }^{n + 1} + \Gamma _{2, + }^{n + 1}.    
		\end{equation}
		\Cref{lemma: B} shows that the operator $S_h$ is invertible. Multiplying both sides of \eqref{D eta S} by $S_h^{-1}$ yields
		\begin{equation}\label{D eta S with B}
			{D_\tau }\eta _{\psi , + }^{n + 1} = \left( {{B_h} - {I_h}} \right)\left( {{D_\tau }\eta _{\psi , + }^n} \right) + {B_h}\left( {{D_\tau }\eta _{\psi , + }^{n - 1}} \right) + S_h^{ - 1}\left( {\Gamma _{1, + }^{n + 1} + \Gamma _{2, + }^{n + 1}} \right).
		\end{equation}
		Applying \Cref{lemma: sequence lemma} to \eqref{D eta S with B}, we obtain
		\begin{equation}
			\left\| {{D_\tau }\eta _{\psi , + }^{n + 1}} \right\| + \left\| {{D_\tau }\eta _{\psi , + }^n} \right\| \le 2\left\| {{S_h}{D_\tau }\eta _{\psi , + }^2} \right\| + 2\left\| {{S_h}{D_\tau }\eta _{\psi , + }^1} \right\| + 2\sum\limits_{n = 2}^m {\left( \|{\Gamma _{1, + }^{n + 1} \|+ \|\Gamma _{2, + }^{n + 1}} \|\right)}. \label{etam+1_and_etam}
		\end{equation}
		Similarly, we have
		\begin{equation}
			\left\| {{D_\tau }\eta _{\psi , - }^{n + 1}} \right\| + \left\| {{D_\tau }\eta _{\psi , - }^n} \right\| \le 2\left\| {{S_h}{D_\tau }\eta _{\psi , - }^2} \right\| + 2\left\| {{S_h}{D_\tau }\eta _{\psi , - }^1} \right\| + 2\sum\limits_{n = 2}^m {\left( \| {\Gamma _{1, - }^{n + 1}\| +\| \Gamma _{2, - }^{n + 1}} \| \right)}, \label{etam-1_and_etam}
		\end{equation}
		where
		\begin{equation*}
			\Gamma _{1, - }^{n + 1} =  - \mathbf{i}{P_h}\left( {R_{2, - }^{n + 1} - R_{2, - }^{n - 1}} \right),\quad \Gamma _{2, - }^{n + 1} =  - \mathbf{i}{P_h}\left( {J_{1, - }^{n + 1} - J_{1, - }^{n - 1}} \right).
		\end{equation*}
		
		\textbf{Step 4.} 
		In this step, we estimate the terms $\|\Gamma _{1, \pm }^{n + 1}\|$ and $\Gamma _{2, \pm }^{n + 1}$ in \eqref{etam+1_and_etam} and \eqref{etam-1_and_etam}, respectively.
		\begin{align}
			&\left\| {\Gamma _{1, \pm }^{n + 1}} \right\| \le \;\left\| {R_{2, \pm }^{n + 1} - R_{2, \pm }^{n - 1}} \right\| \le \left\| {R_{1, \pm }^{n + 1} - R_{1, \pm }^{n - 1}} \right\| + \left\| {{D_\tau }\xi _{\psi , \pm }^{n + 1} - {D_\tau }\xi _{\psi , \pm }^{n - 1}} \right\|, \label{gamma1 n+1}\\
			&\left\| {\Gamma _{2, \pm }^{n + 1}} \right\| \le \left\| {J_{1, \pm }^{n + 1} - J_{1, \pm }^{n - 1}} \right\|. \label{gamma2 n+1}
		\end{align}
		To estimate $\|{\Gamma _{1, + }^{n + 1}}\|$, we need to prove
		\begin{align}
			&\left\| {R_{1, + }^{n + 1} - R_{1, + }^n} \right\| \le C{\tau ^3} \label{R1+ M+1- R1+ M}, \\
			&\left\| {{D_\tau }\xi _{\psi , + }^{n + 1} - {D_\tau }\xi _{\psi , + }^{n - 1}} \right\| \le C\tau {h^{k + 1}}. \label{DXI PSI M+1 - DXI PSI M-1+}
		\end{align}
		Firstly, we notice that
		\begin{equation}\label{expansion of R1+}
			\begin{aligned}
				\left\| {R_{1, + }^{n + 1} - R_{1, + }^n} \right\| & \le \left\| \left( {{\partial _t}\psi _ + ^{n + \frac{1}{2}} - {D_\tau }\psi _ + ^{n + 1}} \right) - \left( {{\partial _t}\psi _ + ^{n - \frac{1}{2}} - {D_\tau }\psi _ + ^{n}} \right)  \right\| \\ & \quad + \frac{1}{2} \left\| \Delta  \left( {\bar \psi _ + ^{n + \frac{1}{2}} - \psi _ + ^{n + \frac{1}{2}}} - {\bar \psi _ + ^{n - \frac{1}{2}} + \psi _ + ^{n - \frac{1}{2}}} \right)  \right\| \\ & \quad +\left((|g|+|G|)C_Z + qC_\phi\right) \left\| \left( {\bar \psi _ + ^{n + \frac{1}{2}} - \psi _ + ^{n + \frac{1}{2}}} - {\bar \psi _ + ^{n - \frac{1}{2}} + \psi _ + ^{n - \frac{1}{2}}} \right)  \right\|.
			\end{aligned}
		\end{equation}
		Then, we employ the Taylor expansion with the integral remainder
		\begin{align}
			\begin{aligned}\label{partial-Dtau}
				&\left( {{\partial _t}\psi _ + ^{n + \frac{1}{2}} - {D_\tau }\psi _ + ^{n + 1}} \right) - \left( {{\partial _t}\psi _ + ^{n - \frac{1}{2}} - {D_\tau }\psi _ + ^{n}} \right) \\ &= \frac{1}{2}\int_{t_n}^{t_{n+\frac{1}{2}}} ({t_{n+\frac{1}{2}}}-t)^2 \partial_{tttt}\psi_+(t)dt - \frac{1}{2}\int_{t_n}^{t_{n-\frac{1}{2}}} ({t_{n-\frac{1}{2}}}-t)^2 \partial_{tttt}\psi_+(t)dt \\ & \quad - \frac{1}{6\tau}\int_{t_n}^{t_{n+1}} ({t_{n+1}}-t)^3 \partial_{tttt}\psi_+(t)dt - \frac{1}{6\tau}\int_{t_n}^{t_{n-1}} ({t_{n-1}}-t)^3 \partial_{tttt}\psi_+(t)dt,
			\end{aligned} \\
			\begin{aligned} \label{psi 4}
				& {\bar \psi _ + ^{n + \frac{1}{2}} - \psi _ + ^{n + \frac{1}{2}}} - {\bar \psi _ + ^{n - \frac{1}{2}} + \psi _ + ^{n - \frac{1}{2}}}  \\ &= \frac{1}{12}\int_{t_n}^{t_{n+1}} ({t_{n+1}}-t)^3 \partial_{tttt}\psi_+(t)dt - \frac{1}{6}\int_{t_n}^{t_{n+\frac{1}{2}}} ({t_{n+\frac{1}{2}}}-t)^3 \partial_{tttt}\psi_+(t)dt \\ & \quad + \frac{1}{12}\int_{t_n}^{t_{n-1}} ({t_{n-1}}-t)^3 \partial_{tttt}\psi_+(t)dt + \frac{1}{6}\int_{t_n}^{t_{n-\frac{1}{2}}} ({t_{n-\frac{1}{2}}}-t)^3 \partial_{tttt}\psi_+(t)dt.
			\end{aligned}
		\end{align}
		Due to the regularity assumption \eqref{regularity assumption} of $\psi_+$, all integral remainder terms are $O(\tau^3)$. Substituting \eqref{partial-Dtau} and \eqref{psi 4} into \eqref{expansion of R1+}, we derive \eqref{R1+ M+1- R1+ M}.
		Moreover, notice that
		\begin{equation}
			D_\tau \psi_+^{n+1}- D_\tau \psi_+^{n-1} = \frac{1}{\tau}(\psi_+^{n+1}-\psi_+^{n}-\psi_+^{n-1}+\psi_+^{n-2}) = \frac{1}{\tau}\int_0^\tau \left( \int_{t_{n-2}+s}^{t_n+s} \partial_{tt} \psi_+(t)dt \right) ds. \label{preceding equation}
		\end{equation}
		By using the projection error estimate \eqref{Ritz projection error} and equation \eqref{preceding equation}, it follows that
		\begin{equation}
			\begin{aligned}
				\left\| {{D_\tau }\xi _{\psi , + }^{n + 1} - {D_\tau }\xi _{\psi , + }^{n - 1}} \right\| = \left\| R_h \left(D_\tau \psi_+^{n+1}- D_\tau \psi_+^{n-1} \right) -\left(D_\tau \psi_+^{n+1}- D_\tau \psi_+^{n-1} \right) \right\| \le C\tau {h^{k + 1}},
			\end{aligned}
		\end{equation}
		which completes the proof of \eqref{DXI PSI M+1 - DXI PSI M-1+}. Then, combining \eqref{R1+ M+1- R1+ M} and \eqref{DXI PSI M+1 - DXI PSI M-1+} yields
		\begin{equation}
			\left\| {\Gamma _{1, + }^{n + 1}} \right\| \le \;\left\| {R_{1, + }^{n + 1} - R_{1, + }^n} \right\| + \left\| {R_{1, + }^n - R_{1, + }^{n - 1}} \right\| + \left\| {{D_\tau }\xi _{\psi , + }^{n + 1} - {D_\tau }\xi _{\psi , + }^{n - 1}} \right\| \le C\tau \left( {{\tau ^2} + {h^{k + 1}}} \right).   \label{Gamma1+} 
		\end{equation}
		Similarly, we also have the following estimate
		\begin{equation}
			\left\| {\Gamma _{1, - }^{n + 1}} \right\| \le C\tau \left( {{\tau ^2} + {h^{k + 1}}} \right). 
		\end{equation}
		Next, by \eqref{gamma2 n+1}, we estimate $\|{\Gamma _{2, + }^{n + 1}}\|$ as follows
		\begin{equation}
			\begin{aligned}
				\left\| {J_{1, + }^{n + 1} - J_{1, + }^{n - 1}} \right\| & \le \left\| {g\left( {Z_ + ^{n + \frac{1}{2}}\bar \psi _ + ^{n + \frac{1}{2}} - Z_{+,h }^{n + \frac{1}{2}}\bar \psi _{+,h }^{n + \frac{1}{2}}} \right) - g\left( {Z_ + ^{n - \frac{3}{2}}\bar \psi _ + ^{n - \frac{3}{2}} - Z_{+,h }^{n - \frac{3}{2}}\bar \psi _{+,h }^{n - \frac{3}{2}}} \right)} \right\|\\
				& \quad + \left\| G{\left( {Z_ - ^{n + \frac{1}{2}}\bar \psi _ + ^{n + \frac{1}{2}} - Z_ - ^{n + \frac{1}{2}}\bar \psi _{+,h }^{n + \frac{1}{2}}} \right) - G\left( {Z_ - ^{n + \frac{1}{2}}\bar \psi _ + ^{n - \frac{3}{2}} - Z_ - ^{n + \frac{1}{2}}\bar \psi _{+,h }^{n - \frac{3}{2}}} \right)} \right\|\\
				& \quad + \left\| q{\left( {\phi _{}^{n + \frac{1}{2}}\bar \psi _ + ^{n + \frac{1}{2}} - \phi _{h}^{n + \frac{1}{2}}\bar \psi _{+,h }^{n + \frac{1}{2}}} \right) - q\left( {\phi _{}^{n + \frac{1}{2}}\bar \psi _ + ^{n - \frac{3}{2}} - \phi _{h}^{n + \frac{1}{2}}\bar \psi _{+,h }^{n - \frac{3}{2}}} \right)} \right\| \\ & :=I_1 + I_2 + I_3.
			\end{aligned}
		\end{equation}
		To estimate $I_1$, upon rewriting and using the triangle inequality, it follows
		\begin{align}
			\begin{aligned}
				I_1 &\le
				\left\| {g\left( {Z_ + ^{n + \frac{1}{2}} - Z_{ + ,h}^{n + \frac{1}{2}}} \right)\left( {\bar \psi _ + ^{n + \frac{1}{2}} - \bar \psi _ + ^{n - \frac{3}{2}}} \right)} \right\| \\ &\quad + \left\| {g\left( {Z_ + ^{n + \frac{1}{2}} - Z_ + ^{n - \frac{3}{2}}} \right)\left( {\bar \psi _ + ^{n - \frac{3}{2}} - \bar \psi _{ + ,h}^{n - \frac{3}{2}}} \right)} \right\|\\
				& \quad + \left\| {gZ_{ + ,h}^{n + \frac{1}{2}}\left( {\bar \psi _ + ^{n + \frac{1}{2}} - \bar \psi _ + ^{n - \frac{3}{2}} - \bar \psi _{ + ,h}^{n + \frac{1}{2}} + \bar \psi _{ + ,h}^{n - \frac{3}{2}}} \right)} \right\| \\& \quad + \left\| {g\left( {Z_ + ^{n + \frac{1}{2}} - Z_ + ^{n - \frac{3}{2}} - Z_{ + ,h}^{n + \frac{1}{2}} + Z_{ + ,h}^{n - \frac{3}{2}}} \right)\bar \psi _{ + ,h}^{n - \frac{3}{2}}} \right\| \\
				& :=I_{1,1} + I_{1,2} + I_{1,3} + I_{1,4}.
			\end{aligned} \label{I1 first}
		\end{align}
		Based on the inductive hypothesis in \eqref{TH3.7 Z} and \eqref{TH3.7 psi}, and applying the Taylor expansion, we estimate the terms $I_{1,1}$ and $I_{1,2}$ as follows
		\begin{align}
			& I_{1,1}  \le g\left\| {Z_ + ^{n + \frac{1}{2}} - Z_{ + ,h}^{n + \frac{1}{2}}} \right\| \left\| \frac{(\psi_+^{n+1}+\psi_+^n)-(\psi_+^{n-1}+\psi_+^{n-2})}{2} \right\| \le C\tau (\tau^2 + h^{k+1}), \label{K1} \\
			& I_{1,2}  \le g \tau \left\| \frac{(\psi_+^{n-1}+\psi_+^{n-2})-(\psi_{+,h}^{n-1}+\psi_{+,h}^{n-2})}{2} \right\| \le C\tau (\|\eta^{n-1}_{\psi,+}\| + \|\eta^{n-2}_{\psi,+}\|) + C\tau h^{k+1} \label{K2}.
		\end{align}
		Using the projection error \eqref{Ritz projection error} and the mean value theorem, we know that there exists $t_1^* \in (t^n, t^{n+1})$, such that
		\begin{equation} \label{D tau xi Z +}
			\left\| D_{\tau} \xi_{Z,+}^{n+1} \right\| = \left\| D_{\tau} Z_+^{n+1} - R_h D_{\tau} Z_+^{n+1} \right\| \le Ch^{k+1} \left\| D_{\tau} Z_+^{n+1} \right\|_{k+1} =  Ch^{k+1} \left\| \partial_t Z_+(x,t_1^*) \right\|_{k+1}.
		\end{equation}
		Then, by utilizing \eqref{Boud Z} and \eqref{D tau xi Z +}, we proceed to estimate $I_{1,3}$,
		\begin{equation}
			\begin{aligned}
				I_{1,3} &\le C \left\| \left( \bar \psi _ + ^{n + \frac{1}{2}} - \bar \psi _ + ^{n - \frac{3}{2}}\right) - \left(\bar \psi _{ + ,h}^{n + \frac{1}{2}} - \bar \psi _{ + ,h}^{n - \frac{3}{2}} \right) \right\| \le C \tau \left\| D_{\tau}e_{\psi,+}^{n+1} + 2D_{\tau}e_{\psi,+}^{n} + D_{\tau}e_{\psi,+}^{n-1}\right\| \\ &\le C \tau (\left\| D_{\tau}\eta_{\psi,+}^{n+1}\right\| +\left\| D_{\tau}\eta_{\psi,+}^{n}\right\|+\left\| D_{\tau}\eta_{\psi,+}^{n-1}\right\|)+ C\tau h^{k+1}. 
			\end{aligned} \label{K3}
		\end{equation}
		To estimate $I_{1,4}$, we notice that there exists $t_2^* \in (t^{n-\frac{3}{2}}, t^{n+\frac{1}{2}})$, 
		\begin{equation}
			\left\| \xi_{Z,+}^{n+\frac{1}{2}} - \xi_{Z,+}^{n-\frac{3}{2}} \right\| = 2 \tau \left\| \frac{Z_+^{n+\frac12}-Z_+^{n-\frac32}}{2\tau} - R_h\left(\frac{Z_+^{n+\frac12}-Z_+^{n-\frac32}}{2\tau}\right) \right\| \le C\tau h^{k+1} \left\| \partial_t Z_+(x,t_2^*) \right\|_{k+1},
		\end{equation}
		Next, by taking the difference of \eqref{deErrorEq4} between the time levels $ {n}$ and ${n-1}$, and setting $\chi_{+,h} = \eta_{Z,+}^{n+\frac{1}{2}} - \eta_{Z,+}^{n-\frac{3}{2}}$, we have
		\begin{equation}
			\left\| \eta_{Z,+}^{n+\frac{1}{2}} - \eta_{Z,+}^{n-\frac{3}{2}} \right\|^2 = (S_{2,+}^{n}-S_{2,+}^{n-1},\eta_{Z,+}^{n+\frac{1}{2}} - \eta_{Z,+}^{n-\frac{3}{2}}) + (T_{1,+}^{n}-T_{1,+}^{n-1},\eta_{Z,+}^{n+\frac{1}{2}} - \eta_{Z,+}^{n-\frac{3}{2}}).
		\end{equation}
		Applying Cauchy-Schwartz inequality, we obtain
		\begin{equation}
			\left\| \eta_{Z,+}^{n+\frac{1}{2}} - \eta_{Z,+}^{n-\frac{3}{2}} \right\| \le \left\|S_{2,+}^{n}-S_{2,+}^{n-1}\right\| + \left\|T_{1,+}^{n}-T_{1,+}^{n-1}\right\|.
		\end{equation}
		Similar to \eqref{lastT11-T10}, by applying \eqref{defT+} and Lemma \ref{lemma2}, we have the following estimate
		\begin{equation}
			\left\| {T_{1, + }^{n} - T_{1, + }^{n-1}} \right\| \le C\tau \left\| {{D_\tau }\eta _{\psi , + }^n} \right\| + C\tau \left\| {\eta _{\psi , + }^n} \right\| + C\tau {h^{k + 1}}. 
		\end{equation}
		Since $S_{2,+}^n$ is independent of the numerical solution, we can estimate $\left\| {S_{2, + }^{n} - S_{2, + }^{n-1}} \right\|$ similarly to \eqref{S21-S20}, obtaining
		\begin{equation}
			\left\| {S_{2, + }^{n} - S_{2, + }^{n-1}} \right\| \le \left\| {S_{1, + }^{n} - S_{1, + }^{n-1}} \right\| + \left\| {\xi _{Z, + }^{n+\frac{1}{2}} - \xi _{Z, + }^{ n- \frac{3}{2}}} \right\|  \le C\tau \left( {{\tau ^2} + {h^{k + 1}}} \right). \label{S21-S20n}
		\end{equation}
		By \eqref{Boud psi}, it follows $\|\bar \psi _{ + ,h}^{n - \frac{3}{2}}\|_\infty \le D_\psi+1$. Therefore, 
		\begin{equation}
			\begin{aligned}
				I_{1,4} &\le C \left\| e_{Z,+}^{n+\frac{1}{2}} - e_{Z,+}^{n-\frac{3}{2}}\right\| \le C \left\| \eta_{Z,+}^{n+\frac{1}{2}} - \eta_{Z,+}^{n-\frac{3}{2}}\right\| + C \left\| \xi_{Z,+}^{n+\frac{1}{2}} - \xi_{Z,+}^{n-\frac{3}{2}} \right\| \\ &\le C \tau (\left\| D_{\tau}\eta_{Z,+}^{n}\right\|+\left\| \eta_{Z,+}^{n}\right\|)+ C\tau (\tau^2 + h^{k+1}).
			\end{aligned} \label{K4}
		\end{equation}
		Combining \eqref{K1}, \eqref{K2}, \eqref{K3}, and \eqref{K4} yields
		\begin{equation}
			I_1 \le C\tau (\|\eta^{n}_{\psi,+}\| + \|\eta^{n-1}_{\psi,+}\| +\|\eta^{n-2}_{\psi,+}\| ) + C \tau \left(\left\| D_{\tau}\eta_{\psi,+}^{n+1}\right\| +\left\| D_{\tau}\eta_{\psi,+}^{n}\right\| +\left\| D_{\tau}\eta_{\psi,+}^{n-1}\right\|\right) + C\tau (\tau^2 + h^{k+1}). \label{I1 final}
		\end{equation}
		Similar to \eqref{K3}, we estimate $I_2$ as
		\begin{equation}
			I_2 \le C\left\| {\left( {\bar \psi _ + ^{n + \frac{1}{2}} - \bar \psi _{+,h }^{n + \frac{1}{2}}} \right) - \left( {\bar \psi _ + ^{n - \frac{3}{2}} - \bar \psi _{+,h }^{n - \frac{3}{2}}} \right)} \right\| \le C \tau \left(\left\| D_{\tau}\eta_{\psi,+}^{n+1}\right\| +\left\| D_{\tau}\eta_{\psi,+}^{n}\right\|+\left\| D_{\tau}\eta_{\psi,+}^{n-1}\right\|\right)+ C\tau h^{k+1}, \label{I2 final}
		\end{equation}
		by using the same estimation process as for $I_1$ from \eqref{I1 first} to \eqref{I1 final}, we can estimate $I_3$ as follows
		\begin{equation}
			I_3 \le C\tau (\|\eta^{n}_{\psi,+}\| + \|\eta^{n-1}_{\psi,+}\| +\|\eta^{n-2}_{\psi,+}\| ) + C \tau \left(\left\| D_{\tau}\eta_{\psi,+}^{n+1}\right\| +\left\| D_{\tau}\eta_{\psi,+}^{n}\right\| +\left\| D_{\tau}\eta_{\psi,+}^{n-1}\right\|\right) + C\tau (\tau^2 + h^{k+1}). \label{I3 final}
		\end{equation}
		Finally, combining \eqref{I1 final}, \eqref{I2 final}, and \eqref{I3 final}, we obtain the following estimate for the term $\|{\Gamma _{2, + }^{n + 1}}\| $:
		\begin{equation}
			\begin{aligned}
				\left\| {\Gamma _{2, + }^{n + 1}} \right\| &\le \left\| {J_{1, + }^{n + 1} - J_{1, + }^{n - 1}} \right\|\\
				&\le C\tau \left( {\left\| {\eta _{\psi , + }^n} \right\| + \left\| {\eta _{\psi , + }^{n - 1}} \right\| + \left\| {\eta _{\psi , + }^{n - 2}} \right\|} \right) \\ & \quad + C\tau \left( {\left\| {{D_\tau }\eta _{\psi , + }^{n + 1}} \right\| + \left\| {{D_\tau }\eta _{\psi , + }^n} \right\| + \left\| {{D_\tau }\eta _{\psi , + }^{n - 1}} \right\|} \right) + C\tau \left( {{\tau ^2} + {h^{k + 1}}} \right), \label{Gamma2+}
			\end{aligned}    
		\end{equation}
		and similarly, for $\| {\Gamma _{2, - }^{n + 1}} \|$,
		\begin{equation}
			\begin{aligned}
				\left\| {\Gamma _{2, - }^{n + 1}} \right\| 
				&\le C\tau \left( {\left\| {\eta _{\psi , - }^n} \right\| + \left\| {\eta _{\psi , - }^{n - 1}} \right\| + \left\| {\eta _{\psi , - }^{n - 2}} \right\|} \right) \\ & \quad + C\tau \left( {\left\| {{D_\tau }\eta _{\psi , - }^{n + 1}} \right\| + \left\| {{D_\tau }\eta _{\psi , - }^n} \right\| + \left\| {{D_\tau }\eta _{\psi , - }^{n - 1}} \right\|} \right) + C\tau \left( {{\tau ^2} + {h^{k + 1}}} \right).
			\end{aligned}    
		\end{equation}
		
		\textbf{Step 5.} In this step, we establish the estimates
		\begin{equation}
			\begin{array}{ll}
				&\left\| {e_{\psi , \pm }^{m + 1}} \right\| \le C\left( {{\tau ^2} + {h^{k + 1}}} \right),\\
				&\left\| {{D_\tau }\eta _{\psi , \pm }^{m + 1}} \right\| \le C\left( {{\tau ^2} + {h^{k + 1}}} \right). \label{step5conclusion}
			\end{array}
		\end{equation}
		By substituting \eqref{Gamma1+} and \eqref{Gamma2+} into \eqref{etam+1_and_etam}, we derive the following relationship
		\begin{equation}
			\left\| {{D_\tau }\eta _{\psi , + }^{m + 1}} \right\| + \left\| {{D_\tau }\eta _{\psi , + }^m} \right\| \le C\tau \sum\limits_{n = 1}^m {\left( {\left\| {\eta _{\psi , + }^n} \right\| + \left\| {{D_\tau }\eta _{\psi , + }^{n + 1}} \right\| + \left\| {{D_\tau }\eta _{\psi , + }^n} \right\|} \right)}  + C\left( {{\tau ^2} + {h^{k + 1}}} \right).
		\end{equation}
		We observe that the initial steps in \eqref{etam+1_and_etam} require separate treatment. Specifically, we have
		\begin{equation} \label{Sh D eta2+}
			\left\| {{S_h}{D_\tau }\eta _{\psi , + }^2} \right\| + \left\| {{S_h}{D_\tau }\eta _{\psi , + }^1} \right\| \le \left\| {{D_\tau }\eta _{\psi , + }^2} \right\| + \left\| {{D_\tau }\eta _{\psi , + }^1} \right\| \le \left\| {{D_\tau }\eta _{\psi , + }^2} \right\| + C\left( {{\tau ^2} + {h^{k + 1}}} \right).
		\end{equation}
		Then we take $v_{+,h} = \bar \eta _{\psi , + }^{n + \frac{1}{2}}$ in \eqref{deErrorEq}, and equating the coefficients of the imaginary part, it holds
		\begin{equation}
			\begin{aligned}
				\frac{1}{{2\tau }}\left( {{{\left\| {\eta _{\psi , + }^{n + 1}} \right\|}^2} - {{\left\| {\eta _{\psi , + }^n} \right\|}^2}} \right) &= {\mathop{\rm Im}\nolimits} \left( {J_{1, + }^{n + 1},\bar \eta _{\psi , + }^{n + \frac{1}{2}}} \right) + {\mathop{\rm Im}\nolimits} \left( {R_{2, + }^{n + 1},\bar \eta _{\psi , + }^{n + \frac{1}{2}}} \right)\\
				& \le \frac{1}{2}\left\| {J_{1, + }^{n + 1}} \right\|\left\| {\eta _{\psi , + }^{n + 1} + \eta _{\psi , +}^n} \right\| + \frac{1}{2}\left\| {R_{2, + }^{n + 1}} \right\|\left\| {\eta _{\psi , + }^{n + 1} + \eta _{\psi , + }^n} \right\|.
			\end{aligned}
		\end{equation}
		Using the same approach as in estimating \eqref{G1+} and leveraging the induction hypothesis, we obtain the estimate
		\begin{align}
			\begin{aligned}
				\left\| {J_{1, + }^{n+1}} \right\| & \le {\left\| {\bar \psi _ + ^{n+\frac{1}{2}}} \right\|_\infty }\left( {g\left\| {e_{Z, + }^{n+\frac{1}{2}}} \right\| + G\left\| {e_{Z, - }^{n+\frac{1}{2}}} \right\| + q\left\| {e_\phi ^{n+\frac{1}{2}}} \right\|} \right) \\ & \quad+ {\left\| {gZ_{+,h }^{n+\frac{1}{2}} + GZ_{-,h }^{n+\frac{1}{2}} + q\phi _h^{n+\frac{1}{2}}} \right\|_\infty }\left\| {e_{\psi , + }^{n+\frac{1}{2}}} \right\|\\
				&\le C\left( {{\tau ^2} + {h^{k + 1}}} \right) + C\left( {\left\| {\eta _{\psi , + }^{n+1}} \right\| + \left\| {\eta _{\psi , + }^n} \right\|} \right),
			\end{aligned}
		\end{align}
		from this, we deduce
		\begin{equation} \label{eta psi n+1 - eta psi n}
			\left\| {\eta _{\psi , + }^{n + 1}} \right\| - \left\| {\eta _{\psi , + }^n} \right\| \le \tau \left\| {J_{1, + }^{n + 1}} \right\| + \tau \left\| {R_{2, + }^{n + 1}} \right\| \le C\tau \left( {\left\| {\eta _{\psi , + }^{n + 1}} \right\| + \left\| {\eta _{\psi , + }^n} \right\|} \right) + C\tau \left( {{\tau ^2} + {h^{k + 1}}} \right).
		\end{equation}
		For the case $n=1$, we obtain
		\begin{equation}
			\left\| {\eta _{\psi , + }^{2}} \right\| - \left\| {\eta _{\psi , + }^1} \right\|  \le C_2\tau \left( {\left\| {\eta _{\psi , + }^{2}} \right\| + \left\| {\eta _{\psi , + }^1} \right\|} \right) + C\tau \left( {{\tau ^2} + {h^{k + 1}}} \right).
		\end{equation}
		Applying \eqref{e1 psi+} and assuming $\tau < \tau_3:=\min\{\tau_1,1/(C_2)\}$, we further deduce
		\begin{equation}
			\left\| {\eta _{\psi , + }^{2}} \right\| \le C\tau \left( {{\tau ^2} + {h^{k + 1}}} \right).
		\end{equation}
		and establish the bound
		\begin{equation} \label{Dtau eta2+}
			\left\| {D_\tau\eta _{\psi , + }^{2}} \right\| \le \frac{1}{\tau}\left(\left\| {\eta _{\psi , + }^{2}} \right\| +\left\| {\eta _{\psi , + }^{1}} \right\|\right) \le  C \left( {{\tau ^2} + {h^{k + 1}}} \right).
		\end{equation}
		Using \eqref{Dtau eta2+}, equation \eqref{Sh D eta2+} simplifies to
		\begin{equation} \label{Sh D eta2+ simplified}
			\left\| {{S_h}{D_\tau }\eta _{\psi , + }^2} \right\| + \left\| {{S_h}{D_\tau }\eta _{\psi , + }^1} \right\| \le C\left( {{\tau ^2} + {h^{k + 1}}} \right).
		\end{equation}
		Summing \eqref{eta psi n+1 - eta psi n} over $n$ from $1$ to $m$, we obtain
		\begin{equation}\label{etapsim+1}
			\left\| {\eta _{\psi , + }^{m + 1}} \right\| \le \left\| {\eta _{\psi , + }^1} \right\| + C\tau \sum\limits_{n = 1}^m {\left( {\left\| {\eta _{\psi , + }^{n + 1}} \right\| + \left\| {\eta _{\psi , + }^n} \right\|} \right)}  + C\left( {{\tau ^2} + {h^{k + 1}}} \right),
		\end{equation}
		then applying Gronwall's inequality in Lemma \ref{Gronwall’s inequality}, we conclude
		\begin{equation}\label{control etam+1+}
			\left\| {\eta _{\psi , + }^{m + 1}} \right\| + \left\| {{D_\tau }\eta _{\psi , + }^{m + 1}} \right\| + \left\| {{D_\tau }\eta _{\psi , + }^m} \right\| \le C\left( {{\tau ^2} + {h^{k + 1}}} \right).
		\end{equation}
		Similarly for ${\eta _{\psi , - }^{m + 1}}$, we obtain
		\begin{equation}\label{control etam+1-}
			\left\| {\eta _{\psi , - }^{m + 1}} \right\| + \left\| {{D_\tau }\eta _{\psi , - }^{m + 1}} \right\| + \left\| {{D_\tau }\eta _{\psi , - }^m} \right\| \le C\left( {{\tau ^2} + {h^{k + 1}}} \right),
		\end{equation}
		which verifies \eqref{step5conclusion}. 
		
		\textbf{Step 6.} 
		In this step, we prove that \eqref{intro Z} and \eqref{intro phi} also hold for $n=m+1$, namely,
		\begin{align}
			\left\| {e_{Z, \pm }^{m + \frac{3}{2}}} \right\| \le C\left( {{\tau ^2} + {h^{k + 1}}} \right), \\
			\left\| {e_{\phi}^{m + \frac{3}{2}}} \right\| \le C\left( {{\tau ^2} + {h^{k + 1}}} \right).
		\end{align}
		By \eqref{inverse inequality}, \eqref{boundedness of the true solution}, \eqref{control etam+1+} and \eqref{control etam+1-}, there exists a constant $h_6>0$, such that for $h<h_6$,
		\begin{equation}\label{m1bdd}
			\|\psi_{\pm,h}^{m+1}\|_{\infty} \le \|R_h\psi_\pm^{m+1}\|_{\infty} + \|\eta_{\psi,\pm}^{m+1}\|_{\infty} \le \|R_h\psi_\pm^{m+1}\|_{\infty} + Ch^{-\frac{d}{2}}\|\eta_{\psi,\pm}^{m+1}\| \le D_{\psi}+1.
		\end{equation}
		Following the approach in \eqref{eta3/2-eta-1/2}, we subtract the two resulting equations of \eqref{deErrorEq4} at $t_n$ and $t_{n-1}$, and set ${\chi _{h, + }} = \eta _{Z, + }^{n + \frac{1}{2}} + \eta _{Z, + }^{n - \frac{3}{2}}$. This yields
		\begin{equation}
			\left\|\eta_{Z,+}^{n+\frac12}\right\|^2-\left\|\eta_{Z,+}^{n-\frac32}\right\|^2=(S_{2,+}^{n}-S_{2,+}^{n-1},\eta _{Z, + }^{n + \frac{1}{2}} + \eta _{Z, + }^{n - \frac{3}{2}})+(T_{1,+}^{n}-T_{1,+}^{n-1},\eta _{Z, + }^{n + \frac{1}{2}} + \eta _{Z, + }^{n - \frac{3}{2}}).
		\end{equation}
		By the Cauchy--Schwarz inequality, 
		\begin{equation}\label{etaZn+1/2}
			\left\|\eta_{Z,+}^{n+\frac12}\right\|-\left\|\eta_{Z,+}^{n-\frac32}\right\|\le \left\|S_{2,+}^{n}-S_{2,+}^{n-1}\right\| +\left\|T_{1,+}^{n}-T_{1,+}^{n-1}\right\|.
		\end{equation}
		Arguing as in \eqref{xi Z 1}–\eqref{S21-S20}, we have
		\begin{equation}\label{S2n+}
			\left\|S_{2,+}^{n}-S_{2,+}^{n-1}\right\| \le C\tau(\tau^2+h^{k+1}).
		\end{equation}
		Similarly, using \eqref{lastT11-T10} together with \eqref{m1bdd}, we have
		\begin{equation}\label{T1n+}
			\left\|T_{1,+}^{n}-T_{1,+}^{n-1}\right\|
			\le
			C\tau\left(\left\|D_\tau\eta_{\psi,h,+}^{n}\right\|+\left\|\eta_{\psi,+}^{n}\right\|\right)
			+
			C\tau(\tau^2+h^{k+1}).
		\end{equation}
		Substituting \eqref{S2n+} and \eqref{T1n+} into \eqref{etaZn+1/2}, we arrive at
		\begin{equation}\label{etadiff}
			\left\| {\eta _{Z, + }^{n + \frac{1}{2}}} \right\| - \left\| {\eta _{Z, + }^{n - \frac{3}{2}}} \right\|  \le C\tau \left( {\left\| {{D_\tau }\eta _{\psi , + }^n} \right\| + \left\| {\eta _{\psi , + }^n} \right\|} \right) + C\tau \left( {{\tau ^2} + {h^{k + 1}}} \right).
		\end{equation}
		Summing \eqref{etadiff} over \( n = 1 \) to \( m+1 \), we obtain
		\begin{equation}
			\left\| {\eta _{Z, + }^{m + \frac{3}{2}}} \right\| \le C\tau \sum\limits_{n = 1}^{m + 1} {\left( {\left\| {{D_\tau }\eta _{\psi , + }^n} \right\| + \left\| {\eta _{\psi , + }^n} \right\|} \right)}  + C\left( {{\tau ^2} + {h^{k + 1}}} \right).
		\end{equation}
		Applying the induction hypothesis \eqref{intro psi}, \eqref{intro Dpsi}, and \eqref{intro Z} with \( n = m \), we simplify
		\begin{equation}
			\left\| {\eta _{Z, + }^{m + \frac{3}{2}}} \right\| \le C\left( {{\tau ^2} + {h^{k + 1}}} \right).
		\end{equation}
		Together with the projection error bound from \eqref{Ritz projection error}, we conclude
		\begin{equation}
			\left\| {e_{Z, + }^{m + \frac{3}{2}}} \right\| \le \left\| {\xi_{Z, + }^{m + \frac{3}{2}}} \right\| + \left\| {\eta_{Z, + }^{m + \frac{3}{2}}} \right\| \le C\left( {{\tau ^2} + {h^{k + 1}}} \right).
			\label{eZ+n}
		\end{equation}
		Following the same procedure for \( e_{Z, - }^{m + \frac{3}{2}} \), we similarly obtain
		\begin{equation}
			\left\| {e_{Z, - }^{m + \frac{3}{2}}} \right\| \le C\left( {{\tau ^2} + {h^{k + 1}}} \right).
			\label{eZ-n}
		\end{equation}
		Finally, combining \eqref{eZ+n} and \eqref{eZ-n}, then using Lemma \ref{propties of A1} and \eqref{ErrorEq3} implies
		\begin{equation}
			\left\| {e_\phi ^{m + \frac{3}{2}}} \right\| \le 4\pi \left| q \right|\left( {\left\| {e_{Z, + }^{m + \frac{3}{2}}} \right\| + \left\| {e_{Z, - }^{m + \frac{3}{2}}} \right\|} \right) + C{h^{k + 1}} \le C\left( {{\tau ^2} + {h^{k + 1}}} \right).
		\end{equation}
		Therefore, the estimates \eqref{intro psi}-\eqref{intro phi} hold for $n=m+1$ if $t_0:=\min{\{\tau_i\}}_{i=1}^4$ and $h_0:=\min\{h_j\}_{j=1}^6$. The constants $\tau_0$ and $h_0$ may depend on $T$, but are independent of $N$. 
		By the method of induction, the estimates \eqref{TH3.7 psi}-\eqref{TH3.7 phi} hold.
	\end{proof}

	\section*{Acknowledgments}
	Li's research was supported by the Postgraduate Scientific Research Innovation Project of Xiangtan University, China (XDCX2024Y179). 
	Yang's research was supported by the National Natural Science Foundation of China Project (No. 12571469), Scientific Research Innovation Capability Support Project for Young Faculty of China (No. SRICSPYF-BS2025132), the Project of Scientific Research Fund of the Hunan Provincial Science and Technology Department (No. 2024JJ1008), the 111 Project (No. D23017), and Program for Science and Technology Innovative Research Team in Higher Educational Institutions of Hunan Province of China.
	P. Yin’s research was supported by the University of Texas at El Paso Startup Award.
	
	\bibliographystyle{plain}
	\bibliography{refs}

\end{document}

%% file: ex_shared.tex
\usepackage{lipsum}
\usepackage{amsfonts}
\usepackage{graphicx}
\usepackage{epstopdf}
\ifpdf
\DeclareGraphicsExtensions{.eps,.pdf,.png,.jpg}
\else
\DeclareGraphicsExtensions{.eps}
\fi

\usepackage{booktabs}
\usepackage{stmaryrd}
\usepackage{bm}
\usepackage{caption}
\usepackage{subcaption}
\usepackage{multirow}
\usepackage{enumitem}

\usepackage{geometry}
\usepackage{float}
\usepackage[linesnumbered, ruled, vlined]{algorithm2e}

\usepackage{./macros}
\usepackage{amsopn}
\usepackage{hyperref}

\newsiamthm{assumption}{Assumption}

\newsiamremark{remark}{Remark}
\newsiamremark{hypothesis}{Hypothesis}
\crefname{hypothesis}{Hypothesis}{Hypotheses}
\newsiamthm{claim}{Claim}
\newsiamremark{exmp}{Example}

\headers{}{}

\title{A fully decoupled and structure-preserving relaxation Crank--Nicolson finite element method for Gross--Pitaevskii--Poisson model}

\author{Dongqian Li\thanks{Hunan Key Laboratory for Computation and Simulation in Science and Engineering, Hunan International Scientific and Technological Innovation Cooperation Base of Computational Science, Xiangtan University, Xiangtan 411105, Hunan, China  
\email{lidongqian@smail.xtu.edu.cn})}
\and Huini Liu\thanks{College of Mathematics and Statistics, Hengyang Normal University, Hengyang 421008, Hunan, China \
email{liuhuini@smail.xtu.edu.cn})}
\and Yin Yang\thanks{Hunan Research Center of the Basic Discipline Fundamental Algorithmic Theory and Novel Computational Methods, Key Laboratory of Intelligent Computing and Information Processing of Ministry of Education, Xiangtan University, Xiangtan 411105, Hunan, China   (\email{yangyin@xtu.edu.cn}).}
\and Peimeng Yin\thanks{Department of Mathematical Sciences, The University of Texas at El Paso, El Paso, TX 79968, USA (\email{pyin@utep.edu}).}} 
